\documentclass[a4paper,10pt]{article}

\usepackage{a4wide}
\usepackage{amsfonts,amssymb,amsmath,amsthm,bm}
\usepackage{mathrsfs, yfonts}
\usepackage{srcltx,tabularx}
\usepackage[utf8]{inputenc} 
\usepackage[alphabetic, nobysame]{amsrefs}
\usepackage{enumerate}
\usepackage[english]{babel}
\usepackage{authblk}
\usepackage[colorlinks,pagebackref=false]{hyperref}
\usepackage{cleveref}
\usepackage{graphicx}
\usepackage{pstricks,pst-plot,pst-node}
\usepackage[bottom]{footmisc}
\usepackage[all]{hypcap}
\usepackage[toc,page]{appendix}

\usepackage{mathtools}

\allowdisplaybreaks

\usepackage{breqn}

\usepackage{lmodern}%this is to remove the warning "Font shape `(...) in size <10> not available" some people complain on the tex.stackexchange that it typesets math accent very badly. I didn't notice that.

\usepackage{tikz}
\usepackage{tikz-cd}
\usepackage[hypcap=true]{caption}
\usetikzlibrary{matrix, positioning, calc}
\usetikzlibrary{arrows}

\newtheorem{theorem}{Theorem}[section]
\newtheorem*{theorem*}{Theorem}
\newtheorem{lemma}[theorem]{Lemma}

\newtheorem*{proposition*}{Proposition}
\newtheorem{proposition}[theorem]{Proposition}
\newtheorem{corollary}[theorem]{Corollary}
\newtheorem*{corollary*}{Corollary}
\newtheorem*{conjecture*}{Conjecture}
\newtheorem*{question*}{Question}
\newtheorem{claim}{Claim}

\newtheorem{maintheorem}{Theorem}

  \newenvironment{claimproof}[1]{\par\noindent\textit{Proof of the claim:}\space#1}{\hfill $\blacksquare$}

\newcommand{\boundellipse}[3]% center, xdim, ydim
{(#1) ellipse (#2 and #3)
}

\theoremstyle{definition}
\newtheorem{definition}[theorem]{Definition}
\newtheorem*{definition*}{Definition}

\newtheorem{remark}[theorem]{Remark}

\newcommand{\suchthat}{\;\ifnum\currentgrouptype=16 \middle\fi|\;}

\DeclareMathOperator{\Aut}{\mathrm{Aut}}
\DeclareMathOperator{\N}{\mathbf{N}}
\DeclareMathOperator{\Z}{\mathbf{Z}}

\DeclareMathOperator{\PGL}{\mathrm{PGL}}

\DeclareMathOperator{\SL}{\mathrm{SL}}

\DeclareMathOperator{\GL}{\mathrm{GL}}

\DeclareMathOperator{\SO}{\mathrm{SO}}
\DeclareMathOperator{\PGU}{\mathrm{PGU}}
\DeclareMathOperator{\Spec}{\mathrm{Spec}}

\DeclareMathOperator{\id}{\mathrm{id}}

\DeclareMathOperator{\Out}{\mathrm{Out}}

\DeclareMathOperator{\Gal}{\mathrm{Gal}}
\DeclareMathOperator{\Hom}{\mathrm{Hom}}

\DeclareMathOperator{\Lie}{\mathrm{Lie}}
\DeclareMathOperator{\Spin}{\mathrm{Spin}}

\DeclareMathOperator{\Id}{\mathrm{Id}}
\DeclareMathOperator{\Alg}{\mathrm{alg}}

\DeclareMathOperator{\Mor}{\mathrm{Mor}}
\DeclareMathOperator{\Ad}{\mathrm{Ad}}

\DeclareMathOperator{\Nrd}{\mathrm{Nrd}}
\DeclareMathOperator{\diag}{\mathrm{Diag}}
\DeclareMathOperator{\ev}{\mathrm{ev}}
\DeclareMathOperator{\intaut}{\mathrm{int}}
\DeclareMathOperator{\Isom}{\mathrm{Isom}}
\DeclareMathOperator{\Extisom}{\mathrm{Extisom}}
\DeclareMathOperator{\inv}{\mathrm{inv}}
\DeclareMathOperator{\BRD}{\mathrm{BRD}}
\DeclareMathOperator{\RedExt}{\mathrm{RedExt}}
\DeclareMathOperator{\QsPin}{\mathrm{QsPin}}

\begin{document}
\title{Semilinear automorphisms of reductive algebraic groups}
%\ead{pe.caprace@uclouvain.be}
\author[1]{Thierry Stulemeijer\thanks{Postdoctoral fellow at Justus-Liebig Universität Giessen}}
%\ead{Thierry.Stulemeijer@uclouvain.be}

\affil[1]{\small{Justus-Liebig Universität Giessen, 35392 Giessen, Germany}}

\date{March 01, 2019}

\maketitle

\begin{abstract}
Let $ G $ be a connected reductive algebraic group over a field $ k $. We study the group of semilinear automorphisms $ \Aut (G\to \Spec k) $ consisting of algebraic automorphisms of $ G $ over automorphisms of $ k $. We focus on the exact sequence $ 1\to \Aut G\to \Aut (G\to \Spec k)\to \Aut_{G}(k)\to 1 $. When $G$ is quasi-split, we show that $\Aut_{G}(k)$ is isomorphic to $\Aut_{\mathcal{R}(G)}(k)$, where $\mathcal{R}(G)$ denotes the scheme of based root datum of $G$. Furthermore, the exact sequence $ 1\to \Aut G\to \Aut (G\to \Spec k)\to \Aut_{G}(k)\to 1 $ splits if and only if the exact sequence $ 1\to \Aut \mathcal{R}(G) \to \Aut (\mathcal{R}(G) \to \Spec k)\to \Aut_{\mathcal{R}(G)}(k)\to 1 $ splits. As a corollary, we get many examples of algebraic groups $ G $ over $ k $ whose group of abstract automorphisms does not decompose as the semidirect product of $ \Aut G $ with $ \Aut_G(k) $. We also study the same questions for inner forms of $ \SL_n $ over a local field.
\end{abstract}

\tableofcontents

\section{Introduction}\label{Sec:introchapsemilinear}
The so-called abstract automorphisms of (the rational points of) a reductive algebraic group have been studied extensively in the literature. In 1955, J.\ Dieudonn\'e wrote a comprehensive treatise (see \cite{Di71} for the third edition) covering the case of classical groups, and even going beyond the world of algebraic groups (since he also considers classical groups over division algebras that are infinite dimensional over their center). Many results in this area have been subsumed in the famous article \cite{BT73} by A.\ Borel and J.\ Tits. To wit, here is one of their results:

\begin{theorem*}[{\cite{BT73}*{excerpt of Corollaire~8.13}}]\label{Thm:abstract homo.}
Let $ k $ be an infinite field, and let $ G $ be an absolutely simple algebraic group over $ k $, which is assumed to be isotropic and adjoint. Furthermore, if $ k $ is of characteristic $ 2 $ or $ 3 $, assume that $ k $ is not perfect. Let $ \alpha $ be an automorphism of $ G(k) $, considered as an abstract group. Then there exists a unique automorphism $ \varphi \colon k\to k $ and a unique semilinear automorphism $ f_{\varphi} \colon G\to G $ over $ \varphi $ such that for $ g\in G(k) = \Hom_{\text{k-schemes}}(\Spec k,G) $, we have $ \alpha (g) = f_{\varphi}\circ g\circ (\Spec \varphi)^{-1}  $.
\end{theorem*}

By a \textbf{semilinear automorphism} $ f_{\varphi} $ of a $ k $-group scheme $ G $ over an automorphism $ \varphi \colon k\to k $, we mean that we have the following commutative diagram in the category of group schemes
 \begin{center}
 \begin{tikzpicture}[->]
  \node (1) [anchor=east] {$ G $};
  \node (2) [right=1.5cm of 1] {$ G $};
  \node (3) [below=0.5cm of 1] {$ \Spec k $};
  \node (4) [below=0.5cm of 2] {$ \Spec k $};
  
  \path [every node/.style={font=\sffamily\small}]
  (1) edge node [above]  {$ f_{\varphi} $} (2)
        edge node  {} (3)
  (2) edge node  {} (4)
  (3) edge node [above] {\footnotesize{$ \Spec \varphi $}} (4);
  \end{tikzpicture}
  \end{center}
where the vertical arrows are the structural morphisms of the $ k $-scheme $ G $. Note that if $G$ is realised as a matrix group, and if $g=(g_{ij})\in G(k)$, then $ f_{\varphi}\circ g\circ (\Spec \varphi)^{-1} $ is given by the matrix whose $ij$-th coefficient is $\varphi^{-1}((f_{\varphi})_{ij}(g))$, so that our convention differs from \cite{BT73} (see Remark~\ref{Rem:difference between B-T and me} for a discussion of this difference).

Another way to phrase the Borel--Tits theorem is to say that the group $ \Aut_{\text{abstract}}(G(k)) $ of abstract automorphisms of $ G(k) $ fits in the exact sequence $ 1\to \Aut G\to \Aut_{\text{abstract}}(G(k))\to \Aut (k) $. Letting $ \Aut_G(k) $ denote the image of $ \Aut_{\text{abstract}}(G(k)) $ in $ \Aut(k) $, it is natural to wonder what $ \Aut_G(k) $ is, and whether the group of abstract automorphisms of $ G(k) $ splits as the semi-direct product $ (\Aut G)\rtimes \Aut_G(k) $. While this semi-direct product decomposition does hold when G is $ k $-split, it turns out that it may fail in general, even if G is quasi-split. One of the aims of the present paper is to address this issue. 

To illustrate some of our results in the most concrete way, here is a corollary of our results (we refer the reader to Corollary~\ref{Cor:very explicit non-splitting} for a proof of this statement and to Remark~\ref{Rem:an explicit description of quasi-split PGU} for an explicit realisation of the algebraic group appearing in this corollary) :
\begin{corollary*}[of Theorem~\ref{Thm: MainThm2}]
Let $ G $ be the quasi-split, absolutely simple, adjoint algebraic group of type $ ^2A_{n-1} $ over a field $ k $ with corresponding quadratic separable extension $ l $. Let $ k_0 $ be the field of rational numbers $ ( $respectively the field of $ p $-adic numbers for some prime $ p) $ and assume that $ k $ and $ l $ are possibly infinite $ ( $respectively finite$ ) $ Galois extensions of $ k_0 $. Then $ \Aut_G(k) = \Gal (k/k_0) $ and the short exact sequence $ 1\to \Aut G\to \Aut_{\text{abstract}}(G(k))\to \Aut_G(k)\to 1 $ splits if and only if the short exact sequence of abstract groups $ 1\to \Gal (l/k)\to \Gal (l/k_0)\to \Gal (k/k_0)\to 1 $ splits.
\end{corollary*}

We now recast the problem in a more useful way for us. Let $ \Aut (G\to \Spec k) $ denotes the group of semilinear automorphisms of $ G $. We can then rephrase the Borel--Tits theorem as saying that under the assumptions of the theorem, the natural homomorphism $ \Aut (G\to \Spec k)\to \Aut_{\text{abstract}} (G(k)) $ is an isomorphism. The rest of the paper is concerned with the study of $ \Aut (G\to \Spec k) $.

Given a $ k $-group scheme $ G $ (in this paper the letter $ k $ is exclusively used to designate a field), we have a homomorphism $ \Aut (G\to \Spec k)\to \Aut (k)\colon f_{\varphi}\mapsto \varphi^{-1} $. Let $ \Aut_{G}(k) $ denotes the image of this homomorphism, and let $ \Aut G $ denotes the group of $ k $-algebraic automorphisms of $ G $, or in other words the kernel of $ \Aut (G\to \Spec k)\to \Aut (k) $. In summary, we defined the exact sequence $ 1\to \Aut G\to \Aut (G\to \Spec k)\to \Aut_G (k)\to 1 $.

It seems that not much is known about the subgroup $\Aut_G(k)\leq \Aut (k)$ in general (in \cite{Di71}*{p. 18, last paragraph}, J.\ Dieudonn\'e makes a related comment, though not exactly about $ \Aut_G(k) $). Actually, for $G = \textbf{SL}_n(D)$ with $D$ a non-commutative finite dimensional central division algebra, $\Aut_G(k)$ is isomorphic to the group of outer automorphisms of $D$ (or to a degree $ 2 $ extension of this group). But computing the latter is probably hard (for example, in \cite{Ha05} T.\ Hanke computes explicitly an outer automorphism of a division algebra of degree $3$ over a number field, and then uses it to construct a non-crossed product division algebra over $\mathbf{Q}(\!(t)\!)$).

Nonetheless, in some special cases, $\Aut_G(k)$ is easily understood. For example, when $G$ is a split connected reductive group, then $\Aut_G(k) = \Aut (k)$ (this follows directly from the fact that split connected reductive groups are defined over the prime field of $k$). The next easiest case should be to compute $ \Aut_G(k) $ for quasi-split groups, and this is indeed feasible.
\begin{theorem}\label{Thm: MainThm1}
Let $G$ be a connected reductive $k$-group and let $\mathcal{R}(G)$ be its $k$-scheme of based root datum. There exists a homomorphism $\Aut (G\to \Spec k)\to \Aut (\mathcal{R}(G)\to \Spec k) $ preserving the underlying automorphism of $k$ and whose kernel is $ (\Ad G)(k) $. Hence $ \Aut_G(k)\leq \Aut_{\mathcal{R}(G)}(k) $. If $G$ is quasi-split, the corresponding exact sequence $1\to (\Ad G)(k)\to \Aut (G\to \Spec k)\to \Aut (\mathcal{R}(G)\to \Spec k)$ splits. Hence if $ G $ is quasi-split $\Aut_G(k) = \Aut_{\mathcal{R}(G)}(k)$.
\end{theorem}

The $k$-scheme of based root datum $\mathcal{R}(G)$ appearing in this result is an object (see Definition~\ref{Def:scheme of based root datum}) which can be described as a variation on the scheme of Dynkin diagram as defined in \cite{Gil-Pol11}*{Exposé 24, Section 3}. In the end, the construction of the $k$-scheme of based root datum is a restatement of \cite{Gil-Pol11}*{Exposé 24, Théorème 3.11}. Let us directly give a brief description of what $\mathcal{R}(G)$ is in concrete terms: for $G_0$ a split connected reductive $k$-group, $\mathcal{R}(G_0)$ is defined to be the constant object on $\Spec k$ associated to the based root datum $R = (M,M^{\ast},R,R^{\ast},\Delta )$ of $G_0$. Now given a $k_s/k$-form $G$ of $G_0$, one can choose a cocycle $c\colon \Gal (k_s/k)\to \Aut((G_0)_{k_s})$ defining $G$ and compose it with the projection $\Aut((G_0)_{k_s})\to \Aut R$ to obtain a twist of $\mathcal{R}(G_0)$ which we define to be the $k$-scheme of based root datum of $G$. In light of this description, it is clear that $\mathcal{R}(G) \cong \mathcal{R}(G')$ (in the category of based root datum over $k$) if and only if $G$ and $G'$ are inner forms of each other.

The proof of Theorem~\ref{Thm: MainThm1} consists mainly of a straightforward adaptation of \cite{Gil-Pol11}*{Exposé 24, Théorème 3.11} to the semilinear situation. One crucial step is to show that taking the scheme of based root datum of $ G $ commutes with base change (see Lemma~\ref{Lem:base change commute with taking root datum}). We chose to prove this in concrete terms, by using cocycle computations adapted to the semilinear situation. Those cocycle computations also lead to the following Galois cohomological formulation.
\begin{theorem}\label{Thm: MainThm3}
Let $ G $ be a connected reductive $k$-group and let $ \mathcal{R}(G) $ be its $ k $-scheme of based root datum. We have two exact sequences
\begin{align*}
1\to (\Ad G)&(k)\to \Aut (G\to \Spec k)\to \Aut (\mathcal{R}(G)\to \Spec k)\to H^1(k,(\Ad G)(k_s))\\
&1\to \Aut G\to \Aut (G\to \Spec k)\to \Aut (k)\to H^1(k,\Aut G_{k_s}).
\end{align*}
\end{theorem}

We refer the reader to Section~\ref{Sec:Semilinear Galois cohomology} for the definition of the coboundary maps involved in those exact sequences. After proving Theorem~\ref{Thm: MainThm3}, we illustrate how it can be used by computing $ \Aut_G(k) $ when $ G\cong \textbf{SL}_1(D) $ and $ D $ is a division algebra of degree $ 3 $ over $ k $. In doing so, we recover some of the results in \cite{Ha07} without needing to introduce bycyclic algebras (see Lemma~\ref{Lem:Hanke result part 1}, Lemma~\ref{Lem:Hanke result part 2} and Remark~\ref{Rem:applicability of these computations}).

Once we have some control over $\Aut_G(k)$, it is natural to wonder whether the exact sequence $ 1\to \Aut G \to \Aut (G\to \nolinebreak \Spec k)\to \Aut_G (k)\to 1 $ splits. Again, for $G$ a split connected reductive algebraic group, this exact sequence does split (a statement already made in \cite{Tits74}*{Corollary~5.10}) because such a group is defined over its prime field. But somewhat surprisingly, this is not any more the case for a general quasi-split group.

\begin{theorem}[The bowtie theorem]\label{Thm: MainThm2}
Let $ G $ be a connected reductive $k$-group which is quasi-split, and let $\mathcal{R}(G)$ be its $k$-scheme of based root datum. Then the short exact sequence $$ 1\to \Aut G \to \Aut (G\to \Spec k)\to \Aut_G (k)\to 1 $$ splits if and only if the short exact sequence $$ 1\to \Aut \mathcal{R}(G)\to \Aut (\mathcal{R}(G)\to \Spec k)\to \Aut_{\mathcal{R}(G)} (k)\to 1 $$ splits.
\end{theorem}

As it turns out, the bowtie theorem (whose name is due to the diagram appearing in its proof) is a direct corollary of Theorem~\ref{Thm: MainThm1} (see the end of Section~\ref{Sec:RD's scheme} for the proof).

When $G$ is absolutely simple, we can identify the short exact sequence $ 1\to \Aut \mathcal{R}(G)\to \Aut (\mathcal{R}(G)\to \Spec k)\to \Aut_{\mathcal{R}(G)} (k)\to 1 $ as a sequence involving various automorphism groups of fields naturally associated with $ G $ (see Proposition~\ref{Prop:epplicit SES for Dyn}). Using this description, we can then give many explicit examples of absolutely simple, quasi-split algebraic $ k $-groups $G$ for which $ \Aut (G\to \Spec k) $ is not a split extension of $ \Aut_G(k) $. The corollary given at the beginning of this introduction illustrates this in the most concrete way.

In the last section of the paper, we also explore the same questions when $ G $ is an inner form of $ \SL_n $ over a local field $ K $ (see the beginning of Section~\ref{Sec:SL_n(D)} for a precise definition of what we mean by a local field). The first step is to get some control over $ \Aut_G(K) $. Actually, a direct corollary of the results of T.\ Hanke in \cite{Ha07} is that in this case, $ \Aut_G(K) = \Aut (K) $.

\begin{theorem}[corollary of \cite{Ha07}]\label{Thm:Main Thm 2.1}
Let $ K $ be a local field, let $ D $ be a central division algebra over $K$ of degree $ d $, and consider the algebraic group $ G = \textbf{SL}_n(D) $. Then $ \Aut_{G}(K) = \Aut (K) $.
\end{theorem}

In \cite{Ha07}, T.\ Hanke does not mention local fields, but he gives an algorithm over an arbitrary field to compute outer automorphisms of cyclic division algebras, and this implies the result for local fields. In fact, we only need to use the simplest version of his algorithm and for the ease of the reader we give a self contained proof of Theorem~\ref{Thm:Main Thm 2.1} in Corollary~\ref{Cor:etxending auto to D}. Also note that in characteristic $0$, this result has probably been known for a long time since it is a direct consequence of \cite{EML48}*{Corollary~7.3}.

Finally, we obtain an explicit characterisation for the splitting of the exact sequence $ 1\to \Aut G\to \Aut (G\to \Spec K)\to \Aut _G(K)\to 1 $ for $ G $ an inner form of $ \SL_n $ over a local field.

\begin{theorem}\label{Thm:Main Thm 2.2}
Let $ K $ be a local field, let $ D $ be a central division algebra over $K$ of degree $ d $, and consider the algebraic $K$-group $ G = \textbf{SL}_n(D) $. The short exact sequence $ 1\to \Aut G\to \Aut (G\to \Spec K)\to \Aut _G(K)\to 1 $ splits if and only if $ \gcd (nd,[K:K']) $ divides $ n $ for all subfields $ K'\leq K $ such that $ K/K' $ is finite Galois.
\end{theorem}

As we prove in Proposition~\ref{Prop:existence of Galois subfield of some degree}, for $ K = \mathbf{F}_{p^i}(\!(T)\!) $, the condition ``$ \gcd (nd,[K:K']) $ divides $ n $ for all subfields $ K'\leq K $ such that $ K/K' $ is finite Galois'' is equivalent to requiring that $ \gcd (d,p) = 1 $ and that $ \gcd (nd,i(p^i-1)) $ divides $ n $, so that this criterion is very explicit in characteristic $p$. On the other hand, in characteristic $0$ note that $\mathbf{Q}_p$ is rigid (see Definition~\ref{Def:strongly rigid field} and Lemma~\ref{Lem:Qp is rigid}), so that the condition is a finite one. See also Remark~\ref{Rem:Vivid example of splitting for divison algebras} for a vivid illustration of Theorem~\ref{Thm:Main Thm 2.2} in characteristic $ 0 $.

The necessity of the condition is proved in Corollary~\ref{Cor:non-splitting condition for SLn(D)}. The hard part of Theorem~\ref{Thm:Main Thm 2.2} is to show that our explicit criterion is sufficient. Whilst in characteristic $0$ this follows from Galois descent (see Theorem~\ref{Thm:splitting in char. 0 for SL_n(D)}), no descent technique can be used in characteristic $p>0$ since $\textbf{SL}_n(D)$ is not defined over the fixed field of $\Aut (K)$ (which is just $\mathbf{F}_p$) when $D$ is non commutative. Hence we have to work by hand and give the splitting explicitly. In order to achieve this, for $ K = \mathbf{F}_{p^i}(\!(T)\!) $, we decompose $ \Aut (K) $ as $ (J(K)\rtimes \mathbf{F}_{p^i}^{\times})\rtimes \Gal (K/\mathbf{F}_{p}(\!(T)\!)) $ (see Definition~\ref{Def:J(K) and C(K)} for the definition of $ J(K) $). Fortunately, it is easy enough to find an explicit section of $ \Aut (G\to \Spec K)\to \Aut _G(K) $ for the $ J(K) $ components, and the theory of Galois descent predicts when a section to $ \Aut (G\to \Spec K)\to \Aut _G(K) $ exists for the component $ \mathbf{F}_{p^i}^{\times} \rtimes \Gal (K/\mathbf{F}_{p}(\!(T)\!)) $. It thus suffices to compute explicitly those sections predicted by Galois descent, and to check that the formulas we found on each component can be glued together.

\section*{Acknowledgements}
This paper grew out from the idea that the scheme of Dynkin diagrams provides an obstruction for a semisimple algebraic group to be defined over a given field. This idea was given in a comment to my MathOverflow question \cite{St15} by user grghxy (probably a close relative to nfdc23), and I gratefully thank him.

Jean-Pierre Tignol suggested at an early stage of this project that it should be possible to give an explicit example in some specific cases, which gave an early form of Corollary~\ref{Cor:very explicit non-splitting}. He also pointed out many innacuracies and made various comments that were very helpful. In particular, he mentioned to us the work of T.\ Hanke and the fact that outer automorphisms of division algebras are related to $H^3 (k, \mathbf{G}_m)$, as explained in \cite{EML48}. He also provided a much cleaner version and proof of Lemma~\ref{Lem:Automatic preservation of the reducednorm}. I thank him for the very interesting discussions we had on this subject.

A preliminary version of this paper was focusing entirely on the case of local fields. I thank Richard Weiss for pointing out that one could consider just as easily the general case.

I am also grateful to Philippe Gille, who kindly indicated that an earlier version of this paper was reproving a special case of \cite{Gil-Pol11}*{Expos\'e~24, Th\'eor\`eme~3.11}. This pointer to the literature was very useful, since in the end this result is the central one around which our paper is organised.

Finally, I warmly thank Pierre-Emmanuel Caprace for asking the question that got this paper started, for encouraging me to investigate the $ \SL_n(D) $ case and for his patient teaching on mathematical exposition.

The research concerning this project was done in part whilst the author was a F.R.S.-FNRS research fellow (in between 2013 and 2017), and then a postdoctoral fellow at the Max Planck Institute for Mathematics in Bonn (in between 2017 and 2018). The final version of this paper was then written whilst the author was a postdoctoral fellow at the Justus-Liebig Universität Giessen. We thank all those institutions for their support.

\section{Semilinear automorphisms and Galois descent}\label{Sec:Semiauto and Galois descent}
For the rest of the paper, the letter $ k $ stands for an arbitrary field. By a $ k $-group scheme, we mean an affine group scheme of finite type over $ k $. A smooth $ k $-group scheme is called an algebraic group. Given an object $ X $ in a category, we write $ \Aut X $ for the automorphisms of $ X $ in that category. Also given a $ k $-scheme $ X $, we denote by \underline{$  \Aut $} $ X $ its $ k $-group functor of automorphisms (i.e.\ for any $ k $-algebra $ R $, (\underline{$  \Aut $} $ X)(R) = \Aut X_R $). With these conventions, for $ G $ a $k$-group scheme, $ \Aut G $ is the automorphism functor of $ G $ evaluated at $ k $, i.e.\ $ \Aut G = $ (\underline{$ \Aut $} $G)(k) $.

Let $ G $ be a $ k $-group scheme.  We gave in the introduction the definition of a semilinear automorphism of $ G $. The vocabulary ``semilinear automorphism'' is already used in the literature (see for example \cite{FSS98}*{Section~1.2}). It has the same meaning than our usage, except that in those references, the underlying automorphisms of the base field are assumed to fix a subfield $ k_0 $ such that $ k/k_0 $ is Galois. We do not make this assumption, and for example in Section~\ref{Sec:SL_n(D)}, we consider the case of arbitrary automorphisms of $ k = \mathbf{F}_p(\!(T)\!) $, which is a more general situation.

In the literature, the notation $ \text{SAut}(G_{k_s}) $ is used for the group of semilinear automorphisms (see for examples \cite{BKLR12}*{Section~3.2} and also the references therein). We prefer to use the notation $ \Aut (G\to \Spec k) $ so that the ground field explicitly appears in the notation.
\begin{remark}
It is tempting to define a ``semilinear automorphism sheaf of $ k $-algebras'' such that $ \Aut (G\to \Spec k) $ would be its $k$-rational points. Unfortunately, this is not possible, because we do not know how to extend functorially automorphisms of $ k $ to automorphisms of an arbitrary $ k $-algebra $ R $.
\end{remark}

Let us continue by recalling some standard vocabulary.
\begin{definition}\label{Def:Notation Aut(l/k) and base change}
Let $ \varphi \colon  k\to l $ be a field homomorphism (if $ l $ is a field containing $ k $, we take $ \varphi $ to be the identity), let $ G, G' $ be $ k $-group schemes and let $ H $ be an $l$-group scheme.
\begin{enumerate}
\item The group of automorphisms of $ l $ whose restriction to $ \varphi (k) $ is trivial is denoted $ \Aut (l/k) $.
\item We set $ \varphi^{\ast} = \Spec \varphi $. We denote the base change of $ G $ along $ \Spec l\xrightarrow{\varphi^{\ast}} \Spec k $ either by $ G_{l} $ or by $ ^{\varphi}G $. If $ G_{l} $ is isomorphic to $ H $ (as an $l$-group scheme), we say that $ G $ is an $ l/k$-\textbf{form of $ H $} (or just a form of $H$ if the field extension is understood from the context). If there exists an $ l/k $-form of $ H $, we say that $ H $ is \textbf{defined over $ k $}.

\item For $ f\colon G\to G' $ a homomorphism of $ k $-group schemes, we denote by $ ^{\varphi}f\colon\,^{\varphi}G\to\,^{\varphi}G' $ the base change of $ f $ along $ \varphi^{\ast} $.%N.B.: It's good here to put the star below because $ \alpha_{\ast}\beta_{\ast}G = (\alpha \beta)_{\ast}G $.
\end{enumerate}
\end{definition}

\begin{remark}\label{Rem:difference between B-T and me}
Having set up our notations, let us elucidate the difference between our conventions and the conventions in \cite{BT73}. Given a $ k $-group scheme $ G $, a $ k' $-group scheme $ G' $ and given an abstract homomorphism $ \alpha \colon G(k)\to G'(k') $, A.\ Borel and J.\ Tits aim to obtain a field homomorphism $ \varphi \colon k\to k' $ and an isogeny $ \beta \colon \,^{\varphi}G\to G' $ such that for $ g\in G(k) = \Hom_{\text{k-schemes}}(\Spec k,G) $, $ \alpha (g) = \beta \circ \,^{\varphi}g $. The following commutative diagram summarises the situation:
\begin{center}
\begin{tikzcd}[row sep=2.5em, column sep=3em] 
G' \arrow{dr}  &^{\varphi}G \arrow[l,swap,"\beta"]\arrow{r}\arrow[d,shift right=1.5ex] & G \arrow[d,shift right=1.5ex]\\ 
& \Spec k' \arrow[r,"\Spec \varphi "] \arrow[u,swap ,"^{\varphi}g"] & \Spec k \arrow[u,swap ,"g"]
\end{tikzcd}
\end{center}
On the other hand, the present paper focuses entirely on the group of semilinear automorphisms of $ G $. To keep the Borel--Tits convention, one should define this group as $ \lbrace \Isom_{k\text{-grp schemes}}(~^{\varphi}G, G)~\vert~ \linebreak \varphi \in \Aut (k)\rbrace $. We prefer to use the more natural definition that a semilinear automorphism over $ \varphi \in \Aut (k) $ is a commutative diagram of the following kind (note that either one of the red arrows determines the other):
\[
\begin{tikzcd}[row sep=2.5em]
G \arrow[drr,red,"f_{\varphi}"] \arrow[dr,red] \arrow[ddr] & & \\
&^{\varphi}G \arrow[r] \arrow[d] & G \arrow[d]\\
&\Spec k \arrow[r,"\Spec \varphi"] & \Spec k
\end{tikzcd}
\]
where $ f_{\varphi} $ and $ \Spec \varphi $ are both automorphisms of group schemes (but they are not automorphisms of $ k $-group schemes when $ \varphi $ is not the identity). In this setting, there are two ways (admittedly not as natural as in the Borel--Tits setting) to obtain an abstract automorphism of $ G(k) $. Either we define this abstract automorphism proceeding ``from right to left'', in which case we would obtain the map $ G(k)\to G(k)\colon g\mapsto f_{\varphi}^{-1}\circ g\circ \Spec \varphi $. Or we proceed ``from left to right'', in which case we obtain the map $ G(k)\to G(k)\colon g\mapsto f_{\varphi}\circ g\circ (\Spec \varphi)^{-1} $. We chose the latter option.
\end{remark}

The following elementary observation plays a fundamental role in this work.
\begin{lemma}\label{Lem:field of definition are in AutG(k)}
Let $ k\leq  l $ be a field extension of $ k $, let $ G $ be an $ l $-group scheme and assume that $ G $ is defined over $k$. Then there exists a homomorphism $ \Aut (l/k)\to \Aut (G \to \Spec l) $ whose composition with $ \Aut (G\to \Spec l)\to \Aut_G(l) $ is the identity on $ \Aut (l/k) $. In particular, $ \Aut_G(l) $ contains $ \Aut (l/k) $.
\end{lemma}
\begin{proof}
Let $ H $ be an $ l/k $-form of $ G $. For $ \varphi \in \Aut (l/k) $, we define $$ f_{\varphi}=\Id_H \times \varphi^*\colon H\times_{\Spec k} \Spec l\to H\times_{\Spec k} \Spec l.  $$ The map $ \Aut (l/k)\to \Aut (G\to \Spec l)\colon \varphi \mapsto f_{\varphi^{-1}} $ is a homomorphism. Furthermore, its composition with $ \Aut (G\to \Spec l)\to \Aut_G(l) $ is the identity on $ \Aut (l/k) $, as wanted.
\end{proof}

In fact, if the field extension $ l/k $ appearing in Lemma~\ref{Lem:field of definition are in AutG(k)} is finite Galois, then we have a converse to Lemma~\ref{Lem:field of definition are in AutG(k)} by the theory of Galois descent.

\begin{theorem}[Galois descent]\label{Thm:Galois descent}
Let $ k\leq  l $ be a field extension of $ k $ such that $ l/k $ is a finite Galois extension and let $ G $ be an $ l $-group scheme. If there exists a homomorphism $ \Gal (l/k)\to \Aut (G\to \Spec l) $ whose composition with $ \Aut (G\to \nolinebreak \Spec l)\to \Aut_G(l) $ is the identity on $ \Gal (l/k) $, then $ G $ is defined over $ k $.
\end{theorem}
\begin{proof}
This is a classical result from descent theory, see \cite{Poo17}*{Section~4.4}. Note that giving such a homomorphism is the same as giving a descent datum on $ G $ by \cite{Poo17}*{Proposition~4.4.2}, so that the result holds by \cite{Poo17}*{Corollary~4.4.6}.
\end{proof}
\begin{remark}
One could also treat the case of infinite Galois extensions by adding a continuity assumption as in \cite{FSS98}*{Remark~1.15}, but we do not need it in our work. See also \cite{Poo17}*{Remark~4.4.8} for how to deal with infinite Galois extensions.
\end{remark}

In view of the strong link between Galois descent and semilinear automorphisms, it seems natural that there should be a cocycle interpretation of semilinear automorphisms. We now take some time to set up this formalism in detail.
\begin{definition}\label{Def:gammatilde and Idgamma}
\begin{enumerate}
\item Let $ k\leq l $ be a field extension of $ k $, let $ G_0 $ be an $l$-group scheme and let $G$ be an $l/k$-form of $G_0$. Choose an isomorphism $  G_0\cong G_l $, or in other words choose an exact diagram
\begin{center}
 \begin{tikzpicture}[->]
  \node (1) [anchor=east] {$ G_0 $};
  \node (2) [right=1.5cm of 1] {$ G $};
  \node (3) [below=0.5cm of 1] {$ \Spec l $};
  \node (4) [below=0.52cm of 2] {$ \Spec k $};
  
  \path [every node/.style={font=\sffamily\small}]
  (1) edge node [above]  {$ \pi_1 $} (2)
        edge node [left] {$ t $} (3)
  (2) edge node  [right] {$ s $} (4)
  (3) edge node [above] {\footnotesize{$ \pi_0 $}} (4);
  \end{tikzpicture}
  \end{center}
For any $\gamma \in \Aut (l/k)$, by the definition of base change there exists a unique isomorphism of $G_0$ above $\gamma $ such that the following diagram commutes:
\[
\begin{tikzcd}[row sep=2.5em]
 &&
  G_0 \arrow[dd,swap,"t" near start] \arrow[dr,"\pi_1"] \\
& G_0 \arrow[rr,crossing over,swap,"\pi_1" near start] \arrow[ur,crossing over,"\tilde{\gamma}_{G}"] &&
  G \arrow[dd,"s"] \\
&& \Spec l \arrow[dr,"\pi_0"] \\
& \Spec l \arrow[rr,"\pi_0"] \arrow[ur,"\gamma^*"] \arrow[uu,<-,crossing over,"t"]&& \Spec k
\end{tikzcd}
\]
We denote this isomorphism $\tilde{\gamma}_{G}$.

\item For $G_0$ a split connected reductive $l$-group we assume that an isomorphism with $ H_l $ has been chosen, where $ H $ is a split algebraic group over the prime field of $ l $. Now in this special situation, for $\gamma \in \Aut (l)$, instead of 
%\fontdimen16\textfont2=4pt
%\fontdimen17\textfont2=4pt
$\tilde{\gamma}_{H}$
%\fontdimen16\textfont2=1pt
%\fontdimen17\textfont2=1pt
 we use the more suggestive notation $\Id_{\gamma}$.
\end{enumerate}
\end{definition}

\begin{remark}
\begin{enumerate}[(i)]
\item Note that when $l/k$ is a finite Galois extension, the collection $\lbrace $%\fontdimen16\textfont2=4pt
%\fontdimen17\textfont2=4pt
$\tilde{\gamma}_{G}$
%\fontdimen16\textfont2=1pt
%\fontdimen17\textfont2=1pt\!
$\rbrace_{\gamma \in \Gal(l/k)}$ is nothing but a descent datum on $G_0$ (as defined in \cite{Poo17}*{Proposition~4.4.2~(i)}) which descends to $G$.
\item Note that for $G_0$ a split connected reductive $ l $-group and $\gamma \in \Aut (l)$, if we choose a realisation of $G_0$ as a matrix group such that the realisation is defined over the prime field of $l$, then for $g=(g_{ij})\in G_0(l)$ and $\gamma \in \Aut (l)$, $ \Id_{\gamma} \circ \, g \circ (\gamma^*)^{-1} \in G_0(l) $ is given by the matrix whose $ij$-th coefficient is $\gamma^{-1}(g_{ij})$. This explains why we prefer to use the notation $\Id_{\gamma}$ in this situation.
\end{enumerate}
\end{remark}

We now study the behaviour of $\tilde{\gamma}_{G}$ under base change.
\begin{lemma}\label{Lem:choosing a good base change}
Let $k \leq l$ be a field extension of $ k $, let $ G_0 $ be an $ l $-group scheme and let $G$ be a $l/k$-form of $G_0$. Fix an isomorphism $ G_l \cong G_0 $, or in other words fix an exact diagram
 \begin{center}
 \begin{tikzpicture}[->]
  \node (1) [anchor=east] {$ G_0 $};
  \node (2) [right=1.5cm of 1] {$ G $};
  \node (3) [below=0.5cm of 1] {$ \Spec l $};
  \node (4) [below=0.52cm of 2] {$ \Spec k $};
  
  \path [every node/.style={font=\sffamily\small}]
  (1) edge node [above]  {$ \pi_1 $} (2)
        edge node [left] {$ t $} (3)
  (2) edge node  [right] {$ s $} (4)
  (3) edge node [above] {\footnotesize{$ \pi_0 $}} (4);
  \end{tikzpicture}
  \end{center}

Let $\alpha \in \Aut (k)$ and let $\beta \in \Aut (l)$ be such that $\beta |_{k} = \alpha $. Further assume that $G_0$ is split connected reductive. Then there exists a unique map $\pi_{\beta}\colon G_0\to G $ such that the following diagram commutes
\[
\begin{tikzcd}[row sep=2.5em]
G_0 \arrow[rr,"\Id_{\beta}"] \arrow[dr,swap,"\pi_{\beta}"] \arrow[dd,swap,"t"] &&
  G_0 \arrow[dd,swap,"t" near start] \arrow[dr,"\pi_1"] \\
& G \arrow[rr,crossing over,"\Id_{G}" near start] &&
  G \arrow[dd,"s"] \\
\Spec l \arrow[rr,"\beta^*" near end] \arrow[dr,swap,"\pi_0"] && \Spec l \arrow[dr,swap,"\pi_0"] \\
& \Spec k \arrow[rr,"\alpha^*"] \arrow[uu,<-,crossing over,"(\alpha^*)^{-1}s" near end]&& \Spec k
\end{tikzcd}
\]

Furthermore, all squares appearing in this diagram are exact.
\end{lemma}
\begin{proof}
The existence and uniqueness of $ \pi_{\beta} $ follows from the fact that the front square of the diagram is a base change. The fact that all squares are exact is a straightforward verification, using the fact that $ \alpha^* $ and $ \beta^* $ are isomorphisms.
\end{proof}

\begin{lemma}\label{Lem:the descent datum of a base change}
Keep the notations of Lemma~\ref{Lem:choosing a good base change}, so that in particular we chose an isomorphism $G_0\cong(\,^{\alpha }G)_l$ via the exact diagram
 \begin{center}
 \begin{tikzpicture}[->]
  \node (1) [anchor=east] {$ G_0 $};
  \node (2) [right=1.5cm of 1] {$ ^{\alpha}G=G $};
  \node (3) [below=0.5cm of 1] {$ \Spec l $};
  \node (4) [below=0.515cm of 2] {$ \Spec k $};
  
  \path [every node/.style={font=\sffamily\small}]
  (1) edge node [above]  {$ \pi_{\beta} $} (2)
        edge node [left] {$ t $} (3)
  (2) edge node  [right] {$ (\alpha^{\ast})^{-1} s $} (4)
  (3) edge node [above] {\footnotesize{$ \pi_0 $}} (4);
  \end{tikzpicture}
  \end{center}
With these identifications of base change, for all $\gamma \in \Aut (l/k)$ we have $ \tilde{\gamma}_{\,^{\alpha}G} = \Id_{\beta}^{-1}(\widetilde{\beta^{-1}\gamma \beta})_{G}\Id_{\beta} $.
\end{lemma}
\begin{proof}
The proof follows from the commutativity of the following diagram
\[
\begin{tikzcd}[row sep=2.5em]
& G_0 \arrow[rrr,swap,"\Id_{\beta}"] \arrow[dr,swap,"\pi_{\beta}"] \arrow[ddd,swap,"t" near start] &&&
  G_0 \arrow[ddd,"t" near start] \arrow[dl,"\pi_1"] & \\
&& G \arrow[r,crossing over,"\Id_{G}"] \arrow[ddd,swap,"(\alpha^*)^{-1}s"] &
  G \arrow[ddd,"s"] & & \\
G_0 \arrow[rrrrr,crossing over, "\Id_{\beta}"] \arrow[uur,"\tilde{\gamma}_{\alpha_{G}}"] \arrow[urr,crossing over,"\pi_{\beta}" near end] \arrow[ddd,swap,"t"] & & & & & G_0 \arrow[uul,swap,"(\widetilde{\beta^{-1}\gamma \beta})_{G}"] \arrow[ull,crossing over,swap,"\pi_{1}" near end] \arrow[ddd,"t"] \\
 & \Spec l \arrow[rrr,swap,"\beta^*"] \arrow[dr,swap,"\pi_0"] & & & \Spec l \arrow[dl,"\pi_0"] &  \\
 & & \Spec k \arrow[r,"\alpha^*"] & \Spec k & & \\
\Spec l \arrow[uur,swap,"\gamma^*"] \arrow[rrrrr,"\beta^*"] \arrow[urr,swap,"\pi_0" near end]  & & & & & \Spec l \arrow[uul,"(\beta^{-1}\gamma\beta)^*"] \arrow[ull,"\pi_0" near end]
\end{tikzcd}
\]
Indeed, $ \tilde{\gamma}_{\alpha_G} $ is defined to be the unique map such that the left hand side of the diagram commutes. But the front side and the back side of the diagram commutes by the definition of $ \pi_{\beta} $ (see Lemma~\ref{Lem:choosing a good base change}), whilst the right hand side of the diagram commutes by definition of $ (\widetilde{\beta^{-1}\gamma \beta})_{G} $. Also note that $ (\beta^{-1}\gamma\beta)^* = \beta^*\gamma^*(\beta^{-1})^* $, so that the bottom square of the diagram commutes as well. A diagram chase then shows that $ \Id_{\beta}^{-1}(\widetilde{\beta^{-1}\gamma \beta})_{G}\Id_{\beta} $ satisfies the property uniquely defining $ \tilde{\gamma}_{\alpha_G} $, as was to be shown.
\end{proof}

We can now state a clean descent formula for semilinear automorphisms in terms of cocycles. In this formula, we use the fact that for $G_0$ a split connected reductive $l$-group, $\Aut (G_0\to \Spec l)\cong \Aut G_0\rtimes \Aut (l)$, where the splitting of the exact sequence $ 1\to \Aut G_0\to \Aut (G_0\to \Spec l)\to \Aut (l) $ is realised by $\gamma \mapsto \Id_{\gamma}^{-1}$ (see Definition~\ref{Def:gammatilde and Idgamma} for the notation $\Id_{\gamma}$). This thus defines a (left) action of $\Aut (l)$ on $\Aut G_0$ that we denote $^{\gamma}f$ (for $\gamma \in \Aut (l)$ and $f\in \Aut G_0$). Explicitly, we have $^{\gamma}f = \Id_{\gamma}^{-1}f\Id_{\gamma} $.

\begin{lemma}\label{Lem:Condition for a semilinear auto. to descend}
Let $G_0$ be a split connected reductive $ l $-group, let $ k $ be a subfield of $ l $ such that $ l/k $ is a $ ( $possibly infinite$ ) $ Galois extension, and let $ G $ be a $ l/k $-form of $ G_0 $. Let $\beta \in \Aut (l) $ be such that $\beta(k)=k $, and let $\alpha = \beta_{\vert_{k}}\in \Aut (k) $. Fix an isomorphism $ G_l\cong G_0 $ and let $ c\colon \Gal (l/k)\to \Aut G_0 $ be the corresponding cocycle. Finally, let $b\in \Aut G_0 $. Then $b\Id_{\beta}\in \Aut (G_0\to \Spec l)$ descends to a semilinear automorphism of $ G $ over $\alpha$ if and only if $c_{\beta^{-1}\gamma \beta}\,^{\beta^{-1}\gamma \beta}b\,^{\beta^{-1}}c_{\gamma}^{-1} = b $ for all $\gamma \in \Gal (l/k)$.
\end{lemma}
\begin{proof}
Recall that a morphism of $G_0$ over $\beta$ is equivalent to an $ l $-morphism from $G_0$ to $^{\beta}G_0$. In this correspondence, $b\Id_{\beta}$ corresponds to $^{\beta}b\in \Aut G_0$, as can be seen directly from the diagram
\[
\begin{tikzcd}[row sep=2.5em]
G_0 \arrow[dr,"b\Id_{\beta}"] \arrow[d,swap," \Id_{\beta}^{-1}b\Id_{\beta} = \,^{\beta}b"] & \\
^{\beta}G_0 = G_0 \arrow[r,swap,"\Id_{\beta}"] \arrow[d,swap,"t"] & G_0 \arrow[d,"t"]\\
\Spec k \arrow[r,"\beta^*"] & \Spec k
\end{tikzcd}
\] Now by the general theory for morphisms between schemes with a descent datum, the $ l $-morphism $^{\beta}b\colon G_0\to \,^{\beta}G_0$ descends to a $ k $-morphism $ G\to \,^{\alpha}G $ if and only if $(\tilde{\gamma}_{^{\alpha}G})^{-1}(\,^{\beta}b)\tilde{\gamma}_{G} = \,^{\beta}b $ for all $\gamma \in \Gal (l/k) $ (technically speaking, this is only true for finite Galois extensions, but our schemes are of finite type. Hence $ ^{\beta}b $ is defined over a finite Galois extension of $ k $ and we are just trying to descend from there). Using Lemma~\ref{Lem:the descent datum of a base change}, we get that $^{\beta}b$ descends if and only $\Id_{\beta}^{-1}(\widetilde{\beta^{-1}\gamma^{-1} \beta})_{G} \Id_{\beta}(\,^{\beta}b)\tilde{\gamma}_{G} = \,^{\beta}b$. Finally, to transfer this to a cocycle condition, recall that in the correspondence between descent datum and cocycle, we have (in our notations) $\tilde{\gamma}_{G} = c_{\gamma^{-1}}\Id_{\gamma}$ (this of course relies on the fact that we used the same isomorphism $G_0\cong G_l$ to define $c_{\gamma}$ and $\tilde{\gamma}$). Furthermore, by definition $c\colon \Gal(l/k)\to \Aut G_0 $ is a cocycle for the Galois action introduced before the statement of the theorem, i.e.\ for $\gamma , \delta \in \Gal(l/k)$ we have $c_{\gamma \delta} = c_{\gamma }\,^{\gamma}c_{\delta} = c_{\gamma}\Id_{\gamma}^{-1}c_{\delta}\Id_{\gamma}$. Hence, the conclusion of the theorem readily follows.
\end{proof}
\begin{remark}
In our conventions, if $ \gamma \in \Gal(l/k) $ appears in exponent, then it acts on the element appearing below it on the right. So if one wishes to put more parenthesis in the formula $ c_{\beta^{-1}\gamma \beta}\,^{\beta^{-1}\gamma \beta}b\,^{\beta^{-1}}c_{\gamma}^{-1} $, the unique way to do so respecting this convention is by writing $ c_{\beta^{-1}\gamma \beta}(\,^{\beta^{-1}\gamma \beta}b)(\,^{\beta^{-1}}c_{\gamma}^{-1}) $. Note also that $ \beta \in \Aut (k) $ acts by group automorphisms on $ \Aut G $, so that $ ^{\beta^{-1}}(c_{\gamma}^{-1}) = (\,^{\beta^{-1}}c_{\gamma})^{-1} $, i.e.\ there is no need for any parenthesis to distinguish the two.
\end{remark}
\begin{remark}
If $\beta $ is the identity, our formula specialises to the usual condition for $b$ to descend to a $ k $-automorphism of $ G $, namely $c_{\gamma}\,^{\gamma}bc_{\gamma}^{-1}=b $ for all $\gamma \in \Gal(l/k)$. Also note that for $\gamma \in \Gal(l/k)$, the automorphism $\tilde{\gamma}_{G} =  c_{\gamma^{-1}} \Id_{\gamma} \in \Aut (G_0\to \Spec l)$ satisfies the condition to descend (of course, it descends to the trivial automorphism of $ G $).
\end{remark}

\section{Schemes of based root datum}\label{Sec:RD's scheme}
In \cite{Gil-Pol11}*{Expos\'e~$ 24 $, section~$ 3 $}, the authors define what they call a ``Dynkin's scheme'' of a reductive group $ G $. The strategy is to first define this Dynkin's scheme for split reductive groups, and then to use descent. The Dynkin's scheme is well suited to describe quasi-split semisimple groups that are adjoint or simply connected. Since there is not much more work to define a scheme of based root datum and since this allows us to treat the more general case of quasi-split reductive groups, we decided it was worth doing it.

In order to define a scheme of based root datum, we need the notion of a $ \Z $-module scheme and of perfect duality between two $ \Z $-module schemes. Recall that throughout the paper, the letter $ k $ stands for a field.
\begin{definition}
Let $ R $ be a $ k $-algebra and let $ M $ be a $ R $-scheme. $ M $ is called a $ \Z $-\textbf{module }$R$-\textbf{scheme} if $ M $ is a (non necessarily affine) commutative $ R $-group scheme.
\end{definition}

Recall that given any set $ E $ and a $ k $-algebra $ R $, we can consider the \textbf{constant object on} $E$ which is defined to be the $ R $-scheme $ E_R = \coprod \limits_{e\in E}\Spec R $. This defines a fully faithfull functor from the category of Sets to the category of $ R $-schemes, called the \textbf{constant object functor}. The constant object functor commutes with forming finite products (see \cite{Gil-Pol11a}*{Exposé~1, Section~1.8}). Hence given a $ \Z $-module $ M $, the constant scheme $ M_R $ acquires the structure of a $ \Z $-module $ R $-scheme. We can now define the notion of perfect pairing for $ \Z $-module $ k $-schemes.

\begin{definition}
Let $M,M'$ be two $\Z$-module $k$-schemes.
\begin{enumerate}
\item The \textbf{dual of} $M$, denoted $M^t$, is defined to be the functor from the category of $k$-algebras to the category of sets sending a $ k $-algebra $R$ to $\Hom_{R}(M_R,\Z_R)$.
\item We say that $M$ and $M'$ are in \textbf{duality} if there exists $f\in \Hom_k (M\times M'\to \Z_k)$. We say that the duality $f$ is \textbf{perfect} if for all $ k $-algebra $R$, the map $M'(R)\to M^t(R)\colon m\mapsto (f(.\,,m)\colon M_R\to \Z_R) $ is an isomorphism (here, for any $R$-algebra $R'$ and $n\in M_R(R')$, $f(.\,,m)(n) = f(n,m_{R'})$, with $m_{R'}$ being the image of $m$ under $M'(R)\to M'(R')$).
\end{enumerate}
\end{definition}

\begin{remark}
As usual in this situation, one should restrict the categories under considerations to avoid set theoretic problems. One way to do so is by using universe.
\end{remark}

Note that $M^t$ is a commutative group functor, and hence $M^t$ is a $\Z$-module $ k $-scheme if $M^t$ is representable. Also note that given $f\in \Hom_k(M,M')$ we can define $f^t\in \Hom_{k\text{-functors}}({M'}^t,M^t)$ by mimicking the definition for $\Z$-modules. Namely for all $ k $-algebras $ R $, $ f^t(R)\colon {M'}^t(R)\to M^t(R)\colon \alpha \mapsto \alpha \circ f_R $, where $ f_R\colon M_R\to M'_R $ denotes the base change of $ f $ to $ R $. We call $ f^t $ the \textbf{dual of} $f$.

By a \textbf{reduced based root datum} $R$, we mean a reduced root datum $ (M,M^*,\Phi ,\Phi^*)$ (as defined in \cite{Gil-Pol11}*{Exposé~21, Définition~1.1.1 and Definition~2.1.3}) together with a choice of simple roots $\Delta \subset \Phi $. We can finally give the definition of a $k$-scheme of root datum.

\begin{definition}\label{Def:scheme of based root datum}
\begin{enumerate}
\item A $k$-\textbf{scheme of based root datum} is a $5$-tuple $\mathcal{R} = (\mathcal{M},\mathcal{M}^*, \Psi ,\Psi^*, \Gamma )$ where:
\begin{enumerate}
\item $\mathcal{M}$ and $ \mathcal{M^*} $ are $ \Z $-module $ k $-schemes in perfect duality.
\item $ \Psi ,\Psi^* $ and $ \Gamma $ are finite $k$-schemes, and there are closed immersions $\Gamma \hookrightarrow \Psi \hookrightarrow \mathcal{M}$ and $ \Psi^*\hookrightarrow \mathcal{M}^* $.
\item There is an isomorphism of $ k $-schemes $ \Psi \cong \Psi^* $.
\item There exists a reduced based root datum $R = (M,M^*,\Phi ,\Phi^*, \Delta)$ and a finite Galois extension $l/k$ together with an isomorphism of $\Z$-module $l$-schemes $f\colon M_l\to \mathcal{M}_l$ such that $f$ induces an isomorphism of $ l $-schemes $\Phi_l\cong \Psi_l $, $f^t$ induces an isomorphism of $ l $-schemes $ {\Psi }_l^* \cong \Phi_l^*$, $f(\Delta_l) = \Gamma_l $ and the composition $ \Phi_l\cong \Psi_l\cong \Psi^*_l\cong \Phi_l^* $ induces the bijection $ \Phi \to \Phi^*\colon \alpha \mapsto \alpha^* $ given in the definition of $ R $. In this case, we say that $ \mathcal{R} $ is \textbf{of type} $ R $.
\end{enumerate}

\item Let $R = (M,M^*,\Phi ,\Phi^*, \Delta)$ be a reduced based root datum. Using the constant object functor, the $5$-tuple $R_k = (M_k,M^*_k,\Phi_k ,\Phi^*_k, \Delta_k)$ has a natural structure of a $ k $-scheme of based root datum. We call it the \textbf{split} $k$\textbf{-scheme of based root datum of type} $R$. A \textbf{split} $k$-\textbf{scheme of based root datum} is a split $k$-scheme of based root datum of type $R$ for some reduced based root datum $R$.

\item Given $\mathcal{R} = (\mathcal{M},\mathcal{M}^*, \Psi ,\Psi^*, \Gamma )$ and $\mathcal{R}' = (\mathcal{M}',{\mathcal{M}'}^*, \Psi ',{\Psi '}^*, \Gamma ')$ two $k$-schemes of based root datum, a $ k $-morphism $\mathcal{R}\to \mathcal{R}'$ is a $ k $-homomorphism $ f\colon \mathcal{M}\to \mathcal{M}' $ such that $f$ induces two $ k $-isomorphisms $\Psi\cong \Psi ' $ and $ \Gamma \cong \Gamma ' $, and such that $f^t$ induces a $ k $-isomorphism $ {\Psi '}^* \cong \Psi^*$.
\end{enumerate}
\end{definition}
\begin{remark}
Let us stress that with this definition, if a scheme of based root datum is of type $ R $, then $ R $ is a \textit{reduced} based root datum. It would be safer (but more tedious) to call these objects ``$ k $-schemes of reduced based root datum''.
\end{remark}

\begin{remark}\label{Rem:Root datum and Galois descent}
Let $ k_s $ be a separable closure of $ k $, let $R_k = (M_k,M^*_k,\Phi_k ,\Phi^*_k, \Delta_k)$ be a split $k$-scheme of based root datum of type $R$ and let $E = \Aut R$ (see \cite{Gil-Pol11}*{Exposé~21, Définition~6.1.1} for the definition of the morphisms in the category of based root datum). Then $ \Aut R_{k_s} = E$, and the action of $\Gal (k_s/k)$ on $E$ is trivial. Hence elements of $H^1(k,E)$ are continuous homomorphisms $\Gal(k_s/k)\to E$ up to conjugation. Also recall that $H^1(k,E)$ classifies $k$-schemes of based root datum of type $R$. Indeed, Galois descent for $M_k$ and $M_k^*$ is effective because they can be covered by quasi-affine open that are stable under $ \Gal(k_s/k) $ (hence effectivenes is ensured by \cite{Poo17}*{Theorem~4.3.5}), and the other structures (the duality $ M_k\times M^*_k\to \Z_k $, the closed immersions $ \Delta_k\hookrightarrow \Phi_k\hookrightarrow M_k $, $ \Phi^*_k\hookrightarrow M^*_k $ and the $ k $-isomorphism $ \Phi_k\cong \Phi_k^* $) will descend as well.
\end{remark}
\begin{remark}
If $R = (M,M^*,\Phi ,\Phi^*, \Delta)$ is a based root datum such that $\Phi$ is empty, then $\Aut R \cong GL_n(\Z )$ where $n$ is the rank of $M$. Hence in this case, $k$-schemes of based root datum of type $R$ classify $k$-tori of rank $n$. 
\end{remark}

Note that given $ \alpha \in \Aut (k) $ and a $ k $-scheme of based root datum $ \mathcal{R} $, we have an obvious notion of base change of $ \mathcal{R} $ along $ \alpha $, and we denote this base change by $ ^{\alpha}\mathcal{R} $.
\begin{definition}\label{Def:semi-auto of root datum}
Let $ \mathcal{R} $ be a $k$-scheme of based root datum, and let $ \alpha \in \Aut (k) $. A \textbf{semilinear automorphism of} $ \mathcal{R} $\textbf{ over} $ \alpha $ is an isomorphism of $k$-schemes of based root datum $ f_{\alpha}\colon \mathcal{R}\to \,^{\alpha}\mathcal{R} $. Given a semilinear automorphism $ f_{\alpha} $ (respectively $f'_{\beta}$) over $ \alpha $ (respectively $ \beta $) in $ \Aut (k) $, their composition is $ (\, ^{\alpha}f'_{\beta})f_{\alpha} $, which is a semilinear automorphism over $ \alpha \beta $. We denote the group of semilinear automorphisms of $ \mathcal{R} $ by $ \Aut (\mathcal{R}\to \Spec k) $.
\end{definition}

As in the case of $ k $-group schemes, for $ \mathcal{R} $ a $ k $-scheme of based root datum, we have a homomorphism $ \Aut (\mathcal{R} \to \Spec k)\to \Aut (k)\colon f_{\alpha}\mapsto \alpha ^{-1}  $. We let $ \Aut_{\mathcal{R}}(k) $ be the image of this homomorphism. Furthermore denoting the $ k $-automorphisms of the based root datum $ \mathcal{R} $ by $ \Aut \mathcal{R} $ (or also (\underline{$ \Aut $} $ \mathcal{R})(k) $, following the conventions discussed at the start of Section~\ref{Sec:Semiauto and Galois descent}), we get a short exact sequence $ 1\to \Aut \mathcal{R}\to \Aut (\mathcal{R} \to \Spec k)\to \Aut_{\mathcal{R}}(k)\to 1 $.

We now discuss how to associate functorially a $ k $-scheme of based root datum to a connected reductive $ k $-group. One possible approach would be to take an inductive limit of based root datum in the split case, and then descend this canonical object to any form. This would lead to the same construction as the one we now explain.

Actually, it suffices to incorporate our definition of the $k$-scheme of based root datum in \cite{Gil-Pol11}*{Exposé~24, Théorème~3.11}, by replacing principal Galois cover of group $E = \Aut R$ with the objects over $k$ that they classify (i.e.\ $k$-schemes of based root datum). As a corollary, we will get the definition of the $ k $-scheme of based root datum of a connected reductive $ k $-group. First, we recall the definition of the group scheme of exterior isomorphisms.
\begin{definition}[\cite{Gil-Pol11}*{Exposé~24, Corollaire~1.10}]
Let $G,G'$ be two connected reductive $ k $-group of type $R$, for some reduced based root datum $R$. Then  $ \Ad G $ acts freely (on the right) on the $k$-group functor $ \underline{\Isom}_{k\text{-gr.}}(G,G') $. We define the $k$-\textbf{group functor of exterior isomorphisms between} $G$ \textbf{and }$G'$ to be the quotient sheaf $ \underline{\Extisom} (G,G') = \underline{\Isom}_{k\text{-gr.}}(G,G')/\Ad G $. 
\end{definition}
\begin{remark}\label{Rem:computing Extisom}
Actually, \cite{Gil-Pol11}*{Exposé~24, Corollaire~1.10} asserts that this quotient is representable. Since it will be useful later, here is an explicit description of $ \underline{\Extisom} (G,G') $ using cocycles: let $k_s$ be a separable closure of $k$, let $ E = \Aut R $, let $ G_0 $ be the (split) connected reductive $ k_s $-group of type $ R $ and identify $ \underline{\Aut}~G_0 $ with $ \Ad G_0\rtimes E_{k_s} $ (after a choice of pinning for $ G_0 $, see \cite{Gil-Pol11}*{Exposé~24, Théorème~1.3}). Fix an isomorphism $G_0\cong G_{k_s} $ (respectively $ G_0\cong G'_{k_s}$), denote by $c$ (respectively $ c' $) the corresponding cocycle $ \Gal (k_s/k)\to \Aut G_0 $ and let $ \tilde{c} $ (respectively $ \tilde{c}' $) be the composition of this cocycle with $ \Aut G_0 \to E $. Further assume that the pinning of $ G_0 $ is defined over $ k $, so that the Galois action on $ \Aut R_{k_s} = E $ is trivial. Then $ \underline{\Extisom} (G,G') $ is a $ k_s/k $-form of $ E_{k_s} $ given by the Galois action $ \gamma .f = \tilde{c}'_{\gamma}f\tilde{c}_{\gamma}^{-1} $ (for all $ f\in E_{k_s} $ and for all $ \gamma \in \Gal (k_s/k) $). This follows directly from the Galois condition for an automorphism of $ G_0 $ to descend to an isomorphism $ G\to G' $, together with the fact that we are moding out by adjoint automorphisms.
\end{remark}

Let us also recall the notion of a quasi-pinning.
\begin{definition}\label{Def:quasi-split groups}
Let $ G $ be a connected reductive $ k $-group. If it exists, a \textbf{quasi-pinning} of $ G $ is:
\begin{enumerate}
\item A choice of a Borel subgroup $ B $ containing a maximal torus $ T $ of $ G $. Once this is chosen, let $ k_s $ be a separable closure of $ k $, let $ \Delta $ be the fundamental roots of $ G_{k_s} $ corresponding to the pair $ (T_{k_s},B_{k_s}) $ and for $ \alpha \in \Delta $, let $ \mathfrak{g}_{\alpha} $ be the corresponding one dimensional subspace of $ \Lie (G_{k_s}) $.
\item A choice of a nontrivial element $ X_{\alpha}\in \mathfrak{g}_{\alpha} $ for all $ \alpha \in \Delta $ such that for all $ \gamma \in \Gal(k_s/k) $, $ \gamma .X_{\alpha} = X_{\gamma (\alpha)} $.
\end{enumerate}
If $ G $ has a quasi-pinning, we say that $ G $ is \textbf{quasi-split}.
\end{definition}

The more classical definition for a connected reductive $ k $-group to be quasi-split is that it possesses a Borel subgroup. It is well-known that this definition agrees with Definition~\ref{Def:quasi-split groups} (and the equivalence is proved in a more general setting in \cite{Gil-Pol11}*{Exposé~24, Proposition~3.9.1}). For the convenience of the reader, let us reprove this fact.
\begin{lemma}
Let $ G $ be a connected reductive $ k $-group. If $ G $ has a borel subgroup, then $ G $ has a quasi-pinning.
\end{lemma}
\begin{proof}
Let $ B $ be a borel subgroup of $ G $. Then it contains a maximal torus $ T $ of $ G $ (see for example \cite{Gil-Pol11}*{Exposé~22, Corollaire~5.9.7}). Let $ k_s, \Delta, \alpha \in \Delta $ and $ \mathfrak{g}_{\alpha} $ be as in Definition~\ref{Def:quasi-split groups} for the pair ($ T,B $). Let $ H\leq \Gal (k_s/k) $ be the stabiliser of $ \alpha $, and let $ k_{\alpha} $ be the subfield of $ k_s $ fixed by $ H $. Then there exists an $ H $-equivariant isomorphism $ k_{\alpha}\otimes_{k_{\alpha}}k_s\cong \mathfrak{g}_{\alpha} $ (this holds because all $ k_s/k_{\alpha} $-forms of the vector space $ k_s $ are equivalent by Hilbert's 90). Set $ X_\alpha = 1\in k_{\alpha}\subset \mathfrak{g}_{\alpha} $. Now, for $ \beta $ in the $ \Gal (k_s/k) $-orbit of $ \alpha $, we set $ X_{\beta} = \gamma .X_{\alpha} $ where $ \gamma \in \Gal (k_s/k) $ is any element such that $ \gamma (\alpha) = \beta $. The point is that $ X_{\beta} $ does not depend on a choice of $ \gamma \in \Gal(k_s/k) $ such that $ \gamma (\alpha )=\beta $ because $ H $ acts trivially on $ X_{\alpha} $. Doing so for each orbit of $ \Gal (k_s/k) $ on $ \Delta $ concludes the proof. 
\end{proof}

\begin{theorem}[\cite{Gil-Pol11}*{Exposé~24, Théorème~3.11}]\label{Thm:Quasi-split groups and BRD}
Let $ R = (M,M^*,\Phi ,\Phi^* ,\Delta) $ be a reduced based root datum. Consider the following categories:
\begin{enumerate}
\item The category $ \underline{\BRD} $ of $k$-schemes of based root datum of type $ R $. The morphisms are the isomorphisms of $ k $-schemes of based root datum.
\item The category $ \underline{\RedExt} $ of connected reductive $k$-groups of type $R$. The morphisms between $G$ and $G'$ are elements of the group $\underline{\Extisom} (G,G')(k)$.
\item The category $ \underline{\QsPin} $ of connected reductive quasi-split $k$-groups of type $R$, together with a choice of quasi-pinning. The morphisms are the isomorphisms preserving the quasi-pinning.
\end{enumerate}
These three categories are equivalent. More precisely, we have a diagram of functors between categories
\[
\begin{tikzcd}[row sep=2.5em]
 \underline{\BRD} \arrow[rr,"\text{qspin}"] &&
  \underline{\QsPin} \arrow[dl,"\iota"] \\
 & \underline{\RedExt} \arrow[ul,"\text{brd}"] & 
\end{tikzcd}
\]
such that the composition of those three functors $ ( $starting with anyone of them$ ) $ is naturally isomorphic to the identity.
\end{theorem}
\begin{proof}
We follow the proof given in \cite{Gil-Pol11}*{Exposé~24, Section~3.11}, with the advantage that we can work with the more concrete Galois descent, and that the notion of pinning is simpler over fields.

Set $ E = \Aut R $. Let $ k_s $ be a separable closure of $ k $ and let $ G_0 $ be the (split) connected reductive $ k_s $-group of type $ R $. For the proof, we choose a pinning for $G_0$ which is defined over $ k $, i.e. we choose a pinning for the split connected reductive $ k $-group of type $ R $ and we base change it to a pinning of $ G_0 $. In particular we choose a torus $ T_0 $ contained in a Borel subgroup $ B_0 $ (both defined over $ k $), and we get an identification $ \underline{\Aut}~G_0 \cong \Ad G_0\rtimes E_{k_s} $ (where the $ \Gal (k_s/k) $-action on $ E_{k_s} $ is trivial), and in particular an embedding $E=  E_{k_s}(k_s)\to (\underline{\Aut}~G_0)(k_s)=\Aut G_0 $.
\begin{enumerate}

\item The functor $\iota$. On objects, $\iota (G) $ is the natural inclusion whilst for $f\in \Mor_{\QsPin}(G,G')$, $\iota (f)$ is the projection of $f$ in $ \underline{\Extisom} (G,G')(k) = (\underline{\Isom}_{k\text{-gr.}}(G,G')/\Ad G)(k) $.
\item The functor qspin. Let $\mathcal{R}$ be a $k$-scheme of based root datum of type $R$. Choose an isomorphism $R_l\cong \mathcal{R}_l$ for some finite Galois extension $ l/k $ and let $c\colon \Gal (k_s/k)\to E $ be the corresponding cocycle. The quasi-split group qspin($ \mathcal{R} $) is the $ k_s/k $-form of $G_0$ defined by the cocycle $ c\colon \Gal (k_s/k)\to E \to \Aut G_0 $. Note that this cocycle preserves $T_0$ and $B_0$, so that qspin($ \mathcal{R} $) is indeed quasi-split. We choose for quasi-pinning on qspin($ \mathcal{R} $) the pair $ (T_0,B_0) $ descended to $ k $, and for $ \alpha \in \Delta $, we choose the element $ X_{\alpha}\in \Lie (G_0) $ to be the same as the one appearing in the pinning of $ G_0 $. Since the pinning of $ G_0 $ is defined over $ k $ by assumption, this indeed constitutes a quasi-pinning of qspin($ \mathcal{R} $). Finally, for a morphism $ f\in \Mor_{\BRD}(\mathcal{R},\mathcal{R}') $, qspin($ f $) is defined to be the descent of $f_{k_s}\in \Mor (\mathcal{R}_{k_s},\mathcal{R}'_{k_s})\cong E\leq \Aut G_0$ to an isomorphism qspin($ \mathcal{R} $)$ \to $qspin($ \mathcal{R}' $).
\item The functor brd. For $G$ a connected reductive group of type $R$, choose an isomorphism $ G_0\cong G_{k_s} $ and let $c\colon \Gal (k_s/k)\to \Aut G_0 $ be the corresponding cocycle. Consider $\tilde{c}$, the composition of $c$ with the projection $ \Aut G_0\to E $. Now brd($G$) is defined to be the $k_s/k$-form of the split $k_s$-scheme of root datum $R_{k_s}$ obtained by Galois descend using the cocycle $ \tilde{c} $. Whilst for a morphism $ f\in \Mor_{\RedExt}(G,G') $, $ f_{k_s}\in \underline{\Extisom}(G_{k_s},G_{k_s}')(k_s)\cong \Aut R_{k_s} $, and brd($ f $) is defined to be the descent of $ f_{k_s} $ to an isomorphism brd($ G $)$ \to $brd($ G' $).
\end{enumerate}

We now check that the composition of those three functors (starting with anyone of them) is naturally isomorphic to the identity.
\begin{enumerate}
\item brd$\, \circ \, i \, \circ \, $qspin $ \cong \Id_{\BRD} $. Let $c'\colon \Gal (k_s/k)\to E $ be the cocycle arising from a choice of $ R_l\cong \mathcal{R}_l $. By definition of qspin, a choice of isomorphism $ G_0\cong \text{qspin}(\mathcal{R})_{k_s} $ gives rise to a cocycle $c$ which is cohomologous to $c'$ (as cocycles with values in $ \Aut G_0 $), hence the $\tilde{c}$ appearing in the definition of brd is cohomologous to $ c' $ as well (as cocycles with values in $ E $).

\item qspin $ \circ $ brd$\, \circ \, i \cong \Id_{\QsPin} $. We need to check that given a quasi-split group $ G $ together with a choice of isomorphism $ G_0\cong G_{k_s} $ and corresponding cocycle $ c\colon \Gal(k_s/k)\to \Aut G_0 $, then $ c $ is cohomologous to $ c $ composed with $ \Aut G_0\to E\to \Aut G_0 $. The quasi-pinning on $ G $ gives a pinning of $ G_{k_s} $, which is sent by $ G_0\cong G_{k_s} $ to a pinning of $ G_0 $. Up to conjugation by $g\in G_0(k_s)$, which has the effect of replacing $c$ by a cohomologous cocycle, we can assume that this pinning of $ G_0 $ is the one we chose from the outset. Because the pinning of $ G_0 $ is defined over $ k $, it is invariant under the action of $ \Gal (k_s/k) $. Hence, the cocycle $ c $ has values in $ E $, as wanted.

\item $ i\, \circ $ qspin $ \circ $ brd $ \cong \Id_{\RedExt} $. Let $ G $ be a connected reductive group of type $ R $. We want to check that $ \underline{\Extisom} (G,G')(k)\neq \emptyset $, where $ G' = (i\, \circ $ qspin $ \circ $ brd)$ (G) $. Let $ G_0\cong G_{k_s} $ be the chosen isomorphism to define brd$(G)$, with corresponding cocycle $ c $, and let $ \tilde{c} $ be the projection of $ c $ under $ \Aut G_0\to E $. By definition, a cocycle defining $ G' $ is cohomologous to $ \tilde{c} $, so we can assume that $ G' $ is defined by $ \tilde{c} $. Now, by Remark~\ref{Rem:computing Extisom}, the identity on $ G_0 $ descends to an element of $ \underline{\Extisom} (G,G')(k) $, concluding the proof. \qedhere
\end{enumerate}
\end{proof}

In view of Theorem~\ref{Thm:Quasi-split groups and BRD}, one can attach in a functorial way a $k$-scheme of based root datum to any connected reductive $k$-group.
\begin{definition}
Let $ G $ be a connected reductive $k$-group. The $k$-\textbf{scheme of based root datum associated to} $ G $ is brd$ (G) $, where brd is the functor appearing in Theorem~\ref{Thm:Quasi-split groups and BRD}. We denote it $ \mathcal{R}(G) $.
\end{definition}

The crucial input is that taking the scheme of based root datum commutes with base change.
\begin{lemma}\label{Lem:base change commute with taking root datum}
Let $ G $ be a connected reductive $k$-group and let $ \alpha $ be an automorphism of $ k $. Then $ \mathcal{R}(\,^{\alpha}G) \cong \,^{\alpha}\mathcal{R}(G) $, naturally in $ G $.
\end{lemma}
\begin{proof}
Let $R$ be the type of $ G $, let $ k_s $ be a separable closure of $ k $, and let $ G_0 $ be the (split) connected $ k_s $-group of type $ R $. Let $ \beta $ be an extension of $ \alpha $ to $ k_s $, choose an isomorphism $ G_0\cong G_{k_s} $, and let $ G_0\cong \,(^{\alpha}G)_{k_s} $ be the corresponding isomorphism defined in Lemma~\ref{Lem:choosing a good base change}. Now by Lemma~\ref{Lem:the descent datum of a base change}, if $ c $ denotes the cocycle defining $ G $, then the corresponding cocycle $ ^{\alpha}c $ defining $ ^{\alpha}G $ is given by $ (\, ^{\alpha}c)_{\gamma} = \Id_{\beta }^{-1}c_{\beta^{-1}\gamma \beta}\Id_{\beta} $, for all $ \gamma \in \Gal (k_s/k) $. Finally, we choose a pinning of $G_0$ defined over the prime field of $k$ (so that we can identify $\Aut G_0\cong (\Ad G_0)(k_s)\rtimes \Aut R $), and we let $ \tilde{c} $ (respectively $^{\alpha }\tilde{c}$) be the projection of $ c $ (respectively $ ^{\alpha}c $) under $ \Aut G_0\to \Aut R $. Note that since the Galois action on $ \Aut R $ is trivial, $ (^{\alpha}\tilde{c})_{\gamma} $ is just the projection of $ c_{\beta^{-1}\gamma \beta} $ onto $ \Aut R $.

On the other side, let $ R_{k_s} $ be the split $ k $-scheme of based root datum of type $ R $. The (choice of) cocycle defining $ \mathcal{R}(G) $ is $ \tilde{c}\colon \Gal(k_s/k)\to \Aut R $. Now exactly the same computation as for algebraic groups (i.e.\ repeating Lemma~\ref{Lem:choosing a good base change} and Lemma~\ref{Lem:the descent datum of a base change} in the category of schemes of based root datum) shows that a cocycle defining $ ^{\alpha}\mathcal{R}(G) $ is given by $\gamma \mapsto \Id_{\beta }^{-1}\tilde{c}_{\beta^{-1}\gamma \beta}\Id_{\beta} = \tilde{c}_{\beta^{-1}\gamma \beta} $. But this a also the chosen cocycle defining $ \mathcal{R}(^{\alpha}G) $, as was to be shown. The naturality in $ G $ of this isomorphism is straightforward.
\end{proof}
\begin{remark}
Of course, for this whole section, we did not need the fact that the base scheme is the spectrum of a field, and for example, Lemma~\ref{Lem:base change commute with taking root datum} should be true over any base scheme, and under any base change. The advantage of working over a field is that the notion of pinning is simpler, and that Galois descent is more concrete than fppf descent.
\end{remark}

The proof of Theorem~\ref{Thm: MainThm1} now follows easily from Theorem~\ref{Thm:Quasi-split groups and BRD} and Lemma~\ref{Lem:base change commute with taking root datum}.
\begin{proof}[Proof of Theorem~\ref{Thm: MainThm1}]
Recall that to give an automorphism $ f_{\alpha} $ of $ G $ over $\alpha \in \Aut (k)$ is equivalent to give an isomorphism of $ k $-group schemes $ f\colon G\to \,^{\alpha}G $. Hence, projecting $ f $ to an element $ \bar{f}\in \underline{\Extisom}(G,\,^{\alpha}G)(k) $ and using the functor brd defined in Theorem~\ref{Thm:Quasi-split groups and BRD}, we get an isomorphism brd$ (\bar{f})\colon \mathcal{R}(G)\to \mathcal{R}(\,^{\alpha}G)\cong \,^{\alpha}\mathcal{R}(G) $ (where we used Lemma~\ref{Lem:base change commute with taking root datum} for the last isomorphism). Now since brd is a functor, and because the isomorphism $ \mathcal{R}(\,^{\alpha}G)\cong \,^{\alpha}\mathcal{R}(G) $ is natural in $ G $, the map $ f_{\alpha}\mapsto $brd$ (\bar{f}) $ is a group homomorphism which is natural in $ G $. Furthermore, the underlying automorphism of the field is preserved by this homomorphism. To conclude the first part of the proof, note that $ f_{\alpha} $ is in the kernel of this homomorphism if and only if $ \alpha $ is trivial and $ \bar{f} $ is trivial in $ \underline{\Extisom} (G, G)(k) $, which is to say that $ f\in (\Ad G)(k) $.

For the last assertion, assume that $ G $ is quasi-split and choose a quasi-pinning of it. Define the subgroup $ H=\lbrace f_{\alpha}\in \Aut (G\to \Spec k)~\vert~ f_{\alpha } \text{ preserves the quasi-pinning of } G\rbrace $. Seeing $ f_{\alpha} $ as an isomorphism $ f $ from $ G $ to $ ^{\alpha}G $, the condition for $ f $ to belong to $ H $ is that it preserves the quasi-pinnings (where $ ^{\alpha}G $ is endowed with the quasi-pinning on $G$ based changed to $ ^{\alpha}G $). Now the fact that $ H $ maps isomorphically onto $ \Aut(\mathcal{R}(G)\to \Spec k) $ under $ \Aut (G\to \Spec k)\to \Aut (\mathcal{R}(G)\to \Spec k) $ is a direct consequence of Lemma~\ref{Lem:base change commute with taking root datum} and of the equivalence of categories $ \underline{\BRD} $ and $ \underline{\QsPin} $ in Theorem~\ref{Thm:Quasi-split groups and BRD}.
\end{proof}

\begin{remark}
For $ G $ a connected reductive $ k $-group which is not quasi-split, the decomposition $ \Aut G\cong (\Ad G)(k)\rtimes \Out G $ as a semidirect product is usually destroyed. Similarly, one should not expect to obtain a semidirect decomposition of $ \Aut (G \to \Spec k) $ for a general connected reductive $ k $-group. Investigating a possible semidirect decomposition of the group of semilinear automorphisms of simple algebraic groups is an entirely different matter when $ G $ is not quasi-split, as is illustrated by our treatment of the $ \SL_n(D) $ case in Section~\ref{Sec:SL_n(D)}.
\end{remark}

As a corollary of Theorem~\ref{Thm: MainThm1}, we obtain a proof of Theorem~\ref{Thm: MainThm2}.
\begin{proof}[Proof of Theorem~\ref{Thm: MainThm2}]
By Theorem~\ref{Thm: MainThm1}, $ \Aut_{\mathcal{R}(G)} (k) = \Aut_G(k) $. We thus obtain the following commutative diagram:

\begin{center}
\begin{tikzcd}[row sep=2.5em, column sep=3em] 1 \arrow{dr} &1 \arrow[d] & &1&1\\ & (\Ad G)(k) \arrow{dd} \arrow{dr} &  & \Aut_{\mathcal{R}(G)} (k) \arrow{u}  \arrow{ur}  & \\ & & \Aut (\! G \to \Spec k\! ) \arrow[dr,swap, "\pi"']  \arrow[ur,"p_1"] &  & \\ & \Aut G \arrow[d]  \arrow{ur} &  & \Aut (\mathcal{R}(G)\to \Spec k) \arrow[uu,"p_2"]\arrow[ul, shift left=1.5ex,swap, "\iota"'] \arrow{dr} & \\ 1 \arrow{ur} & Out G \arrow[d] & & \Aut \mathcal{R}(G) \arrow{u} & 1 \\  &1 & &1 \arrow{u} & 
\end{tikzcd}
\end{center}
where all diagonal lines and vertical lines are exact. Here, $ \pi $ denotes the homomorphism provided by Theorem~\ref{Thm: MainThm1}, and $ \iota $ is a section of $ \pi $ (which exists, again by Theorem~\ref{Thm: MainThm1}). Note that in particular, $ \iota $ preserves the underlying field automorphism, i.e.\ $ p_1\circ \iota = p_2 $.

We thus conclude that the short exact sequence $ 1\to \Aut G\to \Aut (G \to \nolinebreak \Spec k)\to \Aut_{G}(k)\to 1 $ splits if and only if the short exact sequence involving $ k $-schemes of based root datum $ 1\to \Aut \mathcal{R}(G) \to \Aut (\mathcal{R}(G) \to \Spec k)\to \Aut_{\mathcal{R}(G)}(k)\to 1 $ does, as was to be shown.
\end{proof}

\section{Semilinear automorphisms and Galois cohomology}\label{Sec:Semilinear Galois cohomology}
We have just proved that for any connected reductive algebraic $ k $-group $ G $, we have a natural exact sequence $ 1\to (\Ad G)(k)\to \Aut (G\to \Spec k)\to \Aut (\mathcal{R}(G)\to \Spec k) $. It would be nice to be able to express the failure of surjectivity on the right using Galois cohomology. We explain in this section how to do so.

In this section, $ k_s $ denotes a separable closure of $ k $ with Galois group $ \Gamma = \Gal (k_s/k) $, $ R $ is a reduced based root datum, $ G_0 $ is a (split) connected reductive $ k_s $-group of type $ R $ with a choice of pinning defined over the base field of $ k $, and $ R_{k_s} $ is the split $ k_s $-scheme of based root datum of type $ R $. We furthermore set $ E = \Aut R $ and we let $ E_{k_s} $ be the corresponding constant object over $ k_s $. Also, we again use the convention that $ G_0 $ comes together with a preferred split form of it over the prime field of $ k $. In particular, we get a decomposition $ \Aut (G_0\to \Spec k_s)\cong \Aut G_0\rtimes \Aut (k_s) $, where the splitting $ \Aut (k_s)\to \Aut (G_0\to \Spec k_s) $ is given by $ \beta^{-1} \mapsto \Id_{\beta} $ (see Definition~\ref{Def:gammatilde and Idgamma}).

\begin{definition}\label{Def:semilinear auto. preserving k}
\begin{enumerate}
\item Given a field extension $l\geq k$, we set $ \Aut (l\geq k) = \lbrace \alpha \in \Aut (l)~\vert~\alpha (k) = k\rbrace $
\item We set $ \Aut (G_0\to \Spec k_s\geq k) = $  $$ \lbrace f_{\alpha}\in \Aut (G_0\to \Spec k_s)~\vert~\text{the underlying } \alpha \in \Aut (k_s) \text{ belongs to } \Aut (k_s\geq k)\rbrace. $$
\item We denote an element of $ \Aut (G_0\to \Spec k_s\geq k)\cong \Aut G_0\rtimes \Aut (k_s\geq k) $ by $ b\Id_{\beta} $ (where $ b \in \Aut G_0 $ and $ \beta \in \Aut (k_s\geq k) $).
\end{enumerate}
\end{definition}
\begin{remark}
In Definition~\ref{Def:semilinear auto. preserving k}, $ \alpha $ is required to globally preserves $ k $, but its restriction to $ k $ can be non-trivial. Also, we will use the fact that $ \Aut (l/k) $ (see Definition~\ref{Def:Notation Aut(l/k) and base change}) is a normal subgroup of $ \Aut (l\geq k) $.
\end{remark}

\begin{definition}\label{Def:semilinear Galois action}
Let $ G $ be a connected reductive $ k $-group of type $ R $, choose an isomorphism $ G_0\cong G_{k_s} $ and let $ c\colon \Gamma \to \Aut G_0 $ be the corresponding Galois cocycle. We define the \textbf{semilinear Galois action corresponding to }$ c $ (we also say corresponding to $ G_0\cong G_{k_s} $) on $ \Aut (G_0\to \Spec k_s\geq k) \cong \Aut G_0\rtimes \Aut (k_s) $ as follows: 
\begin{equation*}
\text{ For all } b\in \Aut G_0,~\beta \in \Aut (k_s\geq k) \text{ and } \gamma \in \Gamma,~~~\gamma .(b\Id_{\beta}) = c_{\beta^{-1}\gamma \beta}\,^{\beta^{-1}\gamma \beta}b\,^{\beta^{-1}}c_{\gamma}^{-1}\Id _{\beta}.
\end{equation*}
\end{definition}
\begin{remark}
In view of Lemma~\ref{Lem:Condition for a semilinear auto. to descend}, an element of  $ \Aut (G_0\to \Spec k_s\geq k) $ descends to an element of $ \Aut (G\to \Spec k) $ if and only if it is Galois invariant. This is the origin of Definition~\ref{Def:semilinear Galois action}.
\end{remark}

It is important to notice that in general, the $ \Gamma $-action on $ \Aut (G_0\to \Spec k_s\geq k) $ does not preserve the group structure. Let us prove some elementary properties of this action.
\begin{lemma}\label{Lem:computing with the semilinear action}
Keep the notations of Definition~\ref{Def:semilinear Galois action}. Let $ \gamma ,\gamma_1,\gamma_2 \in \Gamma $ and let $ b\Id_{\beta}$, $b_1\Id_{\beta_1}$, $ b_2\Id_{\beta_2} \in \Aut (G_0\to \Spec k_s\geq k)\cong \Aut G_0\rtimes \Aut (k_s\geq k) $.
\begin{enumerate}
\item $\gamma_1\gamma_2 . (b\Id_{\beta}) = \gamma_1.(\gamma_2.(b\Id_{\beta})) $
\item $ \gamma .(b_1\Id_{\beta_1} b_2\Id_{\beta_2}) = \Big( \beta_2^{-1}\gamma \beta_2.(b_1\Id_{\beta_1})\Big) \Big( \gamma.(b_2\Id_{\beta_2})\Big) $
\end{enumerate}
\begin{proof}
\begin{enumerate}
\item $\begin{aligned}[t]
\gamma_1\gamma_2 . (b\Id_{\beta}) 
&= c_{\beta^{-1}\gamma_1\gamma_2 \beta}\,^{\beta^{-1}\gamma_1\gamma_2 \beta}b\,^{\beta^{-1}}c_{\gamma_1\gamma_2}^{-1}\Id _{\beta}\\
&= c_{\beta^{-1}\gamma_1\beta \beta^{-1} \gamma_2 \beta}\,^{\beta^{-1}\gamma_1\gamma_2 \beta}b\,^{\beta^{-1}}(c_{\gamma_1}\,^{\gamma_1}c_{\gamma_2})^{-1}\Id _{\beta}\\
&= c_{\beta^{-1}\gamma_1\beta}\,^{\beta^{-1} \gamma_1 \beta}(c_{\beta^{-1} \gamma_2 \beta}\,^{\beta^{-1} \gamma_2 \beta}b\,^{\beta^{-1}}c_{\gamma_2}^{-1})\,^{\beta^{-1}}c_{\gamma_1}^{-1}\Id _{\beta}\\
&= \gamma_1.(c_{\beta^{-1} \gamma_2 \beta}\,^{\beta^{-1} \gamma_2 \beta}b\,^{\beta^{-1}}c_{\gamma_2}^{-1}\Id_{\beta})\\
&= \gamma_1.(\gamma_2.(b\Id_{\beta}))
\end{aligned}$
\item $\begin{aligned}[t]
&\gamma .(b_1\Id_{\beta_1} b_2\Id_{\beta_2}) = \gamma .(b_1 \,^{\beta_1^{-1}}b_2 \Id_{\beta_2\beta_1})\\
&= c_{(\beta_2\beta_1)^{-1}\gamma \beta_2\beta_1}\,^{(\beta_2\beta_1)^{-1}\gamma \beta_2\beta_1}(b_1 \,^{\beta_1^{-1}}b_2)\,^{(\beta_2\beta_1)^{-1}}c_{\gamma}^{-1}\Id _{\beta_2\beta_1}\\
&= c_{\beta_1^{-1}\beta_2^{-1}\gamma \beta_2\beta_1}\,^{\beta_1^{-1}\beta_2^{-1}\gamma \beta_2\beta_1}b_1\,^{\beta_1^{-1}}c_{\beta_2^{-1}\gamma \beta_2}^{-1}\,^{\beta_1^{-1}}c_{\beta_2^{-1}\gamma \beta_2}       \,^{(\beta_2\beta_1)^{-1}\gamma \beta_2}b_2\,^{(\beta_2\beta_1)^{-1}}c_{\gamma}^{-1}\Id _{\beta_2 \beta_1}\\
&=\Big( \beta_2^{-1}\gamma \beta_2 . (b_1\Id_{\beta_1})\Big) \Id_{\beta_1}^{-1}\,^{\beta_1^{-1}}c_{\beta_2^{-1}\gamma \beta_2}       \,^{\beta_1^{-1}\beta_2^{-1}\gamma \beta_2}b_2\,^{\beta_1^{-1}\beta_2^{-1}}c_{\gamma}^{-1}\Id _{\beta_1}\Id_{\beta_2}\\
&= \Big( \beta_2^{-1}\gamma \beta_2.(b_1\Id_{\beta_1})\Big) \Big( \gamma.(b_2\Id_{\beta_2})\Big)
\end{aligned}$
\end{enumerate}
\end{proof}
\end{lemma}

\begin{lemma}
Keep the notations of Definition~\ref{Def:semilinear Galois action}. The set of elements in $ \Aut (G_0\to \Spec k_s\geq k) $ that are fixed by the $ \Gamma $ action is a subgroup.
\end{lemma}
\begin{proof}
Let $b_1\Id_{\beta_1}$, $ b_2\Id_{\beta_2} \in \Aut (G_0\to \Spec k_s\geq k)\cong \Aut G_0\rtimes \Aut (k_s\geq k) $ be elements that are fixed by the $ \Gamma $ action. For $ \gamma \in \Gamma $, we have $ \gamma . (b_1\Id_{\beta_1} b_2\Id_{\beta_2}) = \Big( \beta_2^{-1}\gamma \beta_2.(b_1\Id_{\beta_1})\Big) \Big( \gamma.(b_2\Id_{\beta_2})\Big) $ by Lemma~\ref{Lem:computing with the semilinear action}. Hence $ b_1\Id_{\beta_1} b_2\Id_{\beta_2} $ is $ \Gamma $ invariant as well.

Similarly, if $ b\Id_{\beta}$ is $\Gamma $ invariant, for all $ \gamma \in \Gamma $ we have $ \Id_{G_0} = \gamma . \Id_{G_0} = \gamma .((b\Id_{\beta})^{-1}b\Id_{\beta}) = \Big( \beta^{-1}\gamma \beta .((b\Id_{\beta})^{-1})\Big) \Big( \gamma.(b\Id_{\beta})\Big)  $. Hence, since $ b\Id_{\beta} $ is $ \Gamma $ invariant, we get $ \beta^{-1}\gamma \beta .((b\Id_{\beta})^{-1}) = (b\Id_{\beta})^{-1} $ for all $ \gamma \in \Gamma $, and hence $ (b\Id_{\beta})^{-1} $ is $ \Gamma $ invariant as well.
\end{proof}
\begin{definition}
In the notations of Definition~\ref{Def:semilinear Galois action}, the subgroup of elements of  $ \Aut (G_0\to \Spec k_s\geq k) $ that are fixed by $ \Gamma $ is denoted $ \Aut (G_0\to \Spec k_s\geq k)^{\Gamma} $.
\end{definition}

We now aim to state that the group $ \Aut (G\to \Spec k) $ is the group $ \Aut (G_0\to \Spec k_s\geq k)^{\Gamma} $ modulo the Galois group. So we need to embed the Galois group as a normal subgroup of $ \Aut (G_0\to \Spec k_s\geq k)^{\Gamma} $.
\begin{definition}\label{Def:embedding Gamma in semilinear automorphisms}
Consider the homomorphism $ \Gamma \to \Aut (G_0\to \Spec k_s\geq k)^{\Gamma}\colon \gamma \mapsto c_{\gamma}\Id_{\gamma}^{-1} $. We denote the image of $ \Gamma $ by $ \tilde{\Gamma} $.
\end{definition}
\begin{remark}
Note that for $ \gamma \in \Gamma $, $ c_{\gamma}\Id_{\gamma}^{-1} $ is an invariant element of $ \Aut (G_0\to \Spec k_s\geq k) $. Indeed, for $ \delta \in \Gamma $ we have $ \delta .(c_{\gamma}\Id_{\gamma}^{-1}) = c_{\gamma \delta \gamma^{-1}}\,^{\gamma \delta \gamma^{-1}}c_{\gamma}\,^{\gamma}c_{\delta}^{-1}\Id_{\gamma}^{-1} = c_{\gamma \delta }\,^{\gamma \delta}c_{\delta^{-1}}\Id_{\gamma}^{-1} = c_{\gamma}\Id_{\gamma}^{-1} $.
\end{remark}
\begin{remark}
If we denote $ \tilde{\gamma} = c_{\gamma}\Id_{\gamma}^{-1} $, it is unfortunate that $ \tilde{\gamma} $ is the inverse of the element $ \tilde{\gamma}_G = c_{\gamma^{-1}}\Id_{\gamma} $ appearing in Definition~\ref{Def:gammatilde and Idgamma}. In the language of descent datum, $ \gamma \mapsto \tilde{\gamma}_G $ is traditionally required to be an anti-homomorphism, whereas it felt more natural to use a homomorphism in Definition~\ref{Def:embedding Gamma in semilinear automorphisms}, so we indulge in this inconsistency.
\end{remark}
\begin{lemma}\label{Lem:Identification of Aut(G Spec k)}
Keeping the notations of Definition~\ref{Def:semilinear Galois action}, $ \Aut (G\to \Spec k) $ is naturally isomorphic to $ \Aut (G_0\to \Spec k_s\geq k)^{\Gamma}/\tilde{\Gamma} $.
\end{lemma}
\begin{proof}
We have a homomorphism $ \Aut (G_0\to \Spec k_s\geq k)^{\Gamma}\to \Aut (G\to \Spec k) $, which maps an invariant element of $ \Aut (G_0\to \Spec k_s\geq k) $ to its descent in $ \Aut (G\to \Spec k) $ (see Lemma~\ref{Lem:Condition for a semilinear auto. to descend}). Note that $ \tilde{\Gamma} $ is in the kernel of this map (since $ c_{\gamma^{-1}}\Id_{\gamma} = \tilde{\gamma}_{G} $ arises as a choice of base change for the identity). Now if $ b\Id_\beta $ is in the kernel of this homomorphism, then $ \beta $ acts trivially on $ k $, i.e.\ $ \beta = \gamma $ for some $ \gamma \in \Gamma $, and $ b\Id_{\beta}c_{\gamma}\Id_{\gamma}^{-1}\in \Aut G_0 $ descends to the identity in $ \Aut G $. But this holds if and only if $  b\Id_{\beta}c_{\gamma}\Id_{\gamma}^{-1} $ is already the identity on $ G_0 $. Since $ b\Id_{\beta}c_{\gamma}\Id_{\gamma}^{-1} = b\,^{\gamma^{-1}}c_{\gamma} = bc_{\gamma^{-1}}^{-1} $, we conclude that $ b\Id_{\beta} $ descends to the identity if and only if it is equal to $ c_{\gamma^{-1}}\Id_{\gamma} $ for some $ \gamma \in \Gamma $, i.e.\ if and only if it belongs to $ \tilde{\Gamma} $.

It remains to check that the homomorphism $ \Aut (G_0\to \Spec k_s\geq k)^{\Gamma}\to \Aut (G\to \Spec k) $ is surjective as well. But this follows from the fact that any automorphism of $ k $ can be extended to an automorphism of $ k_s $.
\end{proof}

Note that there was nothing special about the category of algebraic $k$-groups, and we could as well repeat this construction for other algebraic categories over $ k $ for which descent is effective. In particular, we can repeat everything we did so far for $ k $-schemes of based root datum. Also recall that in Theorem~\ref{Thm:Quasi-split groups and BRD}, still keeping the notations of Definition~\ref{Def:semilinear Galois action}, the cocycle defining $ \mathcal{R}(G) $ is obtained from $ c $ by projecting via $ \Aut G_0\to E $. Recalling that the Galois action on the split $ k_s $-scheme of based root datum is trivial, this gives the following result.
\begin{lemma}\label{Lem:semilinear Galois cohomology for root datum}
Keep the notations of Definition~\ref{Def:semilinear Galois action}. Let $ \tilde{c} $ be the projection of $ c $ under $ \Aut G_0\to E $. Define a semilinear Galois action on $ \Aut (R_{k_s}\to \Spec k_{s}\geq k)\cong E\times \Aut (k_s\geq k) $ as follows:
\begin{equation*}
\text{ For all } b\in E,~\beta \in \Aut (k_s\geq k) \text{ and } \gamma \in \Gamma,~~~\gamma .(b
\Id_{\beta}) = \tilde{c}_{\beta^{-1}\gamma \beta}b\tilde{c}_{\gamma}^{-1}\Id _{\beta}. 
\end{equation*}
Define $ \tilde{\Gamma}\leq \Aut (R_{k_s}\to \Spec k_{s}\geq k)^{\Gamma} $ to be the image of the 
homomorphism $ \Gamma \to \Aut (R_{k_s}\to \Spec k_{s}) \colon \gamma \mapsto \tilde{c}
_{\gamma}\Id_{\gamma}^{-1} $. Then $ \Aut (\mathcal{R}(G)\to \Spec k) $ is naturally isomorphic to 
$ \Aut (R_{k_s}\to \Spec k_{s}\geq k)^{\Gamma}/\tilde{\Gamma} $. Furthermore, the homomorphism 
$ \Aut (G_0\to \Spec k_s)\to \Aut (R_{k_s}\to \Spec k_s) $ induces a homomorphism $  \Aut 
(G_0\to \Spec k_s\geq k)^{\Gamma}/\tilde{\Gamma}\to  \Aut (R_{k_s}\to \Spec k_{s}\geq 
k)^{\Gamma}/\tilde{\Gamma} $.
\end{lemma}

We can now formulate the failure of surjectivity of the map $ \Aut (G\to \Spec k)\to \Aut_G (k) $ and of the map $ \Aut (G\to \Spec k)\to \Aut (\mathcal{R}(G)\to \Spec k ) $ using a variant of Galois cohomology. We first give an approximation of this.
\begin{proposition}\label{Prop:Galois cohomology for semilinear auto}
Keep the notations of Lemma~\ref{Lem:semilinear Galois cohomology for root datum}. Endow $ \Aut G_0 $ and $ (\Ad G_0) (k_s) $ with the Galois action given by restricting the semilinear Galois action on $ \Aut (G_0\to \Spec k_s\geq k) $ $ ( $this is just the Galois action corresponding to the form $ G_0\cong G_{k_s}) $.
\begin{enumerate}
\item There exists a coboundary map $ \Aut (k_{s}\geq k)\xrightarrow{\partial }  H^1(\Gamma , \Aut G_0) $ such that the sequence
\begin{equation*}
1\to (\Aut G_0)^{\Gamma}\to \Aut (G_0\to \Spec k_s\geq k)^{\Gamma}\to \Aut (k_{s}\geq k)\xrightarrow{\partial }  H^1(\Gamma , \Aut G_0)
\end{equation*}
is exact.
\item There exists a coboundary map $ \Aut (R_{k_s}\to \Spec k_{s}\geq k)^{\Gamma}\xrightarrow{\partial }  H^1(\Gamma , (\Ad G_0) (k_s)) $ such that the sequence
\begin{align*}
1\to (\Ad G_0)(k_s)^{\Gamma}\to \Aut (G_0\to &\Spec k_s\geq k)^{\Gamma}\to \\
&\Aut (R_{k_s}\to \Spec k_{s}\geq k)^{\Gamma}\xrightarrow{\partial }  H^1(\Gamma , (\Ad G_0) (k_s))
\end{align*}
is exact.
\end{enumerate} 
\end{proposition}
\begin{proof}
It is important not to confuse the two Galois actions we are considering on $ \Aut G_0 $. One arises from the fact that we chose a form of $ G_0 $ over the prime field of $ k $, and the other is the Galois action arising from $ G_0\cong G_{k_s} $. For $ \gamma \in \Gamma $ and $ b\in \Aut G_0 $, recall that the former is denoted $ ^{\gamma}b = \Id_{\gamma}^{-1}b\Id_{\gamma} $ whilst the latter is denoted $ \gamma .b = c_{\gamma}\,^{\gamma}bc_{\gamma}^{-1} $.

We have a similar remark for $ E = \Aut R_{k_s} $: for $ \gamma \in \Gamma $ and $ b\in E $, we denote $ ^{\gamma}b = \Id_{\gamma}^{-1}b\Id_{\gamma} = b $ (the latter equality holds because any automorphism of $ R_{k_s} $ is defined over $ k $) and $ \gamma .b = \tilde{c}_{\gamma}\,^{\gamma}b\tilde{c}_{\gamma}^{-1} = \tilde{c}_{\gamma}b\tilde{c}_{\gamma}^{-1} $.
\begin{enumerate}
\item Let $ \beta^{-1} \in \Aut (k_s\geq k) $. As usual in this situation, we set $ \partial (\beta^{-1} )\colon \Gamma \to \Aut G_0\colon \gamma \mapsto \partial (\beta^{-1} )_{\gamma} $ where $ \partial (\beta^{-1} )_{\gamma} $ is defined by the equality (in $\Aut (G_0\to \Spec k_s\geq k)$) $ \gamma .\Id_{\beta} = \Id_{\beta }\partial (\beta^{-1} )_{\gamma} $. In other words, $ \partial (\beta^{-1} )_{\gamma} = \Id_{\beta }^{-1}\gamma .\Id_{\beta} $. Hence $ \partial (\beta^{-1} )_{\gamma} $ clearly belongs to $ \Aut G_0 $. The fact that it is a cocycle follows directly from Lemma~\ref{Lem:computing with the semilinear action}. Indeed, $ \partial (\beta^{-1} )_{\gamma \gamma '} = \Id_{\beta }^{-1}\gamma \gamma ' .\Id_{\beta} = \Id_{\beta }^{-1}\gamma.(\Id_{\beta}\Id_{\beta}^{-1} \gamma ' .\Id_{\beta}) = \Id_{\beta }^{-1}\gamma.(\Id_{\beta})\gamma .(\Id_{\beta}^{-1} \gamma ' .\Id_{\beta}) = \partial (\beta^{-1} )_{\gamma} \gamma .\partial (\beta^{-1} )_{\gamma '} $.

The sequence $ 1\to \Aut G_0\to \Aut (G_0\to \Spec k_s\geq k)\to \Aut (k_{s}\geq k) $ is a $ \Gamma $-equivariant exact sequence (where we endow $ \Aut (k_{s}\geq k) $ with the trivial $ \Gamma $ action). Hence, taking $ \Gamma $-invariant elements, it remains exact. So we just need to check exactness at $ \Aut (k_{s}\geq k) $.

Let $ \beta^{-1} \in \Aut (k_s\geq k) $. Then $ \partial (\beta^{-1} ) $ is trivial in $  H^1(\Gamma , \Aut G_0) $ if and only if there exists $ b\in \Aut G_0 $ such that for all $ \gamma \in \Gamma $,  $ b^{-1}\Id_{\beta }^{-1}(\gamma .\Id_{\beta})(\gamma .b) = 1 $. By Lemma~\ref{Lem:computing with the semilinear action}, $ b^{-1}\Id_{\beta }^{-1}(\gamma .\Id_{\beta})(\gamma .b) = (\Id_{\beta }b)^{-1}\gamma .(\Id_{\beta }b) $, and hence $ \partial (\beta^{-1} ) $ is trivial if and only if there exists $b\in \Aut G_0 $ such that $ \Id_{\beta}b $ is a $ \Gamma $-invariant element of  $ \Aut (G_0\to \Spec k_s\geq k) $, which proves exactness at $ \Aut (k_s\geq k) $.

\item Let $ b\Id_{\beta} \in \Aut (R_{k_s}\to \Spec k_{s}\geq k)^{\Gamma}\cong (E\times \Aut (k_s\geq k))^{\Gamma} $. As usual in this situation, we set $ \partial (b\Id_{\beta} )\colon \Gamma \to \Aut G_0\colon \gamma \mapsto \partial (b\Id_{\beta} )_{\gamma} $ where $ \partial (b\Id_{\beta} )_{\gamma} $ is defined by the equality (in $\Aut (G_0\to \Spec k_s\geq k)$) $ \gamma .(b\Id_{\beta}) = b\Id_{\beta }\partial (b\Id_{\beta} )_{\gamma} $.  In other words, $ \partial (b\Id_{\beta} )_{\gamma} = (b\Id_{\beta })^{-1}\gamma .(b\Id_{\beta}) $. Let us check that $  \partial (b\Id_{\beta} )_{\gamma} $ belongs to $ (\Ad G_0) (k_s) $. For $ \gamma \in \Gamma $,  define $ c'_{\gamma} = c_{\gamma}\tilde{c}_{\gamma}^{-1} $. Now $ \partial (b\Id_{\beta} )_{\gamma} = (b\Id_{\beta })^{-1}c_{\beta^{-1}\gamma \beta}\,^{\beta^{-1}\gamma \beta}b\,^{\beta^{-1}}c_{\gamma}^{-1}\Id_{\beta} =$  $(b\Id_{\beta })^{-1}c'_{\beta^{-1}\gamma \beta}b\Id_{\beta}{c'}_{\gamma}^{-1} $ because $ b\Id_{\beta}\in \Aut (R_{k_s}\to \Spec k_{s}\geq k)^{\Gamma} $, hence we conclude that $  \partial (b\Id_{\beta} )_{\gamma} \in (\Ad G_0) (k_s) $ because $ c'_{\gamma }\in (\Ad G_0)(k_s) $ and $ (\Ad G_0)(k_s) $ is a normal subgroup of $ \Aut (G_0\to \Spec k_s\geq k) $.

The sequence $ 1\to (\Ad G_0)(k_s)\to \Aut (G_0\to \Spec k_s\geq k)\to \Aut (R_{k_s}\to \Spec k_{s}\geq k) $ is a $ \Gamma $-equivariant exact sequence. Hence, taking $ \Gamma $-invariant elements, it remains exact. So we just need to check exactness at $\Aut (R_{k_s}\to \Spec k_{s}\geq k)^{\Gamma} $.

Let $ b\id_{\beta} \in \Aut (R_{k_s}\to \Spec k_{s}\geq k)^{\Gamma}\cong (E\times \Aut (k_s\geq k))^{\Gamma} $. Then $ \partial ( b\id_{\beta}) $ is trivial in $  H^1(\Gamma , (\Ad G_0)(k_s)) $ if and only if there exists $ g\in (\Ad G_0)(k_s) $ such that for all $ \gamma \in \Gamma $,  $ g^{-1}(b\id_{\beta})^{-1}\gamma .(b\id_{\beta})\gamma .g = 1 $. But $ g^{-1}(b\id_{\beta})^{-1}\gamma .(b\id_{\beta})\gamma .g = (b\Id_{\beta }g)^{-1}\gamma .(b\Id_{\beta }g) $ by Lemma~\ref{Lem:computing with the semilinear action}, and hence $ \partial (b\id_{\beta}) $ is trivial if and only if there exists $g\in (\Ad G_0)(k_s) $ such that $ b\Id_{\beta }g $ is a $ \Gamma $-invariant element of  $ \Aut (G_0\to \Spec k_s\geq k) $, which proves exactness at $ \Aut (R_{k_s}\to \Spec k_{s}\geq k)^{\Gamma} $. \qedhere
\end{enumerate}
\end{proof}

In order to prove Theorem~\ref{Thm: MainThm3}, it suffices now to observe that in Proposition~\ref{Prop:Galois cohomology for semilinear auto}, the image of $ \tilde{\Gamma} $ under the coboundary operator is trivial.
\begin{lemma}\label{Lem:triviality of Gamma under cocoundary}
Keep the notation of Proposition~\ref{Prop:Galois cohomology for semilinear auto}.
\begin{enumerate}
\item The image of $ \Gamma \unlhd \Aut(k_s\geq k) $ under $ \Aut(k_s\geq k)\xrightarrow{\partial} H^1(\Gamma ,\Aut G_0) $ is trivial.
\item The image of $ \tilde{\Gamma}\unlhd \Aut (R_{k_s}\to \Spec k_s\geq k)^{\Gamma} $ under the coboundary map $ \Aut (R_{k_s}\to \Spec k_s\geq k)^{\Gamma}\xrightarrow{\partial} H^1(\Gamma ,(\Ad G_0)(k_s)) $ is trivial. 
\end{enumerate}
\end{lemma}
\begin{proof}
The proof is just a straightforward computation, using directly the definition of the semilinear $ \Gamma $-action.
\begin{enumerate}
\item Let $ \gamma \in \Gamma $. We want to check that the cocycle $ \Gamma \to \Aut G_0\colon \sigma \mapsto \Id_{\gamma}^{-1}\sigma .\Id_{\gamma} $ is trivial. By definition, $ \Id_{\gamma}^{-1}\sigma .\Id_{\gamma} = \Id_{\gamma}^{-1} c_{\gamma^{-1}\sigma\gamma}\,^{\gamma^{-1}}c_{\sigma}^{-1}\Id_{\gamma} = \,^{\gamma}c_{\gamma^{-1}\sigma\gamma}c_{\sigma}^{-1} = \,^{\gamma}c_{\gamma^{-1}}c_{\sigma}\,^{\sigma}c_{\gamma}c_{\sigma}^{-1} = c_{\gamma}^{-1}\sigma.c_{\gamma} $. Since $ c_{\gamma} $ belongs to $ \Aut G_0 $, this indeed shows that $ \partial (\gamma^{-1}) $ is cohomologous to the trivial cocycle.

\item Let $ \tilde{c}_{\gamma}\Id_{\gamma}^{-1} \in \tilde{\Gamma} \unlhd \Aut (R_{k_s}\to \Spec k_{s}\geq k)^{\Gamma}\cong (E\times \Aut (k_{s}\geq k))^{\Gamma} $. By definition, $ \partial (\tilde{c}_{\gamma}\Id_{\gamma}^{-1})_{\sigma } = (\tilde{c}_{\gamma}\Id_{\gamma}^{-1})^{-1}\sigma .(\tilde{c}_{\gamma}\Id_{\gamma}^{-1}) $. Plugging the definition of the semilinear action, we get
\begin{align*}
(\tilde{c}_{\gamma}\Id_{\gamma}^{-1})^{-1}\sigma .(\tilde{c}_{\gamma}\Id_{\gamma}^{-1}) &= \Id_{\gamma}\tilde{c}_{\gamma}^{-1}c_{\gamma \sigma\gamma^{-1}}\,^{\gamma \sigma\gamma^{-1}}\tilde{c}_{\gamma}\,^{\gamma}c_{\sigma}^{-1}\Id_{\gamma}^{-1}\\
& = \,^{\gamma^{-1}}(\tilde{c}_{\gamma}^{-1}c_{\gamma \sigma\gamma^{-1}})\,^{\sigma\gamma^{-1}}\tilde{c}_{\gamma}c_{\sigma}^{-1}\\
& = \,^{\gamma^{-1}}(\tilde{c}_{\gamma}^{-1}c_{\gamma})c_{\sigma}\,^{\sigma}c_{\gamma^{-1}}\,^{\sigma\gamma^{-1}}\tilde{c}_{\gamma}c_{\sigma}^{-1}\\
& = \,^{\gamma^{-1}}\tilde{c}_{\gamma}^{-1}c_{\gamma^{-1}}^{-1}c_{\sigma}\,^{\sigma}(c_{\gamma^{-1}}\,^{\gamma^{-1}}\tilde{c}_{\gamma})c_{\sigma}^{-1}\\
& = (c_{\gamma^{-1}}\,^{\gamma^{-1}}\tilde{c}_{\gamma})^{-1} \sigma .(c_{\gamma^{-1}}\,^{\gamma^{-1}}\tilde{c}_{\gamma})\\
& = (c_{\gamma^{-1}} \tilde{c}_{\gamma})^{-1} \sigma .(c_{\gamma^{-1}} \tilde{c}_{\gamma})
\end{align*}
where the last equality holds because the Galois action on $ E = \Aut R_{k_s} $ is trivial. Since $ c_{\gamma^{-1}} \tilde{c}_{\gamma} $ belongs to $ (\Ad G_0)(k_s) $, this indeed shows that $ \partial (\tilde{c}_{\gamma}\Id_{\gamma}^{-1}) $ is cohomologous to the trivial cocycle. \qedhere
\end{enumerate}
\end{proof}

\begin{corollary}
Keep the notations of Proposition~\ref{Prop:Galois cohomology for semilinear auto}.
\begin{enumerate}
\item The exact sequence 
\begin{equation*}
1\to (\Aut G_0)^{\Gamma}\to \Aut (G_0\to \Spec k_s\geq k)^{\Gamma}\to \Aut (k_{s}\geq k)\xrightarrow{\partial }  H^1(\Gamma , \Aut G_0)
\end{equation*}
induces an exact sequence
\begin{equation*}
1\to (\Aut G_0)^{\Gamma}\to \Aut (G_0\to \Spec k_s\geq k)^{\Gamma}/\tilde{\Gamma} \to \Aut (k_{s}\geq k)/\Gamma \xrightarrow{\partial }  H^1(\Gamma , \Aut G_0)
\end{equation*}
\item The exact sequence 
\begin{align*}
1\to (\Ad G_0)(k_s)^{\Gamma}\to \Aut (G_0\to &\Spec k_s\geq k)^{\Gamma}\to \\
&\Aut (R_{k_s}\to \Spec k_{s}\geq k)^{\Gamma}\xrightarrow{\partial }  H^1(\Gamma , (\Ad G_0) (k_s))
\end{align*}
induces an exact sequence
\begin{align*}
1\to (\Ad G_0)(k_s)^{\Gamma}\to \Aut (G_0\to &\Spec k_s\geq k)^{\Gamma}/\tilde{\Gamma}\to \\ &\Aut (R_{k_s}\to \Spec k_{s}\geq k)^{\Gamma}/\tilde{\Gamma}\xrightarrow{\partial }  H^1(\Gamma , (\Ad G_0) (k_s))
\end{align*}
\end{enumerate}
\end{corollary}
\begin{proof}
Note that in both cases, moding out by $ \tilde{\Gamma} $ does not modify exactness on the left. Hence the results follows directly from Lemma~\ref{Lem:triviality of Gamma under cocoundary}.
\end{proof}

\begin{proof}[Proof of Theorem~\ref{Thm: MainThm3}]
Note that by Lemma~\ref{Lem:Identification of Aut(G Spec k)}, $ \Aut (G_0\to \Spec k_s\geq k)^{\Gamma}/\tilde{\Gamma} $ is naturally isomorphic to $ \Aut (G\to \Spec k) $. Similarly, by Lemma~\ref{Lem:semilinear Galois cohomology for root datum}, $ \Aut (R_{k_s}\to \Spec k_{s}\geq k)^{\Gamma}/\tilde{\Gamma}\cong \Aut (\mathcal{R}(G)\to \Spec k) $. Also note that the restriction of the semilinear $ \Gamma $-action on $ \Aut (G_0\to \Spec k_s\geq k)^{\Gamma}/\tilde{\Gamma} $ to $ \Aut G_0 $ (respectively $ (\Ad G_0)(k_s) $) is the natural Galois action on $ \Aut G_{k_s} $ (respectively $ (\Ad G)(k_s) $). In particular, the $ \Gamma $ invariant elements are the elements of $ \Aut G $ (respectively $ (\Ad G)(k) $). Finally, noting that $ \Aut (k) \cong \Aut (k_s\geq k)/\Gamma $ and that all those identifications are natural enough, we get the result.
\end{proof}

We now describe how the coboundary map of the exact sequence $ 1\to \Aut G\to \Aut (G\to \Spec k)\to \Aut (k)\to H^1(k,\Aut G_{k_s}) $ can be used to compute $ \Aut_G(k) $. To illustrate this, we set the following notations for the rest of the section:
\begin{definition}\label{Def:setting notation for explicit computation with division algebras}
\begin{enumerate}
\item $ D $ denotes a central division algebra of degree $ 3 $ over $ k $ (hence by a theorem of Wedderburn, $ D $ is cyclic). We fix a maximal Galois subfield $ l $ of $ D $ so that $ \Gal (l/k) $ is cyclic of order $ 3 $ (which exists because $ D $ is cyclic). We choose a generator of $ \Gal (l/k) $ that we denote $ \gamma $. Choosing an element $ u\in D $ normalising $ l $ and such that its action by conjugation on $ l $ generates $ \Gal (l/k) $, we set $ a = u^3\in k $. We set $ G := \textbf{SL}_1(D) $ to be the corresponding algebraic $ k $-group.
\item Set $ G_0 := \SL_3 $ (that we consider over $ k_s $, as in the beginning of this section). Recall that $ \Ad \SL_3 = \PGL_3 $. We denote elements of $ \PGL_3(k_s) $ as $ \begin{bsmallmatrix}
g_{11} & g_{12} & g_{13} \\
g_{21} & g_{22} & g_{23} \\
g_{31} & g_{32} & g_{33}
\end{bsmallmatrix} $, which is to be read as ``the equivalence class corresponding to the matrix $ \begin{psmallmatrix}
g_{11} & g_{12} & g_{13} \\
g_{21} & g_{22} & g_{23} \\
g_{31} & g_{32} & g_{33}
\end{psmallmatrix}\in GL_3(k_s) $''.

\item We choose the usual pinning of $ \SL_3 $ where the pair $ (T,B) $ consists of diagonal matrices and of upper triangular matrices, and where we choose some generators of the corresponding ``basic root groups''. Let $ R $ be the corresponding based root datum. Note that $ \Aut R $ is of order $ 2 $, and that if our choice of generators for the ``basic root groups'' is sensible enough, the splitting of $ \Aut \SL_3\to \Aut R $ is given by the automorphism $ \SL_3\to \SL_3\colon g\mapsto \,^{\text{at}}g^{-1} $, where $ ^{\text{at}}g $ denotes the anti-transposed of $ g $, i.e.\ ``the transposed of $ g $ along the anti-diagonal''. More formally, for $ i,j\in  \lbrace 1, 2, 3\rbrace $, $ (^{\text{at}}g)_{ij} = g_{4-j;4-i} $. Note that taking anti-transpose commutes with taking inverse, so that there is no ambiguity in the notation $ ^{\text{at}}g^{-1} $.

\item Consider the homomorphism $ f\colon \Gal (l/k)\to PGL_3(l)\colon \gamma \mapsto \begin{bsmallmatrix}
0 & 0 & a \\
1 & 0 & 0 \\
0 & 1 & 0
\end{bsmallmatrix} $ (where $ a\in k $ and $ \gamma \in \Gal (l/k) $ have been defined in the first item of these definitions). We choose the cocycle $ c\colon \Gamma \to \Aut G_0 $ defining $ G = \textbf{SL}_1(D) $ over $ k $ to be the composition $ \Gamma \to \Gal (l/k)\xrightarrow{f}\PGL_3(l)\to \Aut G_0 $.
\end{enumerate}
\end{definition}

Having set those notations, we are ready to start computing. The following lemmas are two special cases of \cite{Ha07} that we recover without using any theory of division algebras.
\begin{lemma}\label{Lem:Hanke result part 1}
Keep the notations of Definition~\ref{Def:setting notation for explicit computation with division algebras} and let $ \beta \in \Aut (k_s\geq k) $ be such that $ \beta (l) = l $. Let $ \alpha $ be the restriction of $ \beta $ to $ k $. Then $ \alpha \in \Aut_G(k) $ if and only if there exists $ \lambda \in l^* $ such that either $ \frac{\alpha(a)}{a} = \,^{\gamma^2}\lambda \,^{\gamma}\lambda \lambda $ or $ \alpha (a)a = \,^{\gamma^2}\lambda \,^{\gamma}\lambda \lambda $.
\end{lemma}
\begin{proof}
Note that $ \alpha \in \Aut_G(k) $ if and only if $ \alpha^{-1}\in \Aut_G(k) $. Using the coboundary map of Proposition~\ref{Prop:Galois cohomology for semilinear auto}, for all $ \sigma \in \Gamma $ we have $ \partial (\beta^{-1})_{\sigma} = \Id_{\beta}^{-1}\sigma .\Id_{\beta} = \Id_{\beta}^{-1} c_{\beta^{-1}\sigma \beta}\,^{\beta^{-1}}c_{\sigma}^{-1} \Id_{\beta} = \,^{\beta}c_{\beta^{-1}\sigma \beta}c_{\sigma}^{-1} $. Hence by Proposition~\ref{Prop:Galois cohomology for semilinear auto}, $ \alpha \in \Aut_G(k) $ if and only if the cocycle $ \sigma \mapsto \,^{\beta}c_{\beta^{-1}\sigma \beta}c_{\sigma}^{-1} $ is trivial in $ H^1(k,\Aut SL_3) $, i.e.\ if and only if there exists $ g \in \PGL_3(k_s) $ and $ e\in \Aut R\leq \Aut \SL_3 $ such that $  (ge)^{-1}\,^{\beta}c_{\beta^{-1}\sigma \beta}c_{\sigma}^{-1} \sigma .(ge) = 1 $. But since $ \gamma $ generates $ \Gal (l/k) $ and since $ c_{\sigma}=1 $ for all $ \sigma $ acting trivially on $ l $, this is equivalent to the existence of $ g \in \PGL_3(l) $ and $ e\in \Aut R\leq \Aut \SL_3 $ such that $  (ge)^{-1}\,^{\beta}c_{\beta^{-1}\gamma \beta}c_{\gamma}^{-1} \gamma .(ge) = 1 $. Recalling that $ \gamma .(ge) = c_{\gamma }\,^{\gamma}(ge)c_{\gamma }^{-1} = c_{\gamma }\,^{\gamma}g\,ec_{\gamma }^{-1} $, this is equivalent to $ \,^{\beta}c_{\beta^{-1}\gamma \beta}\,^{\gamma}g\,ec_{\gamma }^{-1}e^{-1} = g $. Also note that $ ec_{\gamma }^{-1}e^{-1} = \begin{cases}
c_{\gamma} \text{ if } e \text{ is non-trivial}\\
c_{\gamma }^{-1} \text{ otherwise}\end{cases} $.

First assume that $ \beta^{-1}\gamma \beta = \gamma $. Then $ \partial (\beta^{-1}) $ is trivial if and only if there exists $ g\in \PGL_3(l) $ such that either $ \,^{\beta}c_{\gamma}\,^{\gamma}gc_{\gamma}^{-1}  = g $ or $ \,^{\beta}c_{\gamma}\,^{\gamma}gc_{\gamma}  = g $. Letting $ g = \begin{bsmallmatrix}
g_{11} & g_{12} & g_{13} \\
g_{21} & g_{22} & g_{23} \\
g_{31} & g_{32} & g_{33}
\end{bsmallmatrix} \in \PGL_3(l) $, we have $ \,^{\beta}c_{\gamma}\,^{\gamma}gc_{\gamma}^{-1} = \begin{bsmallmatrix}
\,^{\gamma}g_{33}\frac{\beta (a)}{a}  & \,^{\gamma}g_{31}\beta (a) & \,^{\gamma}g_{32}\beta (a) \\
\,^{\gamma}g_{13}a^{-1} & \,^{\gamma}g_{11} & \,^{\gamma}g_{12} \\
\,^{\gamma}g_{23}a^{-1} & \,^{\gamma}g_{21} & \,^{\gamma}g_{22}
\end{bsmallmatrix} $. Since $ g $ is invertible, one of $ g_{13} $, $ g_{23} $ and $ g_{33} $ is non-zero. Let us for example assume that $ g_{33}\neq 0 $. Now $ \,^{\beta}c_{\gamma}\,^{\gamma}gc_{\gamma}^{-1}  = g $ if and only if there exists $ \lambda \in l^* $ such that $ \,^{\gamma}g_{33}\frac{\beta (a)}{a} = g_{11}\lambda $, $ \,^{\gamma}g_{11} = g_{22}\lambda $ and $ \,^{\gamma}g_{22} = g_{33}\lambda $. Hence $ \dfrac{^{\gamma}(\,^{\gamma}g_{33}\frac{\beta (a)}{a})}{^{\gamma}\lambda} = g_{22}\lambda $, and furthermore $ \dfrac{^{\gamma^2}(\,^{\gamma}g_{33}\frac{\beta (a)}{a})}{^{\gamma^2}\lambda\,^{\gamma}\lambda} = g_{33}\lambda $. Since the computation is similar if $ g_{13}\neq 0 $ or $ g_{23}\neq 0 $ instead, we conclude that if such a $ g $ exists, then there exists $ \lambda \in l $ such that $ \frac{\beta (a)}{a} = \,^{\gamma^2}\lambda\,^{\gamma}\lambda \lambda $ (because $ \gamma $ acts trivially on $ a\in k $ and $ \gamma^3 $ acts trivially on $ g_{33}\in l $). Conversely, if there exists $ \lambda \in l^* $ such that $ \frac{\beta (a)}{a} = \,^{\gamma^2}\lambda\,^{\gamma}\lambda \lambda $, then $ g = \begin{bsmallmatrix}
\,^{\gamma^2}\lambda \,^{\gamma}\lambda &0 & 0 \\
0 & \,^{\gamma^2}\lambda & 0 \\
0 &0 & 1
\end{bsmallmatrix} $ is such that $ \,^{\beta}c_{\gamma}\,^{\gamma}gc_{\gamma}^{-1}  = g $. A similar computation shows that there exists $ g\in \PGL_3(l) $ such that $ \,^{\beta}c_{\gamma}\,^{\gamma}gc_{\gamma}  = g $ if and only if there exists $ \lambda \in l^* $ such that $ \beta (a)a = \,^{\gamma^2}\lambda\,^{\gamma}\lambda \lambda $.

Now assuming that $ \beta^{-1}\gamma \beta = \gamma^{-1} $, $ \partial (\beta^{-1}) $ is trivial if and only if there exists $ g\in \PGL_3(l) $ such that either $ \,^{\beta}c_{\gamma^{-1}}\,^{\gamma}gc_{\gamma}^{-1}  = g $ or $ \,^{\beta}c_{\gamma^{-1}}\,^{\gamma}gc_{\gamma}  = g $. Or equivalently if and only if there exists $ g\in \PGL_3(l) $ such that either $ \,^{\beta}c_{\gamma}\,^{\gamma^{-1}}gc_{\gamma}  = g $ or $ \,^{\beta}c_{\gamma}\,^{\gamma^{-1}}gc_{\gamma}^{-1}  = g $. Hence we get the same setting as for the case $ \beta^{-1}\gamma \beta = \gamma $ up to replacing $ ^{\gamma}g_{ij} $ with $ \,^{\gamma^{-1}}g_{ij} $, which leaves the conclusion unchanged.
\end{proof}

\begin{lemma}\label{Lem:Hanke result part 2}
Keep the notations of Definition~\ref{Def:setting notation for explicit computation with division algebras} and let $ \beta \in \Aut (k_s\geq k) $ be such that $ \beta (l) = l' \neq l $. Let $ \delta = \beta \gamma \beta^{-1} $ $ ( $a generator of $ \Gal (l'/k))$, let $ K $ be the compositum $ ll'\le k_s $ and let $ \Gal (K/k) = \Gal (l/k)\times \Gal (l'/k) $ be the corresponding decomposition of $ \Gal (K/k) $. Let $ \alpha $ be the restriction of $ \beta $ to $ k $. Then $ \alpha \in \Aut_G(k) $ if and only if there exist $ \lambda, \mu \in K^* $ such that
\begin{enumerate}
\item either $ \,^{\gamma^2}\lambda\,^{\gamma}\lambda \lambda = \frac{1}{a} $, $ \,^{\delta^2}\mu\,^{\delta}\mu \mu = \beta (a) $ and $ \frac{^{\delta}\lambda}{\lambda} = \frac{^{\gamma}\mu}{\mu} $,
\item or $ \,^{\gamma^2}\lambda\,^{\gamma}\lambda \lambda = a $, $ \,^{\delta^2}\mu\,^{\delta}\mu \mu = \beta (a) $ and $ \frac{^{\delta}\lambda}{\lambda} = \frac{^{\gamma}\mu}{\mu} $.
\end{enumerate}

\end{lemma}
\begin{proof}
Arguing as in the beginning of the proof of Lemma~\ref{Lem:Hanke result part 1}, $ \alpha \in \Aut_G(k) $ if and only if there exists $ g \in \PGL_3(K) $ and $ e\in \Aut R\leq \Aut \SL_3 $ such that 
\begin{equation}\label{Eqn 1}
^{\beta}c_{\beta^{-1}\sigma \beta} \,^{\sigma}g ec_{\sigma}^{-1}e^{-1} = g \text{ for all } \sigma \in \Gal (K/k) 
\end{equation}

We first do the case $ e=1 $. Taking $ \sigma = \gamma $ in Equation~\ref{Eqn 1}, we get $ ^{\gamma}gc_{\gamma}^{-1} = g $. On the other hand, taking $ \sigma  = \delta $ in Equation~\ref{Eqn 1}, we get $ ^{\beta}c_{\gamma}\,^{\delta}g = g $. Hence Equation~\ref{Eqn 1} implies that 
\begin{equation*}
\begin{bsmallmatrix}
\frac{^{\gamma}g_{13}}{a} & ^{\gamma}g_{11} & ^{\gamma}g_{12} \\
\frac{^{\gamma}g_{23}}{a} & ^{\gamma}g_{21} & ^{\gamma}g_{22} \\
\frac{^{\gamma}g_{33}}{a} & ^{\gamma}g_{31} & ^{\gamma}g_{32}
\end{bsmallmatrix} = \begin{bsmallmatrix}
g_{11} & g_{12} & g_{13} \\
g_{21} & g_{22} & g_{23} \\
g_{31} & g_{32} & g_{33}
\end{bsmallmatrix} = \begin{bsmallmatrix}
^{\delta}g_{31}\beta (a) & ^{\delta}g_{32}\beta (a) & ^{\delta}g_{33}\beta (a) \\
^{\delta}g_{11} & ^{\delta}g_{12} & ^{\delta}g_{13} \\
^{\delta}g_{21}& ^{\delta}g_{22} & ^{\delta}g_{23}
\end{bsmallmatrix}
\end{equation*}

To fix ideas, assume $ g_{11}\neq 0 $ (the computation is similar if $ g_{21}\neq 0 $ or $ g_{31}\neq 0 $ instead). So we can further assume that $ g_{11} = 1 $. Hence there exists $ \lambda , \mu \in K^* $ such that
\begin{center}
 \begin{tikzpicture}[->]
  \node (1) at (-0.3,0) {$ \lambda^{-1}\begin{pmatrix}
\frac{^{\gamma}(^{\gamma}\lambda^{-1}\lambda^{-1})}{a} & 1 & ^{\gamma}\lambda^{-1} \\
\frac{^{\gamma}g_{23}}{a} & ^{\gamma}g_{21} & ^{\gamma}g_{22} \\
\frac{^{\gamma}g_{33}}{a} & ^{\gamma}g_{31} & ^{\gamma}g_{32}
\end{pmatrix} $};
  \node (2) at (7.5,0) {$ \mu^{-1}\begin{pmatrix}
^{\delta}(^{\delta}\mu^{-1}\mu^{-1})\beta (a) & ^{\delta}g_{32}\beta (a) & ^{\delta}g_{33}\beta (a) \\
1 & ^{\delta}g_{12} & ^{\delta}g_{13} \\
^{\delta}\mu^{-1}& ^{\delta}g_{22} & ^{\delta}g_{23}
\end{pmatrix} $} ;
  \node (3) at (3.5,-2) {$ \begin{pmatrix}
1 & \lambda^{-1} & ^{\gamma}\lambda^{-1}\lambda^{-1} \\
\mu^{-1} & g_{22} & g_{23} \\
^{\delta}\mu^{-1}\mu^{-1} & g_{32} & g_{33}
\end{pmatrix} $};
  \node (4) at (1.7,-1.1) [rotate = -45] {$ = $};
  \node (5) at (5.3,-1.1) [rotate = 45] {$ = $};
%  \node (4) [below=0.5cm of 2] {$ \Spec k $};
  
%  \path [every node/.style={font=\sffamily\small}]
%  (1) edge node [above]  {$ f_{\varphi} $} (2)
%        edge node  {} (3)
%  (2) edge node  {} (4)
%  (3) edge node [above] {\footnotesize{$ \Spec \varphi $}} (4);
  \end{tikzpicture}
  \end{center}

By looking at the $ 11 $ coefficient, this already implies that $ ^{\gamma^2}\lambda\,^{\gamma}\lambda \lambda = a^{-1} $ and $ ^{\delta^2}\mu\,^{\delta}\mu \mu = \beta (a) $. Finally, by looking at the central coefficient, we see that $ \lambda^{-1}\,^{\gamma}\mu^{-1} = \mu^{-1}\,^{\delta}\lambda^{-1} $. Conversely, if there exists $ \lambda ,\mu \in K^* $ such that $ \,^{\gamma^2}\lambda\,^{\gamma}\lambda \lambda = \frac{1}{a} $, $ \,^{\delta^2}\mu\,^{\delta}\mu \mu = \beta (a) $ and $ \frac{^{\delta}\lambda}{\lambda} = \frac{^{\gamma}\mu}{\mu} $, one can check that $ g = \begin{bmatrix}
1 & \lambda^{-1} & ^{\gamma}\lambda^{-1}\lambda^{-1} \\
\mu^{-1} & \lambda^{-1}\,^{\gamma}\mu^{-1} & \lambda^{-1}\,^{\gamma}(\lambda^{-1}\,^{\gamma}\mu^{-1}) \\
^{\delta}\mu^{-1}\mu^{-1} & \lambda^{-1}\,^{\gamma}(^{\delta}\mu^{-1}\mu^{-1}) & \lambda^{-1}\,^{\gamma} (\lambda^{-1}\,^{\gamma}(^{\delta}\mu^{-1}\mu^{-1}))
\end{bmatrix} $ satisfies Equation~\ref{Eqn 1}.

In the case where $ e $ is the non-trivial automorphism of $ R $, one uses the fact $ ec_{\gamma}e^{-1} = c_{\gamma}^{-1} $ and then imitates the above computation to get the second condition on $ \lambda ,\mu $ (to carry out this computation most easily, assume that $ g_{11} = 1 $ and for the last condition, consider the $ 23 $ coefficient).
\end{proof}
\begin{remark}\label{Rem:applicability of these computations}
The kind of computations we perform in Lemma~\ref{Lem:Hanke result part 1} can easily be adapted for any cyclic division algebra, and the computations in Lemma~\ref{Lem:Hanke result part 2} can easily be adapted to any cyclic division algebras of prime degrees. When the degree is not prime, $ l \cap \beta (l) $ might be a non-trivial extension of $ k $ (for $ l $ a maximal cyclic subfield of $ D $), and we did not try to overcome this complication using our methods. Note that in \cite{Ha07}*{Section~3}, T.\ Hanke deals with this extra difficulty very efficiently.
\end{remark}
\begin{remark}
In light of the results in \cite{Ha07}, we proved that $ \alpha \in \Aut_G(k) $ if and only if $ ^{\alpha}D\cong D $ or $ ^{\alpha}D\cong D^{\text{opp}}  $ (as $ k $-algebras). This is consistent with the fact that $ ^{\alpha }\textbf{SL}_1(D) \cong \textbf{SL}_1(\,^{\alpha}D) $, and that for two finite dimensional central division algebras $ D_1,D_2 $ over $ k $, $ \textbf{SL}_1(D_1)\cong \textbf{SL}_1(D_2) $ (as algebraic $ k $-groups) if and only if $ D_1\cong D_2 $ or $ D_1\cong D_2^{\text{opp}} $ (as $ k $-algebras).
\end{remark}

\section{Semilinear automorphisms of based root datum}\label{Sec:computing the group of semilinear automorphism of based root datum}
We aim to give an explicit description of the short exact sequence $ 1\to \Aut \mathcal{R}(G) \to \Aut (\mathcal{R}(G) \to \Spec k)\to \Aut_{\mathcal{R}(G)}(k)\to 1 $. We base this computation on Lemma~\ref{Lem:semilinear Galois cohomology for root datum}.

Recall that for $ R_{k_s} $ a split $ k $-scheme of based root datum of type $ R $, the $ \Gal (k_s/k) $-action on $ \Aut R_{k_s} $ is trivial, so that $ H^1(k_s/k, \Aut R_{k_s}) $ is isomorphic to the set of continuous homomorphisms $ \Hom (\Gal (k_s/k),\Aut R) $ up to conjugation.

\begin{definition}
Let $ \mathcal{R} $ be a $ k $-scheme of based root datum of type $ R $. 
\begin{enumerate}
\item Fix an isomorphism $ \mathcal{R}_{k_s}\cong R_{k_s} $ and let $ \tilde{c}\colon \Gal (k_s/k)\to \Aut R $ be the corresponding cocycle. Let $ N\unlhd \Gal (k_s/k) $ be the kernel of the homomorphism $ \tilde{c} $, and let $ l $ be the Galois extension of $ k $ fixed by $ N $. We call $ l $ \textbf{the classifying field of} $ \mathcal{R} $. Once a separable closure of $ k $ has been fixed, the classifying field of $ \mathcal{R} $ is uniquely determined by $ \mathcal{R} $.

\item We say that $ \mathcal{R} $ (or $ R $) is \textbf{semisimple} (respectively \textbf{simply connected}, respectively \textbf{adjoint}, respectively \textbf{simple}) if the split connected reductive group of type $ R $ is semisimple (respectively simply connected, respectively adjoint, respectively simple).
\end{enumerate}
\end{definition}
\begin{remark}
We use the following terminology: a connected reductive $ k $-group is simple if it is non-abelian and has no non-trivial connected closed normal subgroup (some author prefer to call such groups quasi-simple).
\end{remark}

\begin{lemma}\label{Lem:classification of Dynkin diagrams}
Let $R $ be a simple reduced based root datum and let $ k_s $ be a separable closure of $ k $. The map which associates to a $ k_s/k $-form of $ R_{k_s} $ its classifying field is a bijection between $ k $-schemes of based root datum of type $ R $ (up to $ k $-isomorphism) and subfields $ l\leq k_s $ such that $ l $ is Galois over $ k $ and $ \Gal (l/k) $ is isomorphic to a subgroup of $ \Aut R $.
\end{lemma}
\begin{proof}
Let $ D $ be the Dynkin diagram associated to $ R $. Since $ R $ is semisimple and reduced, $ \Aut R\leq \Aut D $. Furthermore, since $ R $ is simple $ \Aut D $ is either trivial, $ \Z/2\Z $ or $ S_3 $. Hence, if two subgroups of $ \Aut R $ are isomorphic, they are actually conjugate. The result follows from the fact that the Galois action on $ \Aut R $ is trivial, and hence $ H^1(k_s/k,\Aut R_{k_s}) $ is isomorphic to the set of continuous homomorphisms $ \Hom (\Gal (k_s/k),\Aut R) $ modulo conjugation.
\end{proof}
\begin{remark}
In the notations of the proof of Lemma~\ref{Lem:classification of Dynkin diagrams}, one might wonder when the inclusion $ \Aut R\leq \Aut D $ is an equality. This is always the case, except possibly when $ R $ is not simply connected or adjoint and is of type $ D_{2n} $. See \cite{Con14}*{Proposition~1.5.1} for a precise statement.
\end{remark}

\begin{definition}
Let $ \mathcal{R} $ be a simple $ k $-scheme of based root datum, and let $ k_s $ be a separable closure of $ k $. We define the \textbf{Tits index of} $ \mathcal{R} $ to be $ ^{g}X_{n,l} $ where 
\begin{enumerate}
\item $ l\leq k_s $ is the classifying field of $ \mathcal{R} $ (hence $ l $ is a finite Galois extension of $ k $).
\item $ X_n $ is the label of the Dynkin diagram associated to $ R $.
\item $ g $ is the order of the Galois group $ \Gal (l/k) $.
\end{enumerate}
\end{definition}

\begin{lemma}\label{Lem:descriptionofabstractautomoprhismsofDyn}
Let $ \mathcal{R} $ be a simple $ k $-scheme of based root datum of type $ R $ with index $ ^gX_{n,l} $.
\begin{enumerate}
\item $ g \in  \lbrace 1,2,3,6\rbrace $.
\item If $ g=1 $, $ \Aut (\mathcal{R} \to \Spec k)\cong \Aut R\times \Aut (k) $, and this isomorphism restricts to $ \Aut \mathcal{R} \cong \Aut R $.
\item If $g=2$ or $g=3$, $ \Aut (\mathcal{R} \to \Spec k)\cong  \Aut (l\geq k) $, and this isomorphism restricts to $ \Aut \mathcal{R}\cong \Gal (l/k) $.
\item If $g=6$, $ \Aut (\mathcal{R} \to \Spec k)\cong \Aut (l_3\geq k) $, where $ l_3 $ is any non-normal cubic subextension of $ l/k $. Furthermore, $ \Aut \mathcal{R} $ is trivial.
\end{enumerate}
\end{lemma}
\begin{proof}
\begin{enumerate}
\item Let $ D $ be the Dynkin diagram associated to $ R $. Since $ R $ is semisimple and reduced, $ \Aut R\leq \Aut D $. Furthermore, since $ R $ is simple $ \Aut D $ is either trivial, $ \Z/2\Z $ or $ S_3 $. It follows that $ g\in \lbrace 1,2,3,6\rbrace $.

\item The case $ g=1 $ means that $ \mathcal{R} $ is a split $ k $-scheme of based root datum. Hence, $ \Aut \mathcal{R}\cong \Aut R $ (because the functor of constant objects is fully faithfull). Furthermore the short exact sequence $ 1\to \Aut \mathcal{R}\to \Aut (\mathcal{R} \to \Spec k)\to \Aut (k)\to 1 $ splits. Also note that $ \Aut (k) $ acts trivially on $ \Aut R $, so that the result follows.

\item Recall that by Lemma~\ref{Lem:semilinear Galois cohomology for root datum}, $ \Aut (\mathcal{R} \to \Spec k)\cong (\Aut R\times \Aut (k_s\geq k))^{\Gamma}/\tilde{\Gamma} $, where $ \Aut R\times \Aut (k_s\geq k) $ is endowed with a semilinear $ \Gamma $-action arising from a choice of cocycle $ \tilde{c}\colon \Gamma \to \Aut R $ defining $ \mathcal{R} $. For $ \beta \in \Aut (l\geq k) $, let $ \tilde{\beta} $ denote an extension of $ \beta $ to an element of $ \Aut (k_s\geq k) $. Let also $ s\in \Aut R $ be an element of order $ 2 $ (note that the only case where the existence of such an element is not clear is in type $ ^3D_4 $, in which case it follows from Lemma~\ref{Lem:elements of order 2 in Aut R}). We define a map
\begin{align*}
\Phi \colon \Aut (l\geq k)&\to (\Aut R \times \Aut (k_s\geq k))^{\Gamma}/\tilde{\Gamma}\\
\beta &\mapsto \begin{cases} \Id_{\tilde{\beta}}^{-1} \text{ if } \beta\gamma \beta^{-1} = \gamma \text{ for all } \gamma \in \Gal (l/k).\\
s\Id_{\tilde{\beta}}^{-1} \text{ if } \beta \gamma \beta^{-1} \neq \gamma \text{ for } \gamma \text{ a generator of } \Gal (l/k).
\end{cases}
\end{align*}

We first check that $ \Phi $ is well-defined. If $ \tilde{\beta} $ and $ \tilde{\beta}' $ are two extensions of $ \beta $, we have to check that $ \Id_{ \tilde{\beta}'}\Id_{\tilde{\beta}}^{-1} $ belongs to $ \tilde{\Gamma} $. But $ (\tilde{\beta}')^{-1}\tilde{\beta} $ acts trivially on $ l $, hence the image of $ (\tilde{\beta}')^{-1}\tilde{\beta} $ under $ \Gamma \to (\Aut R \times \Aut (k_s\geq k))^{\Gamma}\colon \delta \mapsto \tilde{c}_{\delta}\Id_{\delta}^{-1} $ is indeed equal to $ \Id_{(\tilde{\beta}')^{-1}\tilde{\beta}}^{-1} = \Id_{ \tilde{\beta}'}\Id_{\tilde{\beta}}^{-1} $.

We now check that the image of $ \Phi $ is $ \Gamma $-invariant. Let $ \delta \mapsto \bar{\delta} $ denotes the projection $ \Aut (k_s/k)\to \Aut (l/k) $. When $ \beta \gamma \beta^{-1} = \gamma $ for all $ \gamma \in \Gal (l/k) $, $ \delta .\Id_{\tilde{\beta}}^{-1} = \tilde{c}_{\tilde{\beta} \delta \tilde{\beta}^{-1}}\tilde{c}_{\delta}^{-1}\Id_{\tilde{\beta}}^{-1} = \tilde{c}_{\beta \bar{\delta} \beta^{-1}}\tilde{c}_{\bar{\delta}}^{-1}\Id_{\tilde{\beta}}^{-1} = \tilde{c}_{\bar{\delta}}\tilde{c}_{\bar{\delta}}^{-1}\Id_{\tilde{\beta}}^{-1} = \Id_{\tilde{\beta}}^{-1}$ for all $ \delta \in \Gal (k_s/k) $. On the other hand, when $ \beta \gamma \beta^{-1} \neq \gamma $ for $ \gamma $ a generator of $ \Gal (l/k) $, we have $$ \delta .(s\Id_{\tilde{\beta}}^{-1}) = \tilde{c}_{\tilde{\beta} \delta \tilde{\beta}^{-1}}s\tilde{c}_{\delta}^{-1}\Id_{\tilde{\beta}}^{-1} = \tilde{c}_{\beta \bar{\delta} \beta^{-1}}s\tilde{c}_{\bar{\delta}}^{-1}\Id_{\tilde{\beta}}^{-1} = s\tilde{c}_{\bar{\delta}}\tilde{c}_{\bar{\delta}}^{-1}\Id_{\tilde{\beta}}^{-1} = s\Id_{\tilde{\beta}}^{-1}$$ for all $ \delta \in \Gal (k_s/k) $.

It is readily checked that $ \Phi $ is a homomorphism, so it remains to check that $ \Phi $ is bijective. If $ \Phi (\beta) $ is trivial, then $ \Id_{\tilde{\beta}}^{-1} $ or $ s \Id_{\tilde{\beta}}^{-1} $ belongs to $ \tilde{\Gamma} $, i.e.\ there exists $ \delta \in \Gamma $ such that either $  \Id_{\tilde{\beta}}^{-1} $ or $ s \Id_{\tilde{\beta}}^{-1}$ is equal to $ \tilde{c}_{\delta}\Id_{\delta}^{-1} $. This implies that $ \tilde{c}_{\delta} $ is trivial, so that $ \delta $ acts trivially on $ l $, and hence $ \beta $ is trivial. 

Finally, we check that $ \Phi $ is surjective. Let $ b\Id_{\tilde{\beta}} \in (\Aut R \times \Aut (k_s\geq k))^{\Gamma} $. We claim that $ \tilde{\beta} $ preserves $ l $. Indeed, for all $ \delta \in \Gamma $, $ \delta . (b\Id_{\tilde{\beta}}) = \tilde{c}_{\tilde{\beta}^{-1}\delta \tilde{\beta}}b\tilde{c}_{\delta}^{-1}\Id_{\tilde{\beta}} $. Hence $ b\Id_{\tilde{\beta}} $ is $ \Gamma $-invariant if and only if $ b\tilde{c}_{\delta} = \tilde{c}_{\tilde{\beta}^{-1} \delta \tilde{\beta}}b $ for all $ \delta \in \Gamma $. But if $ \tilde{\beta} $ does not preserve $ l $, there exists $ \delta \in \Gamma $ such that $ \tilde{c}_{\delta} = 1 \neq \tilde{c}_{\tilde{\beta}^{-1} \delta \tilde{\beta}}  $, a contradiction. Hence the claim is proved. To conclude, note that if $ b\Id_{\tilde{\beta}} $ is $ \Gamma $-invariant, $ \tilde{\beta} $ preserves $ l $ and hence up to an element in the image of $ \Phi $, we can assume that $ \tilde{\beta} $ acts trivially on $ l $. Hence, since $ b\Id_{\tilde{\beta}} $ is $ \Gamma $-invariant, either $ b $ is trivial, or $ b $ commutes with the image of $ \tilde{c} $, and hence belongs to the image of $ \tilde{c} $. In the first case, $ b\Id_{\tilde{\beta}} = \Id_{\tilde{\beta}} $ is in $ \tilde{\Gamma} $. In the second case, $ b\Id_{\tilde{\beta}} = \tilde{c}_{\gamma}\Id_{\tilde{\gamma}}^{-1}\Id_{\tilde{\gamma}}\Id_{\tilde{\beta}} $ for some $ \gamma \in \Gal (l/k) $. But this is in the image of $\Phi$ because $ \tilde{c}_{\tilde{\gamma}}\Id_{\tilde{\gamma}}^{-1} $ and $ \Id_{\tilde{\beta}} $ are in $ \tilde{\Gamma} $, whilst $ \Id_{\tilde{\gamma}} = \Phi (\gamma^{-1}) $.

For the last statement, note that under the isomorphism $ \Aut (l\geq k)\cong \Aut (\mathcal{R}\to \Spec k) $, the algebraic automorphisms are the one acting trivially on $ k $, i.e.\ we have $ \Aut \mathcal{R}\cong \Aut (l/k) $.
 
\item We begin by proving the following claim.
\begin{claim}\label{Claim:1}
Any automorphism $ \beta \in \Aut (l_3\geq k) $ has a unique extension $ \beta_0 \in \Aut (l\geq k) $ such that for all $ \gamma \in \Gal (l/k) $, $ \beta_0^{-1}\gamma \beta_0 = \gamma $.
\end{claim}

\begin{claimproof}
Let $ \beta \in \Aut (l_3\geq k) $ and let $ \tilde{\beta}\in \Aut (k_s/k) $ be an extension to $ k_s $. Since $ \beta $ preserves $ l_3 $ and since $ l $ is the normal closure of $ l_3 $, $ \tilde{\beta} $ preserves $ l $. Let $ l_3' $ and $ l_3'' $ be the two other degree $ 3 $ extension of $ k $ contained in $ l $. Either $ \tilde{\beta}(l_3') = l_3' $, or $ \tilde{\beta}(l_3') = l_3'' $. In the latter case, replace $ \tilde{\beta} $ by $ \tilde{\beta}\gamma_0 $, where $ \gamma_0 \in \Gal(l/k) $ acts trivially on $ l_3 $ and exchanges $ l_3' $ and $ l_3'' $. Hence we can assume that $ \tilde{\beta} $ preserves $ l_3 $, $ l_3' $ and $ l_3'' $. But now the restriction of $ \tilde{\beta} $ to $ l $ has the desired property. For uniqueness, note that if $ \beta_0' $ is another such extension, then $ \beta_0'\beta_0^{-1} $ is an element of $ \Gal (l/k) $ preserving $ l_3 $, $ l_3' $ and $ l_3'' $, hence $ \beta_0'\beta_0^{-1}  =1 $.
\end{claimproof}

\medskip

For $ \beta \in \Aut (l_3\geq k) $, we denote by $ \beta_0 $ the unique extension of $ \beta $ to an element of $ \Aut (l\geq k) $ provided by Claim~\ref{Claim:1}, and by $ \tilde{\beta_0} $ an extension of $ \beta_0 $ to $ \Aut (k_s/k)  $. Now the proof follows the same line as the previous proof of the previous item, and we discuss it more briefly. We define a map
\begin{align*}
\Phi \colon \Aut (l_3\geq k)&\to (\Aut R \times \Aut (k_s\geq k))^{\Gamma}/\tilde{\Gamma}\\
\beta &\mapsto \Id_{\tilde{\beta}_0}^{-1}
\end{align*}

The proof that $ \Phi (\beta) $ does not depend on a lift of $ \beta_0 $ and that $ \Id_{\tilde{\beta}_0}^{-1} $ is $ \Gamma $-invariant follows the same line as in the previous item. Furthermore, $ \Phi $ is clearly a homomorphism.

Assume now that $ \Phi (\beta) $ is trivial. Hence there exists $ \delta \in \Gamma $ such that $ \Id_{\tilde{\beta}_0}^{-1} = \tilde{c}_{\delta}\Id_{\delta}^{-1} $. Hence $ \tilde{c}_{\delta} $ is trivial, which implies that $ \delta $ acts trivially on $ l $, so that $ \beta $ was trivial. Hence $ \Phi $ is injective. Let us now prove surjectivity.  Let $ b\Id_{\tilde{\beta}} \in (\Aut R \times \Aut (k_s\geq k))^{\Gamma} $. Since $ \tilde{c}\colon \Gal (k_s/k)\to \Aut R $ is surjective and since we are working modulo $ \tilde{\Gamma} $, we can assume that $ b=1 $. We claim that $ \tilde{\beta} $ preserves $ l $ and that $ \tilde{\beta}^{-1}\gamma \tilde{\beta} = \gamma $ for all $ \gamma \in \Aut (l/k) $. Indeed, for all $ \delta \in \Gamma $, $ \delta . \Id_{\tilde{\beta}} = \tilde{c}_{\tilde{\beta}^{-1}\delta \tilde{\beta}}\tilde{c}_{\delta}^{-1}\Id_{\tilde{\beta}} $. Hence $ \Id_{\tilde{\beta}} $ is $ \Gamma $-invariant if and only if $ \tilde{c}_{\delta} = \tilde{c}_{\tilde{\beta}^{-1} \delta \tilde{\beta}} $ for all $ \delta \in \Gamma $. But if $ \tilde{\beta} $ does not preserve $ l $, there exists $ \delta \in \Gamma $ such that $ \tilde{c}_{\delta} = 1 \neq \tilde{c}_{\tilde{\beta}^{-1} \delta \tilde{\beta}}  $, a contradiction. The fact that $ \tilde{\beta}^{-1}\gamma \tilde{\beta} = \gamma $ for all $ \gamma \in \Aut (l/k) $ also follows directly, and the claim is proved.

To conclude, note that the claim implies that $ \tilde{\beta} $ preserves $ l_3 $, and hence up to an element in the image of $ \Phi $, we can assume that $ \tilde{\beta} $ acts trivially on $ l $, so that $ \Id_{\tilde{\beta}} $ is trivial modulo $ \tilde{\Gamma} $, as wanted. \qedhere
\end{enumerate}
\end{proof}

In the proof of Lemma~\ref{Lem:descriptionofabstractautomoprhismsofDyn}, we needed the following lemma.
\begin{lemma}\label{Lem:elements of order 2 in Aut R}
Let $ \mathcal{R} $ be a simple based root datum of type $ R $ with Tits index $ ^3D_{4,l} $. Then $ \mathcal{R} $ is simply connected or adjoint, and hence $ \Aut R $ contains an element of order $ 2 $.
\end{lemma}
\begin{proof}
If $ R $ is neither simply connected nor adjoint, the corresponding split connected reductive group is the split $ \SO_8 $ (there are actually three proper subgroups in the center of the split $ \Spin_8 $, but the corresponding intermediate quotients are all isomorphic). But the split $ \SO_8 $ does not have an outer automorphism of order $ 3 $, contradicting the fact that the Tits index of $ \mathcal{R} $ is $ ^3D_{4,l} $. The last part of the lemma follows from the fact that if $ R $ is simply connected or adjoint, $ \Aut R = \Aut D_4 $ (see \cite{Con14}*{Proposition~1.5.1}) and the fact that $ \Aut D_4 = S_3 $.
\end{proof}

\begin{corollary}\label{Cor:field auto preserving k}
Let $ \mathcal{R} $ be a simple $ k $-scheme of based root datum with classifying field $ l $. If $ \Aut (l/k)\ncong S_3 $, then  $ \Aut_{\mathcal{R}}(k) \cong \lbrace \alpha \in \Aut (k)~\vert~ \text{there exists } \tilde{\alpha }\in\Aut (l) \text{ extending } \alpha \rbrace $. While if $ \Aut (l/k)\cong S_3 $, then $ \Aut_{\mathcal{R}}(k) \cong \lbrace \alpha \in \Aut (k)~\vert~ \text{there exists } \tilde{\alpha }\in\Aut (l_3) \text{ extending } \alpha \rbrace $, where $ l_3 $ is a chosen non-normal cubic subextension of $ l/k $. %Note that a different choice $ l_3' $ does not change anything, because you can extend to $ l_3 $ iff you can extend to $ l_3' $, because there is an iso. fixing $ k $ and sending $ l_3 $ to $ l_3' $.
\end{corollary}
\begin{proof}
This follows from the surjectivity of $ \Aut (\mathcal{R} \to \Spec k)\to \Aut_{\mathcal{R}}(k) $ and from the description of $ \Aut (\mathcal{R} \to \Spec k) $ contained in Lemma~\ref{Lem:descriptionofabstractautomoprhismsofDyn}.
\end{proof}

In view of Corollary~\ref{Cor:field auto preserving k}, it is useful to introduce the following notation.
\begin{definition}
Let $ l\geq k $ be a field extension of $ k $. We denote by $ \Aut_l (k) $ the group of automorphisms of $ k $ which extend to an automorphism of $ l $, i.e.\ $ \Aut_l (k) = \lbrace \alpha \in \Aut (k)~\vert~ \text{there exists } \tilde{\alpha}\in \Aut (l) \text{ extending } \alpha \rbrace $.
\end{definition}

Using the identifications we made in Lemma~\ref{Lem:descriptionofabstractautomoprhismsofDyn} and Corollary~\ref{Cor:field auto preserving k}, we can rewrite in a very explicit form the short exact sequence $ 1\to \Aut \mathcal{R}(G) \to \Aut (\mathcal{R}(G) \to \Spec k)\to \Aut_{\mathcal{R}(G)}(k)\to 1 $.

\begin{proposition}\label{Prop:epplicit SES for Dyn}
Let $ \mathcal{R} $ be a simple $ k $-scheme of based root datum of type $ R $ with Tits index $ ^gX_{n,l} $.
\begin{enumerate}
\item If $ g=1 $, the short exact sequence $ 1\to \Aut \mathcal{R} \to \Aut (\mathcal{R} \to \Spec k)\to \Aut_{\mathcal{R}}(k)\to 1 $ is isomorphic to the short exact sequence $ 1\to \Aut R\to \Aut R\times \Aut (k)\to \Aut (k)\to 1 $. In particular, it always splits.

\item If $ g=2 $ or $ g=3 $, the short exact sequence $ 1\to \Aut \mathcal{R} \to \Aut (\mathcal{R} \to \Spec k)\to \Aut_{\mathcal{R}}(k)\to 1 $ is isomorphic to $ 1\to \Gal (l/k)\to \Aut (l\geq k)\to \Aut_l (k)\to 1 $.

\item If $ g = 6 $, let $ l_3 $ be a $ ( $non normal$ ) $ cubic subextension of $ l/k $. The short exact sequence $ 1\to \Aut \mathcal{R} \to \Aut (\mathcal{R} \to \Spec k)\to \Aut_{\mathcal{R}}(k)\to 1 $ is isomorphic to $ 1\to 1\to \Aut (l_3\geq k)\to \Aut_{l_3} (k)\to 1 $. In particular, it always splits.
\end{enumerate}
\end{proposition}
\begin{proof}
This is a direct consequence of Lemma~\ref{Lem:descriptionofabstractautomoprhismsofDyn} and Corollary~\ref{Cor:field auto preserving k}. Note that in each case, the map $ \Aut (l\geq k)\to \Aut_l (k) $ is given by restriction to $ k $. Also note that when $ g=6 $, and since $ l_3 $ is a non normal cubic extension of $ k $, the group $ \Aut (l_3/k) $ is trivial, and $ \Aut (l_3\geq k)\cong \Aut_{l_3}(k) $.
\end{proof}

We end this discussion with examples where the short exact sequence $ 1\to \Aut \mathcal{R} \to \Aut (\mathcal{R} \to \Spec k)\to \Aut_{\mathcal{R}}(k)\to 1 $ does not split.
\begin{definition}\label{Def:strongly rigid field}
The field $ k $ is called \textbf{rigid} if for any finite Galois extension $ k' $ of $ k $ such that $ k' $ is not algebraically closed, every automorphism of $ k' $ fixes $ k $ pointwise.
\end{definition}
\begin{definition}
A \textbf{prime field} is either the field of rational numbers of a finite field of order $ p $ for some prime $ p $.
\end{definition}

Examples of rigid fields include prime fields and $ \mathbf{Q}_p $ (the field of $ p $-adic numbers) for any prime $ p $. Let us give a reference for this latter assertion.
\begin{lemma}\label{Lem:Qp is rigid}
Let $ p $ be a prime number and let $ \mathbf{Q}_p $ be the field of $ p $-adic numbers. Let $ \mathbf{Q}_p\leq k' $ be a finite Galois extension. Then every automorphism of $ k' $ fixes pointwise $ \mathbf{Q}_p $.
\end{lemma}
\begin{proof}
The field $ k' $ is complete and non algebraically closed. Hence by \cite{Sch33}, all complete norms on $ k' $ are equivalent. Hence an automorphism of $ k' $ has to preserve the norm, which is to say that it has to be continuous. But since any automorphism acts trivially on $ \mathbf{Q} $, by continuity it also has to act trivially on $ \mathbf{Q}_p $.
\end{proof}

\begin{corollary}\label{Cor:very explicit non-splitting}
Assume that $ k $ is a finite $ ( $respectively possibly infinite$ ) $ Galois extension of a rigid $ ( $respectively prime$ ) $ field $ k_0 $. Let $ G $ be a connected reductive $ k $-group which is quasi-split and absolutely simple. Assume that $ \mathcal{R}(G) $ has Tits index $ ^{g}X_{n,l} $, with $ g = 2 $ or $ g = 3 $. Further assume that $ l $ is a Galois extension of $ k_0 $. Then $ \Aut_G(k) = \Aut (k) $ and the short exact sequence $ 1\to \Aut G\to \Aut (G \to \Spec k)\to \Aut_G(k)\to 1 $ splits if and only if $ 1\to \Gal (l/k)\to \Gal (l/k_0)\to \Gal (k/k_0)\to 1 $ splits. 
\end{corollary}
\begin{proof}
In view of Theorem~\ref{Thm: MainThm2} and Proposition~\ref{Prop:epplicit SES for Dyn}, the short exact sequence $ 1\! \to \! \Aut G\to \Aut (G \to \Spec k)\to \Aut_G(k)\to 1 $ splits if and only if the short exact sequence $ 1\to \Gal (l/k)\to \Aut (l\geq k)\to \Aut_l (k)\to 1 $ splits. Since $ k_0 $ is rigid (or even prime if $ k $ is an infinite Galois extension) and $ k $ is a normal extension, $ \Aut (l\geq k) = \Gal (l/k_0) $. Furthermore,  $ \Aut (k) = \Gal (k/k_0) $, and since $ l/k_0 $ is Galois, every element of $ \Gal (k/k_0) $ extends to $ \Gal (l/k_0) $. Hence $ \Aut_l (k) = \Gal (k/k_0) $, as wanted.
\end{proof}

\begin{remark}\label{Rem:an explicit description of quasi-split PGU}
Corollary~\ref{Cor:very explicit non-splitting} directly implies the corollary stated at the beginning of the 
introduction of this paper. Indeed, $ \mathbf{Q} $ is a prime field and $ \mathbf{Q}_p $ is rigid by Lemma~\ref{Lem:Qp is rigid}. Furthermore, $ \Aut_{\text{abstract}}(G(k)) = \Aut (G\to \Spec k) $ by the Borel--Tits theorem that we stated at the very beginning of the introduction. For the ease of non-expert readers, let us also give an explicit realisation of the quasi-split, absolutely simple, adjoint algebraic $ k $-group of type $ ^2A_{n-1} $ with corresponding quadratic separable extension $ l $: denote the Galois conjugation on $ l $ by $ x\mapsto \bar{x} $, and for $ g\in \PGL_n(l) $, set $ (^{\text{at}}\bar{g})_{ij} = \bar{g}_{n+1-j;n+1-i} $ (i.e.\ the anti-transposed conjugated matrix). We define $ \PGU_n(k) = \lbrace g\in \PGL_n(l)~\vert~^{\text{at}}\bar{g}g = 1\rbrace $. This is easily interpreted as the $ k $-rational points of an algebraic $ k $-group, and one readily sees that this algebraic $ k $-group is the quasi-split, absolutely simple, adjoint algebraic $ k $-group of type $ ^2A_{n-1} $ with corresponding quadratic separable extension $ l $ (because the corresponding cocycle is $ g\mapsto \,^{\text{at}}g^{-1} $, which is an outer automorphism of $ \PGL_n $ preserving its Borel subgroup consisting of upper triangular matrices).
\end{remark}

\section{The \texorpdfstring{$\textbf{SL}_n(D)$}{SLn(D)} case over a local field}\label{Sec:SL_n(D)}
\subsection[Outer automorphisms of CSA over local fields]{Outer automorphisms of finite dimensional central simple algebras over local fields}
We now explore the same question for algebraic groups of the form $ \SL_n(D) $. First, we need to be a bit more precise and make a distinction between the algebraic $ k $-group and its group of $ k $-rational points.
\begin{definition}\label{Def:algebraic SL_n(D)}
Let $ A $ be a finite dimensional central simple $ k $-algebra. Following the notation of \cite{KMRT98}, we denote the corresponding algebraic $ k $-group of ``reduced norm $ 1 $ elements'' by $ \textbf{SL}_1(A) $. The $ k $-rational points of $ \textbf{SL}_1(A) $ are the elements of $ A $ of reduced norm $ 1 $, and we denote this group by $ \SL_1(A) $. When $ A = M_n(A') $ for some finite dimensional central simple $ k $-algebra $ A' $, we also denote $ \textbf{SL}_1(A) $ (respectively $ \SL_1(A) $) by $ \textbf{SL}_n(A') $ (respectively $ \SL_n(A') $). 
\end{definition}
\begin{remark}
Note that for $ A $ a finite dimensional central simple $ k $-algebra and $ \alpha \in \Aut (k) $, $ ^{\alpha}\textbf{SL}_n(A) $ is naturally isomorphic (as an algebraic $ k $-group) to $\textbf{SL}_n(\,^{\alpha}A)  $. Hence by (a slightly enhanced version of) \cite{KMRT98}*{Remark~26.11}, $ \alpha \in \Aut_{\textbf{SL}_n(A)}(k) $ if and only if $ A\cong \,^{\alpha}A $ or $ A^{opp}\cong \,^{\alpha}A $ (as $ k $-algebras).%This is the case because M_n(D)\cong M_n(D') as k-algebras if and only if D\cong D' by Wedderburn theorem.
\end{remark}

We will restrict ourselves to working over a local field. For us, a \textbf{local field} is a non-archimedean non-discrete topological field which is locally compact (or equivalently, a field isomorphic to $ \mathbf{F}_{p^n}(\!(T)\!) $ or a finite extension of $ \mathbf{Q}_p $ for some prime number $ p $). For the rest of the paper, the letter $ K $ exclusively stands for a local field. Let us begin by recalling the classification of central simple algebras over local fields.

\begin{definition}\label{Def:the cyclic algebra A_(l/k, sigma , a)}
Let $ k $ be a field and let $ l/k $ be a finite cyclic extension of degree $d$. Let $ \sigma \in \Gal (l/k) $ be a generator of the cyclic group $\Gal (l/k)$, let $a\in k$ and let $ u $ be an abstract symbol. The cyclic algebra $ A(l/k,\sigma ,a, u) $ is defined as follows: as a $ k $-vector space, $ A(l/k,\sigma ,a, u)\cong \bigoplus \limits_{i=0}^{d-1}u^il $, and the multiplication is defined by using the relations $ u^{d} = a $ and $ u^{-1}xu=\sigma (x) $ for all $ x\in l $. We also denote it $ A(l/k,\sigma ,a) $.
\end{definition}
We recall that the algebra $ A(l/k,\sigma ,a, u) $ of Definition~\ref{Def:the cyclic algebra A_(l/k, sigma , a)} is always central simple over $k$, and that it is isomorphic to the $ k $-algebra $ M_n(k) $ if and only if $a$ is the norm of an element in $l$.

\begin{definition}\label{Def:the CSA A_(d,r)}
Let $ K $ be a local field and let $ d,r \in \mathbf{N} $ with $ d\geq 1 $. Let $ K_d $ be the unramified extension of $ K $ of degree $ d $, let $ \sigma \in \Gal (K_d/K) $ be the Frobenius automorphism (i.e.\ the automorphism inducing the Frobenius automorphism on $ \Gal (\overline{K_d}/\overline{K}) $), and let $ \pi $ be a uniformiser of $K$. We define $ A(d,r) $ to be the cyclic algebra $ A(K_d/K,\sigma ,\pi^{r}) $.
\end{definition}

Note that up to isomorphism, $ A(d,r) $ does not depend on the choice of $ \pi $. In fact, given two uniformisers $ \pi $ and $ \tilde{\pi} $, an explicit isomorphism $ (K_d/K,\sigma ,\pi^{r})\cong (K_d/K,\sigma ,\tilde{\pi}^{r}) $ having the same form as the one appearing in Lemma~\ref{Lem:automorphisms of A(d,r)} can be given.
\begin{lemma}\label{Lem:automorphisms of A(d,r)}
Let $ K $ be a local field. Let $ A = A(d,r) $ and $ K_d, \sigma , \pi $ be as in Definition~\ref{Def:the CSA A_(d,r)}. Let $ \alpha $ be an automorphism of $ K_d $ such that $ \alpha (K) = K $, and assume that there exists an element $ x $ in $ K_d $ such that $ N_{K_d/K}(x) = \frac{\alpha (\pi^{r})}{\pi^{r}} $. Then the map $ \phi (\alpha ,x) \colon A\to A\colon \sum \limits_{i=0}^{d-1}u^ia_i\mapsto \sum \limits_{i=0}^{d-1}(ux)^i\alpha (a_i) $ is a ring automorphism of $ A $.
\end{lemma}
\begin{proof}
We view $ A $ as a quotient of the twisted polynomial ring $ K_d[u; \sigma ] $ (see \cite{Jac96}*{Section~1.1} for the definition of a twisted polynomial ring) modulo the relation $ u^d = \pi^{r} $. Given an automorphism $ \alpha $ in $ \Aut (K_d) $, we can define a map $ f_{\alpha }\colon K_d[u;\sigma ]\to K_d[u;\sigma ]\colon \begin{cases}
u\mapsto ux\\
a\mapsto \alpha (a) \text{ for all } a \in  K_d
\end{cases} $. By \cite{Jac96}*{Proposition~4.6.20}, $ f_{\alpha} $ is a ring automorphism as soon as $ \alpha \sigma = \sigma \alpha  $. Recall that by assumption, $ \alpha (K) = K $. Hence $ \sigma^{-1}\alpha\sigma\alpha^{-1} $ belongs to $ \Gal (K_d/K) $, and its induced automorphism on the residue field $ \overline{K_d} $ is a commutator in $ \Aut (\overline{K_d}) $, thus trivial (note that since every automorphism of a local field is continuous, it always induces an automorphism of the residue field). We conclude that $ \sigma^{-1}\alpha\sigma\alpha^{-1} $ itself was trivial by \cite{S79}*{Chapter III, \S 5, Theorem 3}. Hence, $ f_{\alpha} $ is indeed a ring automorphism.

Furthermore, if it passes to the quotient, $ f_{\alpha} $ induces the automorphism $ \phi (\alpha ,x) $. Hence it suffices to check that $ f_{\alpha} $ preserves the relation. But we have $f_{\alpha}(u^d-\pi^{r}) = (ux)^d-\alpha (\pi^r) = u^dN_{K_d/K}(x)-\alpha (\pi^r) = (u^d-\pi^r)\frac{\alpha (\pi^r)}{\pi^r} $, as wanted.
\end{proof}

For $ \alpha $ an automorphism of a (non-necessarily commutative) ring $ R $, we denote by $ \tilde{\alpha} $ the corresponding automorphism of $ M_n(R) $ (the algebra of $ n\times n $ matrices with coefficient in $ R $) obtained by applying $ \alpha $ coefficient by coefficient. Also, for $ A $ a finite dimensional central simple algebra over a field $ k $, we denote by $ \Nrd \colon A\to k $ its reduced norm.

\begin{lemma}\label{Lem:Automatic preservation of the reducednorm}
Let $k$ be a field and let $A$ be a central simple $k$-algebra. For every ring automorphism $\alpha$ of $A$ and $x\in A$,
\[
\Nrd(\alpha(x))=\alpha(\Nrd(x)).
\]
\end{lemma}
\begin{proof}
Let $k_s$ be a separable closure of $k$. Every automorphism $\alpha$
of $A$ preserves the center $k$; the restriction $\alpha\rvert_k$
extends to an automorphism $\beta$ of $k_s$, and we may consider the
tensor product
\[
\alpha\otimes\beta\colon A\otimes_kk_s\to A\otimes_kk_s.
\]
Since $k_s$ splits $A$, we may also consider an isomorphism of
$k_s$-algebras $f\colon A\otimes_kk_s\to M_d(k_s)$. The ring
automorphism $f\circ(\alpha\otimes\beta)\circ f^{-1}$ of $M_d(k_s)$
restricts to $\beta$ on the center $k_s$, hence
$f\circ(\alpha\otimes\beta)\circ f^{-1}\circ \tilde\beta^{-1}$ is the
identity on $k_s$. Since every $k_s$-automorphism of $M_d(k_s)$ is
inner, we may find $g\in \GL_d(k_s)$ such that
\[
 f\circ(\alpha\otimes\beta)\circ f^{-1}\circ\tilde\beta^{-1} = \intaut (g).
\]
The following diagram then commutes:
\[
\begin{tikzcd}
A\otimes_kk_s \ar[r, "f"] \ar[d,swap, "\alpha\otimes\beta"] & M_d(k_s)
\ar[r,"\det"] \ar[d,swap, "\intaut (g)\circ\tilde\beta"]
  &
  k_s^\times\ar[d,"\beta"]
  \\
A\otimes_kk_s \ar[r,swap, "f"]&M_d(k_s)\ar[r,swap, "\det"] & k_s^\times
\end{tikzcd}
\]
Since $\Nrd=\det\circ f$, the lemma follows.
\end{proof}

We set some notations that we use for the rest of the paper.
\begin{definition}\label{Def:notation for automorphism of SLn(D)}
Let $ K $ be a local field. Let $ A(d,r) $ and $ K_d, \sigma , \pi $ be as in Definition~\ref{Def:the CSA A_(d,r)}. Let $ \alpha $ be an automorphism of $ K_d $ such that $ \alpha (K) = K $, and assume that there exists an element $ x $ in $ K_d $ such that $ N_{K_d/K}(x) = \frac{\alpha (\pi^{r})}{\pi^{r}} $. The map $ \tilde{\phi}(\alpha ,x)\colon M_n(A)\to M_n(A) $ corresponding to the automorphism $ \phi (\alpha ,x)\colon A\to A $ from Lemma~\ref{Lem:automorphisms of A(d,r)} preserves elements of reduced norm $ 1 $ by Lemma~\ref{Lem:Automatic preservation of the reducednorm}. We again denote its restriction to $ \SL_n(A) $ by $ \tilde{\phi}(\alpha ,x) $.
\end{definition}
\begin{remark}\label{Rem:abstract automorphisms are algebraic}
In the notations of Definition~\ref{Def:notation for automorphism of SLn(D)}, $ \tilde{\phi}(\alpha ,x) $ is an isomorphism of $ k $-algebras $ M_n(A)\cong \,^{\alpha^{-1}}M_n(A) $. Hence, by \cite{KMRT98}*{Theorem~26.9}, $ \tilde{\phi}(\alpha ,x) $ corresponds to a unique $ k $-isomorphism of algebraic groups $\textbf{SL}_n(A)\cong \,^{\alpha^{-1}}\textbf{SL}_n(A) $. More concretely, this can also be seen by using a representation of $ A $ in $ M_{d^2}(K) $ (where $ d $ is the degree of $ A $).
\end{remark}

The following observation explains in part why the local field case is so much simpler than say the global field case (see also the end of Remark~\ref{Rem:litterature on outer automorphism of D}).
\begin{lemma}\label{Lem:extending to unramfiied extensions}
Let $ K $ be a local field and let $ K_d $ be a finite dimensional unramified extension of $ K $. Any automorphism of $ K $ extends to an automorphism of $ K_d $.
\end{lemma}
\begin{proof}
Let $ \alpha \in \Aut (K) $. There exists an extension $ \beta $ of $ \alpha $ to the separable closure of $ K $. Note that if $ K_d \cong K[X]/(f) $, then $ \beta (K_d)\cong K[X]/(\,^{\alpha}f) $, where $ ^{\alpha}f  $ is the polynomial obtained from $ f $ by applying $ \alpha $ to its coefficients. But $ \alpha $ is continuous, and an extension is unramified if and only if it is isomorphic to $ K[X]/(g) $ for some polynomial $ g $ whose coefficients are all of valuation $ 0 $. Hence by uniqueness of unramified extensions of a given degree, $ \beta $ preserves $ K_d $.
\end{proof}

\begin{corollary}\label{Cor:etxending auto to D}
Let $ A $ be a finite dimensional central simple algebra over a local field $ K $. Every automorphism of $ K $ extends to an automorphism of $ A $. Hence, $ \Aut_{\textbf{SL}_n(A)}(K)= \Aut (K) $.
\end{corollary}
\begin{proof}
By Theorem~\ref{Thm:classificatio nof CSA over local fields}, the central simple algebra $ A $ is an algebra of the form $ A(d,r) $, i.e.\ a cyclic algebra of the form $ (K_d/K, \sigma ,\pi^{r}) $ with $ K_d, \sigma , \pi $ as in Definition~\ref{Def:the CSA A_(d,r)}.

Let $ \alpha \in \Aut (K) $. By Lemma~\ref{Lem:extending to unramfiied extensions}, there exists $ \beta \in \Aut (K_d) $ extending $ \alpha $. Also, by \cite{S79}*{Chapter V,\S 2, Corollary}, $ N_{K_d/K} $ is surjective on $ \mathcal{O}_{K}^{\times} $. Furthermore, any automorphism of a local field preserves the valuation. Hence there exists $ x\in K_d $ such that $ N_{K_d/K}(x) = \frac{\alpha (\pi^{r})}{\pi^{r}} $. Then the automorphism $ \phi (\beta ,x) $ defined in Lemma~\ref{Lem:automorphisms of A(d,r)} is an extension of $ \alpha $ to $ A $. Finally, $ \tilde{\phi}(\beta ,x) $ from Definition~\ref{Def:notation for automorphism of SLn(D)} is defined over $ \alpha^{-1} $, so that the last claim follows from Remark~\ref{Rem:abstract automorphisms are algebraic}.
\end{proof}

\begin{remark}\label{Rem:litterature on outer automorphism of D}
If $ \alpha \in \Aut (K) $ is of finite order, the result in Corollary~\ref{Cor:etxending auto to D} asserting that $ \alpha $ extends to an automorphism of $ A $ is an old result. Indeed, using Lemma~\ref{Lem:Base change of CSA}, it is a direct consequence of \cite{EML48}*{Corollary~7.3} (see also \cite{Han07}*{Theorem~5.6}) and the fact that $ A = M_n(D) $ for some division algebra $ D $. This  already settles the question in characteristic $ 0 $. In positive characteristic, Lemma~\ref{Cor:etxending auto to D} can be seen as a direct corollary of the results in \cite{Ha07}. Note that the fact that any extension of $ \alpha \in \Aut (K) $ to the separable closure of $ K $ preserves $ K_d $ simplifies matters (compare with Lemma~\ref{Lem:Hanke result part 2} when the extension $ \beta $ does not preserve the chosen maximal subfield $ l $).
%Remark to myself : Strangely, the paper of T. Hanke "Computation of Outer Automorphisms..." is referenced nowhere (because the proceedings of the Tenth Rhine Workshop on Computer Algebra (RWCA) does not appear anywhere).
\end{remark}

\subsection[Condition for the exact sequence not to split]{Sufficient condition for the exact sequence not to split}
We turn to the splitting question for the exact sequence $ 1\to \Aut G\to \Aut (G\to \Spec k)\to \Aut_G(k)\to 1 $, still assuming $ G = \textbf{SL}_n(A) $ over a local field. Let us introduce another notation for a subgroup of the group of semilinear automorphisms, which allow us to introduce a ``ground field''.
\begin{definition}
Let $ G $ be a $ k $-group scheme. Let $ k' $ be a subfield of $ k $. We denote by $ \Aut (G\to \Spec k/k') $ the subgroup of $ \Aut (G\to \Spec k) $ consisting of semilinear automorphisms over an automorphism $ \alpha $ belonging to $ \Aut (k/k') $. Furthermore, we denote by $ \Aut_G(k/k') $ the image of $ \Aut (G\to \Spec k/k') $ under the map $ \Aut (G\to \Spec k)\to \Aut_G(k) $.
\end{definition}

\begin{theorem}\label{Thm:local non-splitting for SLn(D)}
Let $ D $ be a central division algebra of degree $ d $ over a local field $ K $ and let $ G = \textbf{SL}_n(D) $. Let $ K' $ be a subfield of $ K $ such that $ K/K' $ is a finite Galois extension. Then the short exact sequence $ 1\to \Aut G\to \Aut (G\to \Spec K/K')\to \Aut_G(K/K')\to 1 $ splits if and only if $ \gcd (nd,[K:\nolinebreak K']) $ divides $ n $.
\end{theorem}
\begin{proof}
By Corollary~\ref{Cor:etxending auto to D}, $ Aut_G(K)=Aut(K) $. Hence, since $ \Gal (K/K') $ is contained in $ Aut_G(K) $, the short exact sequence splits if and only if $ G $ is defined over $ K' $ (see Theorem~\ref{Thm:Galois descent}). Let $ H $ be this hypothetical form of $ G $ over $ K' $.

The case $ d=1 $ being obviously true, let us assume that $ d\geq 2 $. Now, by the classification of simple groups over local fields (see \cite{Tits77}*{Section~4.2 and 4.3}), the Tits index of $ H $ is of the form $  ^{1}A^{(d')} $ or $ ^{2}A^{(1)} $, since these are the only groups of type $ A $ over local fields. Note that a distinguished orbit has to remain distinguished after scalar extension, because a non-trivial root remains non-trivial after scalar extension. Hence $ H $ cannot be of type $ ^{2}A^{(1)} $, because groups of type $ ^{2}A^{(1)} $ have extremal roots that are distinguished, whereas $ G $ has undistinguished extremal roots when $ d\geq 2 $. But the only groups of type $  ^{1}A^{(d')} $ are groups of the form $ \textbf{SL}_{n'}(D') $ where $ n'\geq 1 $ and $ D' $ is a division algebra over $ K' $. So we conclude that $ H $ is of this form.

We use the notation $ \inv $ for the map classifying division algebras over local fields (see Theorem~\ref{Thm:classificatio nof CSA over local fields} for a precise definition of $ \inv $). Let $ d' $ be the degree of $ D' $ over $ K' $, and let $ r' $ be such that $ [\frac{r'}{d'}] = \inv ([D']) $ in $ \mathbf{Q}/\mathbf{Z} $. Also, let $ a = \gcd (d',[K:K']) $. The base change of $ \textbf{SL}_{n'}(D') $ from $ K' $ to $ K $ is the algebraic group $ \textbf{SL}_{an'}(A(\frac{d'}{a},\frac{[K:K']}{a}r')) $ by Proposition~\ref{Prop:base change of SL_n(A)}. Since $ H $ is isomorphic to $ G $ over $ K $, $ an' = n $ and $ ad = d' $. Hence, $ a = \gcd (ad,[K:K']) $, which implies that $ \gcd (adn',[K:K']) $ divides $ an' $. Now, the equation $ an' = n $ already proves that if $ H $ exists, then $ \gcd (nd,[K:K']) $ divides $ n $.

Conversely, let $ a = \gcd (nd,[K:K']) $, and assume that $ a $ divides $ n $. We then set $ n' = \frac{n}{a} $, $ d' = ad $ and $ r' $ such that $ \frac{[K:K']}{a}r' - r \in d\mathbf{Z} $ (such an $ r' $ exists because $ \frac{[K:K']}{a} $ is prime to $ d $). With those parameters, the algebraic group $ \textbf{SL}_{n'}(A(d',r')) $ is a form of $ G $ over $ K' $, as wanted.
\end{proof}
\begin{remark}
The condition that $ \gcd (nd,[K:\nolinebreak K']) $ divides $ n $ is equivalent to require that for all primes $ p $ dividing $ d $, the $ p $-adic valuation of $ [K:K'] $ is less than or equal to the $ p $-adic valuation of $ n $.
\end{remark}

\begin{corollary}\label{Cor:non-splitting condition for SLn(D)}
Let $ D $ be a central division algebra of degree $ d $ over a local field $ K $ and let $ G = \textbf{SL}_n(D) $. The short exact sequence $ 1\to \Aut G\to \Aut (G\to \Spec K)\to \Aut_G(K)\to 1 $ does not split if there exists a subfield $ K'\leq K $ such that $ K/K' $ is finite Galois and $ \gcd (nd,[K:K']) $ does not divide $ n $.
\end{corollary}
\begin{proof}
$ 1\to \Aut G\to \Aut (G\to \Spec K/K')\to \Aut_G(K/K')\to 1 $ does not split by Theorem~\ref{Thm:local non-splitting for SLn(D)}, hence neither does $ 1\to \Aut G\to \Aut (G\to \Spec K)\to \Aut_G(K)\to 1 $.
\end{proof}

\subsection{Sufficient condition for the exact sequence to split}
In characteristic $ 0 $, it is actually straightforward to prove the converse of Corollary~\ref{Cor:non-splitting condition for SLn(D)}.
\begin{theorem}\label{Thm:splitting in char. 0 for SL_n(D)}
Let $ D $ be a central division algebra of degree $ d $ over a local field $ K $ of characteristic $ 0 $ and let $ G = \textbf{SL}_n(D) $. The short exact sequence $ 1\to \Aut G\to \Aut (G\to \Spec K)\to \Aut_G(K)\to 1 $ does not split only if there exists a subfield $ K'\leq K $ such that $ K/K' $ is finite Galois and $ \gcd (nd,[K:K']) $ does not divide $ n $.
\end{theorem}
\begin{proof}
By Corollary~\ref{Cor:etxending auto to D}, $ Aut_G(K)=Aut(K) $. Since $ K $ is of characteristic $ 0 $, it is a finite extension of $ \mathbf{Q}_p $ for some prime $ p $. But every automorphism of $ K $ acts trivially on $ \mathbf{Q}_p $ by Lemma~\ref{Lem:Qp is rigid}. Hence, by Galois theory, $ \Aut (K) $ is a finite group. Furthermore, letting $ K^{\Aut (K)} $ be the subfield of $ K $ fixed by $ \Aut (K) $, the extension $ K/K^{\Aut (K)} $ is Galois with Galois group $ \Aut (K) $.

Let $ a = \gcd (nd,[K:K^{\Aut (K)}]) $. Assuming that there does not exist a subfield $ K'\leq K $ such that $ K/K' $ is finite Galois and such that $ \gcd (nd,[K:K']) $ does not divide $ n $, we have in particular that $ a $ divides $ n $. Also, let $ r\in \mathbf{N} $ be such that $ [\frac{r}{d}] = \inv ([D]) $. Since $ \frac{[K:K^{\Aut (K)}]}{a} $ is prime to $ d $, there exists $ r'\in \mathbf{N} $ such that $ \frac{[K:K^{\Aut (K)}]}{a}r' - r \in d\mathbf{Z} $. Hence, by Proposition~\ref{Prop:base change of SL_n(A)}, the algebraic group $ \SL_{\frac{n}{a}}(A(ad,r')) $ is a form of $ G $ over $ K^{\Aut (K)} $, because $ \gcd (ad,[K:K^{\Aut (K)}]) = a $. But in view of Lemma~\ref{Lem:field of definition are in AutG(k)}, this implies that the homomorphism $ \Aut (G\to \Spec K)\to \Aut (K) = \Gal (K/K^{\Aut( K)}) $ has a section, as wanted.
\end{proof}
\begin{remark}\label{Rem:Vivid example of splitting for divison algebras}
Putting Corollary~\ref{Cor:non-splitting condition for SLn(D)} and Theorem~\ref{Thm:splitting in char. 0 for SL_n(D)} together already proves Theorem~\ref{Thm:Main Thm 2.2} in characteristic $ 0 $. In particular, the sequence always splits for $ K = \mathbf{Q}_p $ (this actually directly follows from the rigidity of $ \mathbf{Q}_p $, which was used in the proof of Theorem~\ref{Thm:splitting in char. 0 for SL_n(D)}). For a more interesting example, if $ K $ is a Galois extension of $ \mathbf{Q}_p $ of degree $ p^i $ for some prime $ p $ and some $ i\in \N $, then Theorem~\ref{Thm:Main Thm 2.2} asserts that the following are equivalent:
\begin{enumerate}
\item The sequence $ 1\to \Aut \textbf{SL}_n(D)\to \Aut (\textbf{SL}_n(D)\to \Spec K)\to \Aut_{\textbf{SL}_n(D)}(K)\to 1 $ splits.
\item If $ n $ is not divisible by $ p^i $, the degree of $ D $ is not divisible by $ p $.
\end{enumerate}
\end{remark}

We now aim to prove an analogue of Theorem~\ref{Thm:splitting in char. 0 for SL_n(D)} but in positive characteristic. When $ K $ is of positive characteristic, the fixed field $ K^{\Aut (K)} $ is finite and $ K/K^{\Aut( K)} $ is not Galois. Thus we cannot use the same method than in characteristic $ 0 $.

Instead, the strategy goes as follows: we decompose $ \Aut( K) $ in various pieces, we give a section of $ \Aut (\textbf{SL}_n(D)\to \Spec K)\to \Aut (K) $ separately for each pieces and then we check that everything can be glued. Let us begin by decomposing $ \Aut (K) $.

\begin{lemma}
Let $ K = \mathbf{F}_{p^i}(\!(T)\!) $. Since $ \mathbf{F}_{p^i} $ is the algebraic closure in $ K $ of the prime field of $ K $, $ \mathbf{F}_{p^i} $ is preserved by any automorphism of $ K $. Let $ N(K) = \lbrace \alpha \in \Aut (K)~\vert~\alpha $ acts trivially on $ \mathbf{F}_{p^i} \rbrace $. We have $ \Aut (K) \cong N(K)\rtimes \Gal (K/\mathbf{F}_{p}(\!(T)\!)) $.
\end{lemma}
\begin{proof}
We want to show that the short exact sequence $ 1\to N(K)\to \Aut (K)\xrightarrow{f} \Gal (\mathbf{F}_{p^i}/\mathbf{F}_p)\to 1 $ splits. But by \cite{S79}*{Chapter III, \S 5, Theorem 3}, $ f $ maps $ \Gal (K/\mathbf{F}_{p}(\!(T)\!)) $ isomorphically onto $ \Gal (\mathbf{F}_{p^i}/\mathbf{F}_p) $, hence the result.
\end{proof}

We furthermore decompose the group $ N(K) $. Since automorphisms of $ K $ are continuous, an element $ \alpha $ of $ N(K) $ is therefore determined by its action on $ T $, and we have $ \alpha (T) = \sum \limits_{j=1}^{\infty}a_jT^j $, where $ a_1 \in \mathbf{F}_{p^i}^{\times} $ and $ a_j \in \mathbf{F}_{p^i} $ for all $ j\geq 2 $.
\begin{definition}\label{Def:J(K) and C(K)}
Let $ J(K) = \lbrace \alpha \in N(K)~\vert~\alpha (T) = T + \sum \limits_{j=2}^{\infty}a_jT^j,~a_j \in \mathbf{F}_{p^i}\rbrace $ and let $ C_{p^i-1} = \lbrace \alpha \in N(K)~\vert~\alpha (T) = aT,~a \in \mathbf{F}_{p^i}^{\times}\rbrace $. With those notations, the group $ N(K) $ is isomorphic to $ J(K)\rtimes \mathbf{F}_{p^i}^{\times} $. For $ x\in \mathbf{F}_{p^i}^{\times} $, we denote by $ \ev (xT) $ the corresponding element of $ \Aut (K) $.
\end{definition}

In summary, we have decomposed $ \Aut (K) $ as the group $ (J(K)\rtimes \mathbf{F}_{p^i}^{\times})\rtimes \Gal (K/\mathbf{F}_{p}(\!(T)\!)) $. We go on by giving a section to $ \Aut (\textbf{SL}_n(D)\to \Spec K)\to \Aut (K) $ for each component of $ \Aut (K) $, one at a time. In doing so, we will at the same time take care that the given section glues well with the other sections (though each are studied separately). Hence, a given formula for a section on one component of $ \Aut (K) $ will at times be slightly more complicated than a formula one would naturally consider if one was not aiming for a global section. In each case, we write a remark to explain how the given formula could be simplified if not aiming for a global splitting.

We need to set a few notations.
\begin{definition}\label{Def:notation for splitting explicitly}
\begin{enumerate}
\item Let $ \mathbf{F}_p^{\Alg} $ be an algebraic closure of $ \mathbf{F}_p $. We denote by $ F $ the Frobenius automorphism of $ \mathbf{F}_p^{\Alg}(\!(T)\!) $ (i.e.\ the automorphism of $ \mathbf{F}_p^{\Alg}(\!(T)\!)  $ fixing $ \mathbf{F}_p(\!(T)\!)  $ and inducing the Frobenius automorphism on $ \mathbf{F}_p^{\Alg} $). For any finite extension extension $ L $ of $ \mathbf{F}_p $, we also denote by $ F $ the restriction of $ F $ to $ L $ and to $ L (\!(T)\!) $.

\item We fix  $ p $ a prime number, $i,n,d,r\in \mathbf{N}_{>0} $ such that $ \gcd (d,r) = 1 $ and two symbols $u,T $. We set $ K = \mathbf{F}_{p^i}(\!(T)\!) $, $ E = \mathbf{F}_{p^{id}}(\!(T)\!) $ and $ D = (E/K,F^i ,T^r, u) $ (a cyclic division algebra of degree $ d $ over $ K $ with symbol $ u $ as in Definition~\ref{Def:the CSA A_(d,r)}). Furthermore, we let $ G $ be the algebraic $ K $-group $ \textbf{SL}_n(D) $.

\item For $ \alpha \in N(K) $ we define its extension $ \alpha_{E} $ to $ \Aut (E) $ as follows: $ \alpha_{E} $ acts trivially on the residue field, while $ \alpha_E(T) = \alpha (T) $. We thus get an injective homomorphism $ N(K)\to N(E)\colon \alpha \mapsto \alpha_{E} $. Abusing notations, we again denote $ \alpha_{E} $ by $ \alpha $.

\item \label{Def:decomposing C(pi-1)} We fix $ a, b\in \mathbf{N} $ such that $ ab = p^i-1 $, $ \gcd (d^{p^i-1},p^i-1) = \gcd (d^{b},b) = b $ and $ \gcd (d,a) = 1 $.

\item \label{Def:decomposing C(i)} We fix $ a', b'\in \mathbf{N} $ such that $ a'b' = i $, $ \gcd (d^{i},i) = \gcd (d^{i},b') = b' $ and $ \gcd (d,a') = 1 $.

\item We choose a generator $ \zeta $ of the multiplicative group $ \mathbf{F}_{p^{i}}^{\times} $.

\item For $ g\in \PGL_n(D) $, we denote by $ \intaut (g) $ the automorphism by conjugation of $ g $ on $ G $, i.e.\ $ \intaut (g)\colon G\to G\colon h\mapsto ghg^{-1} $.
\end{enumerate}
\end{definition}
\begin{remark}\label{Rem:on the notation for explicit splitting}
The natural numbers $ a,b,a',b' $ are uniquely determined by their definition. Note that we have in particular $ \gcd(a,b) = 1 = \gcd (a', b') $. Also note that by Definition~\ref{Def:the CSA A_(d,r)}, $ u^{-1}xu = F^i(x) $ for all $ x\in E $.
\end{remark}

We will further make use of the following notation: if $ l,m\in \mathbf{N} $ and $ A_1,\dots , A_l $ are $ m\times m $ matrices, we denote $ \diag (A_1,\dots , A_l) $ the corresponding block diagonal $ lm\times lm $ matrix. Furthermore, the $ m\times m $ identity matrix is denoted $ \Id_m $. We will denote the cyclic group of order $ m $ by $ C_m $. 

\begin{proposition}\label{Prop:constructing fJK}
Keep the notations of Definition~\ref{Def:notation for splitting explicitly}. Assume that $ \gcd (p,d) = 1 $ and that $ bb' $ divides $ n $. For $ \alpha \in J(K) $, there exists a unique $ x_{\alpha} \in 1+T\mathbf{F}_{p^i}[\![T]\!] $ such that $ x_{\alpha}^{db'} = \frac{\alpha (T^{r})}{T^{r}} $. Let $ M = \diag (\Id_b, x_{\alpha}\Id_b, x_{\alpha}^2\Id_b,\dots ,x_{\alpha}^{b'-1}\Id_b) $ $ ( $so that $ M $ is a $ bb'\times bb' $ matrix which is block diagonal$ ) $ and let $ X_{\alpha} = \diag (M,\dots ,M) $ where we have $ \frac{n}{bb'} $ terms $ ( $so that $ X_{\alpha} $ is a $ n\times n $ matrix which is block diagonal with coefficients in $ 1+T\mathbf{F}_{p^i}[\![T]\!]) $. Recalling the notation introduced in Remark~\ref{Rem:abstract automorphisms are algebraic}, the map
\begin{align*}
f_{J(K)}\colon J(K)&\to \Aut (G\to \Spec K)\\
\alpha &\mapsto \intaut (X_{\alpha})\tilde{\phi}(\alpha ,x_{\alpha}^{b'})
\end{align*}
is a homomorphism whose composition with the map $ \Aut (\! G\to \Spec K\! )\! \to \Aut_G(K) $ is the identity on $ J(K) $.
\end{proposition}
\begin{proof}
Note that $ \gcd(p,d) = 1 $ and $ \gcd (d^{b'},b') = b' $ implies $ \gcd (p,b') = 1 $, so that $ \gcd (p,db')  =1 $. Hence, for $ \alpha \in J(K) $ the existence and uniqueness of $ x_{\alpha} $ in $ 1+T\mathbf{F}_{p^i}[\![T]\!] $ such that $ x_{\alpha}^{db'} = \frac{\alpha (T^{r})}{T^{r}} $ follows directly from Hensel's lemma. We claim that for $ \alpha ,\beta \in J(K) $, $ x_{\beta \circ \alpha} = x_{\beta}.\beta (x_{\alpha}) $. By uniqueness, this equation holds if and only if $ \frac{(\beta \circ \alpha )(T^{r})}{T^{r}}=[x_{\beta}.\beta (x_{\alpha})]^{db'} $. But the right hand side is equal to $ \frac{\beta (T^{r})}{T^r}.\beta (\frac{\alpha (T^{r})}{T^r})  $, which is indeed equal to $ \frac{(\beta \circ \alpha )(T^{r})}{T^{r}} $. Checking that $ f_{J(K)} $ is a homomorphism is now straightforward: 
\begin{align*}
\intaut (X_{\beta})\tilde{\phi}(\beta ,x_{\beta}^{b'}) \circ \intaut (X_{\alpha})\tilde{\phi}(\alpha ,x_{\alpha}^{b'})&=\intaut (X_{\beta}\beta (X_{\alpha}))\tilde{\phi}(\beta \circ \alpha , x_{\beta}^{b'}.\beta (x_{\alpha}^{b'}))\\
&=\intaut (X_{\beta \circ \alpha})\tilde{\phi}(\beta \circ \alpha , x_{\beta \circ \alpha}^{b'}). \qedhere 
\end{align*}
\end{proof}

Note that if we were not aiming to define a global section of $ \Aut (G\to \Spec K)\to \Aut_G(K) $, we could just as well get rid of the factor $ \intaut(X_{\alpha}) $ and hence we would not need the assumption that $ bb' $ divides $ n $. In light of this, the next proposition really is a converse to Proposition~\ref{Prop:constructing fJK}.

\begin{proposition}
Keep the notations of Definition~\ref{Def:notation for splitting explicitly}. If $ \gcd (p,d) \neq 1 $, there does not exist a homomorphism $ J(K)\to \Aut (G\to \Spec K) $ whose composition with $ \Aut (G\to \Spec K)\to \Aut_G(K) $ is the identity on $ J(K) $.
\end{proposition}
\begin{proof}
By Theorem~\ref{Thm:local non-splitting for SLn(D)}, it suffices to prove that there exists $ K'\leq K $ such that $ K/K' $ is finite Galois with $ \Gal (K/K')\leq J(K) $ and such that $ \gcd (nd,[K:K']) $ does not divide $ n $. Let $ H $ be a group of order $ p^n $.  By \cite{C97}*{Theorem 3}, there exists an injective homomorphism $ H\hookrightarrow J(\mathbf{F}_{p}(\!(T)\!)) $. Also note that $ J(\mathbf{F}_{p}(\!(T)\!)) $ can be seen as a subgroup of $ J(K) $ in a natural way, so that $ J(K) $ has a subgroup of order $ p^n $, that we again denote by $ H $. Now, let $ K' = K^H = \lbrace x\in K~\vert~\alpha (x) = x \text { for all } \alpha \in H\rbrace $. Hence, $ K/K' $ is a Galois extension with $ \Gal (K/K') = H\leq J(K) $ and $ \gcd (nd,[K:K']) =\gcd (nd,p^n) $ does not divide $ n $ because $ \gcd (p,d) \neq 1 $, as wanted.
\end{proof}

We now construct a section of $ \Aut (G\to \Spec K)\to \Aut (K) $ for $ \mathbf{F}_{p^i}^{\times} $. In fact, using the same line of argument as for Theorem~\ref{Thm:local non-splitting for SLn(D)}, we know that a section for $ \mathbf{F}_{p^i}^{\times} $ exists if and only if $ \gcd (nd,p^{i}-1) $ divides $ n $ (where $ d $ and $ n $ appear in the form of $ G = \SL_n(D) $, $ d $ denoting as usual the degree of $ D $). But we need to have an explicit formula, since we want to ensure that it glues well with the map $ f_{J(K)} $ constructed in Proposition~\ref{Prop:constructing fJK}. We found those explicit formulas by working out by hand some low degree examples for which we could follow explicitly what the theory was predicting, and then by generalising our findings to any degree. Again, we give a section which is going to be slightly more complicated than necessary, because we aim to define a global section in the end.

We first need a $ (db') $-th root of $ \zeta^{br} $.
\begin{lemma}\label{Lem:db' root of zeta}
Keep the notations of Definition~\ref{Def:notation for splitting explicitly} and let $ C_a $ be the group generated by $ \zeta^b $ in $ \mathbf{F}_{p^i}^{\times} $. There exists a unique $ z\in C_a $ such that $ z^{db'} = \zeta^{br} $.
\end{lemma}
\begin{proof}
Note that $ \gcd(d,a) = 1 $ and $ \gcd (d^{b'},b') = b' $ implies $ \gcd (b',a) = 1 $, so that $ \gcd (db',a)  =1 $. Hence the result follows from the fact that $ db' $ is invertible in the cyclic group of order $ a $.
\end{proof}

\begin{proposition}\label{Prop:constructing fCl(d)}
Keep the notations of Definition~\ref{Def:notation for splitting explicitly} and of Lemma~\ref{Lem:db' root of zeta}, so that $ z^{db'} = \zeta^{br} $. Assume that $ bb' $ divides $ n $. Let $ M = \diag (\Id_b, z\Id_b, z^2\Id_b,\dots , z^{b'-1}\Id_b) $ $ ( $so that $ M $ is a $ bb'\times bb' $ matrix which is block diagonal$ ) $ and let $ Z = \diag (M,\dots ,M) $ where we have $ \frac{n}{bb'} $ terms $ ( $so that $ Z $ is a $ n\times n $ matrix which is block diagonal with coefficients in $ \mathbf{F}_{p^i} )$. Recalling the notation introduced in Remark~\ref{Rem:abstract automorphisms are algebraic}, the map
\begin{align*}
f_{C_a}\colon C_a&\to \Aut (G\to \Spec K)\\
\ev (\zeta^{bj}T) &\mapsto \intaut (Z^j)\tilde{\phi}(\ev (\zeta^{bj}T) ,z^{b'j})
\end{align*}
is a homomorphism whose composition with the map $ \Aut (\! G\to \Spec K\! )\! \to \Aut_G(K) $ is the identity on $ C_a $.
\end{proposition}
\begin{proof}
Note that $ \frac{\ev(\zeta^{bj}T)(T^{r})}{T^{r}}=\zeta^{brj}=(z^{b'j})^d=N_{E/K}(z^{b'j}) $, so that we can indeed use Definition~\ref{Def:notation for automorphism of SLn(D)}. With these definitions, for all $ j,j' \in \mathbf{N} $, we have
\begin{align*}
\intaut (Z^j)\tilde{\phi}(\ev (\zeta^{bj}T) ,z^{b'j}) \circ \intaut (Z^{j'})\tilde{\phi}(\ev (\zeta^{bj'}T) ,z^{b'j'})&=\intaut (Z^jZ^{j'})\tilde{\phi}(\ev (\zeta^{bj}T) \circ \ev (\zeta^{bj'}T) , z^{b'j}.z^{b'j'})\\
&=\intaut (Z^{j+j'})\tilde{\phi}(\ev (\zeta^{b(j+j')}T) ,z^{b'(j+j')}). 
\end{align*}
Hence the fact that $ f_{C_a} $ is well-defined follows from $ z^a = 1 $ (which holds because $ z^a $ is the unique $ (db') $-th root in $ C_a $ of $ \zeta^{abr} = 1 $).
\end{proof}
\begin{remark}
In the proof of Proposition~\ref{Prop:constructing fCl(d)}, we needed to show that $ f_{C_a}(\ev (\zeta^{b}T))^a $ is a trivial element of $ \Aut (G\to \Spec K) $. We proved it by showing that this algebraic automorphism of $ \textbf{SL}_n(D) $ induces a trivial automorphism of $ SL_n(D) $, hence is itself trivial by the density of rational points for $ G $. This will be used repeatedly to show that $ K $-automorphisms of $ G $ are trivial. In using that argument, it is also important to notice that a semilinear automorphism of the form $ \intaut (g)\tilde{\phi}(\alpha ,x) $ is algebraic if and only if $ \alpha $ acts trivially on $ K $.
\end{remark}
\begin{remark}
Note that in Proposition~\ref{Prop:constructing fCl(d)}, the factor $ \intaut (Z^j) $ is unnecessary if one is just interested in a section defined on $ C_a $ alone. Hence, a section of $ \Aut (G\to \Spec K)\to \Aut_G(K) $ only defined on $ C_a $ always exists (i.e.\ one does not need to assume that $ bb' $ divides $ n $). 
\end{remark}

\begin{proposition}\label{Prop:constructing fCk(d)}
Keep the notations of Definition~\ref{Def:notation for splitting explicitly} and assume that $ b $ divides $ n $. Let $ C_b $ be the group generated by $ \zeta^a $ in $ \mathbf{F}_{p^i}^{\times} $. There exists an element $ y\in \mathbf{F}_{p^{idb}} $ such that $ \frac{F^{id}(y)}{y} = \zeta^{ar} $. Choosing a $ \mathbf{F}_{p^{id}} $-basis of $ \mathbf{F}_{p^{idb}} $, we obtain an embedding $ \varphi \colon \mathbf{F}_{p^{idb}}\to M_b(\mathbf{F}_{p^{id}}) $. Let $ g = \varphi (y^{-1}) $ and let $ Y = \diag (g,\dots ,g) $ where we have $ \frac{n}{b} $ terms $ ( $so that $ Y $ is a $ n\times n $ matrix which is block diagonal with coefficients in $ \mathbf{F}_{p^{id}} )$. Recalling the notation introduced in Remark~\ref{Rem:abstract automorphisms are algebraic}, the map
\begin{align*}
f_{C_b}\colon C_b&\to \Aut (G\to \Spec K)\\
\ev (\zeta^{aj}T) &\mapsto \intaut (Y^j)\tilde{\phi}(\ev (\zeta^{aj}T) ,(\frac{F^i(y)}{y})^j)
\end{align*}
is a homomorphism whose composition with the map $ \Aut (\! G\to \Spec K\! )\! \to \Aut_G(K) $ is the identity on $ C_b $.
\end{proposition}
\begin{proof}
For the existence of $ y\in \mathbf{F}_{p^{idb}} $ such that $ \frac{F^{id}(y)}{y} = \zeta^{ar} $, note that $ N_{\mathbf{F}_{p^{idb}}/\mathbf{F}_{p^{id}}}(\zeta^{ar}) = \zeta^{abr} = 1 $. Also note that the extension $ \mathbf{F}_{p^{idb}}/\mathbf{F}_{p^{id}} $ is Galois cyclic, and that $ F^{id} $ generates its Galois group. Hence, by Hilbert's Theorem~90, there indeed exists $ y\in \mathbf{F}_{p^{idb}} $ such that $ \frac{F^{id}(y)}{y} = \zeta^{ar} $. For the rest of the proof, we choose such an $ y $.

From $ \frac{F^{id}(y)}{y} = \zeta^{ar} $, it readily follows that $ F^i(y)y^{-1} $ and $ y^b $ belong to $ \mathbf{F}_{p^{id}} $, since they are both invariant under $ F^{id} $. Note that $ \frac{\ev(\zeta^{aj}T)(T^{r})}{T^{r}}=\zeta^{ajr}=(\frac{F^{id}(y)}{y})^j=N_{E/K}((\frac{F^i(y)}{y})^j) $, so that we can indeed use Definition~\ref{Def:notation for automorphism of SLn(D)}.

It remains to check that $ f_{C_b} $ is well-defined and is a homomorphism. Note that for all $ j,j' \in \mathbf{N} $, we have
\begin{align*}
&\intaut (Y^j)\tilde{\phi}(\ev (\zeta^{aj}T) ,(\frac{F^i(y)}{y})^j) \circ \intaut (Y^{j'})\tilde{\phi}(\ev (\zeta^{aj'}T) ,(\frac{F^i(y)}{y})^{j'})\\
&=\intaut (Y^jY^{j'})\tilde{\phi}(\ev (\zeta^{aj}T) \circ \ev (\zeta^{aj'}T) ,(\frac{F^i(y)}{y})^{j+j'})\\
&=\intaut (Y^{j+j'})\tilde{\phi}(\ev (\zeta^{a(j+j')}T) ,(\frac{F^i(y)}{y})^{j+j'}). 
\end{align*}
Hence, it suffices to check that  $ \intaut (Y^b)\tilde{\phi}(\ev (\zeta^{ab}T) ,(F^i(y)y^{-1})^b) $ is the identity on $ \SL_n(D) $. But since $ y^b\in \mathbf{F}_{p^{id}} $, $ g^b $ and $ Y^b $ are diagonal and by definition of $ Y $, $ y^{-b}Y^{-b} = 1 $. Furthermore, recall that $ u^{-1}Y^bu = F^i(Y^b) $ (see Remark~\ref{Rem:on the notation for explicit splitting}). Hence $ \intaut (Y^b)\tilde{\phi}(1,(F^i(y)y^{-1})^b) (u.\Id_n) = Y^b (u(F^i(y)y^{-1})^b \Id_n) Y^{-b} = u F^i(Y^b)F^i(y^b)y^{-b}Y^{-b} = u.\Id_n $. Since $ \intaut (Y^b)\tilde{\phi}(1,(F^i(y)y^{-1})^b) $ also acts trivially on $ \SL_n(E)\leq \SL_n(D) $, this concludes the proof.
\end{proof}
\begin{remark}
When trying to find a section of $ \Aut (G\to \Spec K)\to \Aut_G(K) $ only defined on $ C_b $, this is the only formula we could come up with. Otherwise stated, the complicatedness of the formula defining $ f_{C_b} $ does not come from the need to adjust it to other partial sections of $ \Aut (G\to \Spec K)\to \Aut_G(K) $.
\end{remark}

Though not needed, we check that the automorphism $ \intaut (Y^j)\tilde{\phi}(\ev (\zeta^{aj}T) ,(F^i(y)y^{-1})^j) $ appearing in Proposition~\ref{Prop:constructing fCk(d)} does not depend on the choice of $ y $.
\begin{lemma}\label{Lem:the splitting on Cb does not depend on the choice of y}
Keep the notations of Proposition~\ref{Prop:constructing fCk(d)}. Let $ y'\in \mathbf{F}_{p^{idb}} $ such that $ \frac{F^{id}(y')}{y'} = \zeta^{ar} $. Let $ g' = \varphi ({y'}^{-1}) $ and let $ Y' = \diag (g',\dots ,g') $ where we have $ \frac{n}{b} $ terms. Then $ f_{C_b}(\ev (\zeta^{aj}T)) = \intaut ({Y'}^j)\tilde{\phi}(\ev (\zeta^{aj}T) ,(\frac{F^i(y')}{y'})^j) $.
\end{lemma}
\begin{proof}
Those two elements of $ \Aut (G\to \Spec K) $ differ by $ \intaut (Y^j{Y'}^{-j})\tilde{\phi}(1,(\frac{F^i(y)y'}{F^i(y')y'})^j) $. Let $ x = y{y'}^{-1} $. Note that $ x $ belongs to $ \mathbf{F}_{p^{id}} $ because $ x $ is invariant under $ F^{id} $, and hence $ Y^j{Y'}^{-j} = \diag (g{g'}^{-1}, \dots ,g{g'}^{-1})^j $ is actually the diagonal matrix $ x^j.\Id_n $.

Hence $ \intaut (Y^j{Y'}^{-j})\tilde{\phi}(1,(\frac{F^i(y)y'}{F^i(y')y'})^j) = \intaut (x^j.\Id_n)\tilde{\phi}(1,(\frac{F^i(x)}{x})^j) $. But this automorphism is trivial, since $ \intaut (x^j.\Id_n)\tilde{\phi}(1,(\frac{F^i(x)}{x})^j) (u.\Id_n) = x^j.u(\frac{F^i(x)}{x})^j.x^{-j}\Id_n = u.\Id_n $.
\end{proof}

As before, we can also prove a converse to Proposition~\ref{Prop:constructing fCk(d)}.
\begin{proposition}
Keep the notations of Proposition~\ref{Prop:constructing fCk(d)}. If $ b $ does not divide $ n $, there does not exist a homomorphism $ C_b\to \Aut (G\to \Spec K) $ whose composition with $ \Aut (G\to \Spec K)\to \Aut_G(K) $ is the identity on $ C_b $.
\end{proposition}
\begin{proof}
By Theorem~\ref{Thm:local non-splitting for SLn(D)}, it suffices to prove that there exists $ K'\leq K $ such that $ K/K' $ is finite Galois with $ \Gal (K/K')\leq C_b $ and such that $ \gcd (nd,[K:K']) $ does not divide $ n $. Recall (see Definition~\ref{Def:notation for splitting explicitly}) that $ a $ is prime to $ b $ with $ ab = p^i-1 $ so that $ \zeta^a $ is a $ b $-th primitive root of unity of $ K $. Hence, $ K' = \mathbf{F}_{p^i}(\!(T^b)\!) $ is such that $ K/K' $ is Galois of degree $ b $, and $ \Gal (K/K') $ is generated by the automorphism of $ K $ sending $ T $ to $ \zeta^a T $, so that $ \Gal (K/K') = C_b $. Finally, $ \gcd (nd,[K:K']) = \gcd (nd,b) $ does not divide $ n $ because by definition $ b = \gcd (d^b,b) $.
\end{proof}

Finally, we construct a section to $ \Aut (G\to \Spec K)\to \Aut (K) $ for $ \Gal (\mathbf{F}_{p^i}(\!(T)\!)/\mathbf{F}_{p}(\!(T)\!)) $.

\begin{proposition}\label{Prop:constructing fCa'}
Keep the notations of Definition~\ref{Def:notation for splitting explicitly}. Let $ C_{a'} $ be the group generated by $ F^{b'} $ in $ \Gal (K/\mathbf{F}_{p}(\!(T)\!)) $. Let $ c\in \mathbf{N} $ be such that $ ca'+1\in d\mathbf{Z} $ $ ( $which exists because $ \gcd (a',d) = 1) $. Recalling the notation introduced in Remark~\ref{Rem:abstract automorphisms are algebraic}, the map
\begin{align*}
f_{C_{a'}}\colon C_{a'}&\to \Aut (G\to \Spec K)\\
F^{b'j} &\mapsto \tilde{\phi}(F^{j(ci+b')},1)
\end{align*}
is a homomorphism. Furthermore, its composition with the map $ \Aut (\! G\! \to \Spec K)\to \Aut_G(K) $ is the identity on $ C_{a'} $.
\end{proposition}
\begin{proof}
Since $ F^i $ acts trivially on $ K $, $ F^{j(ci+b')} $ is indeed an extension of $ F^{jb'} $ (seen as restricted to $ K $) to $ E $. Furthermore, $ F^{a'(ci+b')} $ acts trivially on $ E = \mathbf{F}_{p^{id}}(\!(T)\!) $, because $ a'(ci+b') = i(ca'+1)\in id\Z $ by definition of $ c $. Hence, $ f_{C_{a'}} $ is indeed a well-defined homomorphism.
\end{proof}

We need one more bit of notation before defining the last portion of the section $ \Aut (G\to \Spec K)\to \Aut_G(K) $  on $ \Gal (K/\mathbf{F}_p(\!(T)\!)) $.
\begin{definition}\label{Def:the matrix w}
With the notations of Definition~\ref{Def:notation for splitting explicitly}, let $ 0_b $ be the zero $ b\times b $ matrix, and let $ w $ be the following $ bb'\times bb' $ matrix:
$$ w = \begin{pmatrix}
0_b  & 0_b  & 0_b  &\dots & 0_b  & u\Id_b\\
\Id_b & 0_b  & 0_b  & \dots & 0_b  & 0_b \\
0_b  & \Id_b & 0_b &\dots & 0_b & 0_b \\
0_b &  0_b &\Id_b&\dots & 0_b & 0_b \\
\vdots &\vdots & \vdots & \ddots & \vdots &\vdots \\
0_b  & 0_b  & 0_b  &\dots &\Id_b& 0_b \\
\end{pmatrix} $$
\end{definition}
\begin{proposition}\label{Prop:constructing fCb'}
Keep the notations of Definition~\ref{Def:notation for splitting explicitly}. Assume that $ bb' $ divides $ n $. Let $ C_{b'} $ be the group generated by $ F^{a'} $ in $ \Gal (K/\mathbf{F}_{p}(\!(T)\!)) $. With the notations of Definition~\ref{Def:the matrix w}, let $ W = \diag (w,\dots ,w) $ where we have $ \frac{n}{bb'} $ terms $ ( $so that $ W $ is a $ n\times n $ matrix which is block diagonal with coefficients in the set $ \lbrace 0,1,u\rbrace )$. Recalling the notation introduced in Remark~\ref{Rem:abstract automorphisms are algebraic}, the map
\begin{align*}
f_{C_{b'}}\colon C_{b'}&\to \Aut (G\to \Spec K)\\
F^{a'j} &\mapsto \intaut (W^j)\tilde{\phi}(F^{a'j},1)
\end{align*}
is a homomorphism. Furthermore, its composition with the map $ \Aut (\! G\! \to \Spec K)\to \Aut_G(K) $ is the identity on $ C_{b'} $.
\end{proposition}
\begin{proof}
The only assertion that requires a justification is that the map is well-defined, i.e.\ we have to check that $ (\intaut (W)\tilde{\phi}(F^{a'},1))^{b'} $ is the identity on $ \SL_n(D) $. By definition, $ W^{b'} $ is the scalar matrix $ u.\Id_n $. Hence $ (\intaut (W)\tilde{\phi}(F^{a'},1))^{b'} = \intaut(W^{b'})\tilde{\phi}(F^{a'b'},1) = \intaut(u.\Id_n)\tilde{\phi}(F^{i},1) $, which indeed acts as the identity on $ \SL_n(D) $ because for all $ x\in E $, $ uxu^{-1} = F^{-i}(x) $ (see Remark~\ref{Rem:on the notation for explicit splitting}).
\end{proof}
\begin{remark}
Note that if one is just interested in a section defined on $ C_{b'} $ alone, one can take $ b=1 $ in Proposition~\ref{Prop:constructing fCb'}. Hence, a section of $ \Aut (G\to \Spec K)\to \Aut_G(K) $ only defined on $ C_{b'} $ exists if and only if $ b' $ divides $ n $ (i.e.\ the stronger assumption that $ bb' $ divides $ n $ is there to ensure that we can glue $ f_{C_{b'}} $ with $ f_{C_b} $).
\end{remark}

As before, we have a converse to Proposition~\ref{Prop:constructing fCb'}.
\begin{proposition}
Keep the notations of Proposition~\ref{Prop:constructing fCb'}. If $ b' $ does not divide $ n $, there does not exist a homomorphism $ C_{b'}\to \Aut (G\to \Spec K) $ whose composition with $ \Aut (G\to \Spec K)\to \Aut_G(K) $ is the identity on $ C_{b'} $.
\end{proposition}
\begin{proof}
By Theorem~\ref{Thm:local non-splitting for SLn(D)}, it suffices to prove that there exists $ K'\leq K $ such that $ K/K' $ is finite Galois, $ \Gal (K/K')\leq C_{b'} $  and $ \gcd (nd,[K:K']) $ does not divide $ n $. But $ K' = \mathbf{F}_{p^{a'}}(\!(T)\!) $ is such a subfield.
\end{proof}

We can finally glue all the previous constructions to obtain a global splitting of the initial short exact sequence.
\begin{theorem}\label{Thm:splitting for SLn(D) in char. p}
Keep the notations of Definition~\ref{Def:notation for splitting explicitly}. Assume that for all subfields $ K'\leq K $ such that $ K/K' $ is finite Galois, $ \gcd (nd,[K:K']) $ divides $ n $. Then the short exact sequence $ 1\to \Aut G\to \Aut (G\to \Spec K)\to \Aut_G(K)\to 1 $ splits.
\end{theorem}
\begin{proof}
In view of Proposition~\ref{Prop:existence of Galois subfield of some degree}, the hypotheses imply that $ \gcd(d,p)=1 $ and $ \gcd (nd,i(p^i-1)) $ divides $ n $. Hence $ bb' $ divides $ n $ and we can apply Propositions~\ref{Prop:constructing fJK},~\ref{Prop:constructing fCl(d)},~\ref{Prop:constructing fCk(d)},~\ref{Prop:constructing fCa'} and~\ref{Prop:constructing fCb'}. For the rest of the proof, we strictly adhere to the notations that are introduced in the statements of those propositions.

Recall that $ \Aut_G(K) = \Aut (K) $ (Corollary~\ref{Cor:etxending auto to D}). Also recall that we decomposed $ \Aut (K) $ as $ (J(K)\rtimes (C_a\times C_b))\rtimes (C_{a'}\times C_{b'}) $, where
\begin{enumerate}[(i)]
\item $ C_a\leq \mathbf{F}_{p^i}^{\times} $ is generated by $ \zeta^b $.
\item $ C_b\leq \mathbf{F}_{p^i}^{\times} $ is generated by $ \zeta^a $.
\item $ C_{a'}\leq Gal (K/\mathbf{F}_p(\!(T)\!)) $ is generated by $ F^{b'} $ restricted to $ K $.
\item $ C_{b'}\leq Gal (K/\mathbf{F}_p(\!(T)\!)) $ is generated by $ F^{a'} $ restricted to $ K $.
\end{enumerate}

We define a map
\begin{align*}
f\colon &(J(K)\rtimes (C_a\times C_b))\rtimes (C_{a'}\times C_{b'})\to \Aut (G\to \Spec K)\\
&(g_1 ,g_2,g_3,g_4,g_5)\mapsto f_{J(K)}(g_1)f_{C_a}(g_2)f_{C_b}(g_3)f_{C_{a'}}(g_4)f_{C_{b'}}(g_5)
\end{align*}
We claim that $ f $ is a homomorphism. To prove this claim, it suffices to compute various commutators in $ \Aut (G\to \Spec K) $. To carry the computation, we pick $ j,j' \in \mathbf{N} $.
\begin{enumerate}
\item The images of $ f_{C_a} $ and $ f_{C_b} $ commute. Indeed, $ \intaut (Z^j)\tilde{\phi}(\ev (\zeta^{bj}T) ,z^{b'j})$ readily commutes with $ \intaut (Y^{j'})\tilde{\phi}(\ev (\zeta^{aj'}T) ,(\frac{F^i(y)}{y})^{j'}) $ (note that $ Y $ and $ Z $ are both $ b\times b $ block diagonal matrices, and that the blocks defining $ Z $ are scalars).

\item Let $ \alpha \in J(K) $. We compute $ f_{C_a}(\ev (\zeta^{bj}T))f_{J(K)}(\alpha)f_{C_a}(\ev (\zeta^{-bj}T)) $:
\begin{align*}
&\intaut (Z^j)\tilde{\phi}(\ev (\zeta^{bj}T),z^{b'j})\circ \intaut (X_{\alpha})\tilde{\phi}(\alpha ,x_{\alpha}^{b'})\circ \intaut (Z^{-j})\tilde{\phi}(\ev (\zeta^{-bj}T),z^{-b'j}) \\
&= \intaut (Z^j \ev (\zeta^{bj}T)(X_{\alpha}) \alpha (Z^{-j})) \tilde{\phi}(\ev (\zeta^{bj}T)\circ \alpha \circ \ev (\zeta^{-bj}T) ,z^{b'j}\ev (\zeta^{bj}T)(x_{\alpha}^{b'}) \alpha (z^{-b'j})) \\
&= \intaut (\ev (\zeta^{bj}T)(X_{\alpha})) \tilde{\phi}(\ev (\zeta^{bj}T)\circ \alpha \circ \ev (\zeta^{-bj}T) ,\ev (\zeta^{bj}T)(x_{\alpha})^{b'})
\end{align*}
where the last equality follows from the fact that $ Z $ and $ X_{\alpha} $ commutes because they are block diagonal matrices, together with the equality $ \alpha (Z^{-j}) = Z^{-j} $ which holds because $ Z $ has coefficients in $ \mathbf{F}_{p^i} $. 

But $ \ev (\zeta^{bj}T) (x_{\alpha}) = x_{\ev (\zeta^{bj}T)\circ \alpha \circ \ev (\zeta^{-bj}T)} $, because $ \ev (\zeta^{bj}T) (x_{\alpha}) $ belongs to $ 1+T\mathbf{F}_{p^i}[\![T]\!] $, and 
\begin{align*}
\ev (\zeta^{bj}T)(x_{\alpha})^{b'd} &= \ev (\zeta^{bj}T)(\frac{\alpha (T^r)}{T^r}) \\
&= \zeta^{-bj}.\dfrac{\ev (\zeta^{bj}T)\alpha (T^r)}{T^r} \\
&= \dfrac{(\ev (\zeta^{bj}T)\circ \alpha \circ \ev (\zeta^{-bj}T)) (T^r)}{T^r}
\end{align*}
Hence $ f_{C_a}(\ev (\zeta^{bj}T))f_{J(K)}(\alpha)f_{C_a}(\ev (\zeta^{-bj}T)) = f_{J(K)}(\ev (\zeta^{bj}T)\circ \alpha \circ \ev (\zeta^{-bj}T)) $.

\item The equality $ f_{C_b}(\ev (\zeta^{aj}T))f_{J(K)}(\alpha)f_{C_b}(\ev (\zeta^{-aj}T)) =f_{J(K)}(\ev (\zeta^{aj}T)\circ \alpha \circ \ev (\zeta^{-aj}T))  $ is proved by doing a similar computation than in the previous item.

\item The images of $ f_{C_{a'}} $ and $ f_{C_{b'}} $ commute. Indeed, $ \tilde{\phi}(F^{j(ci+b')},1) $ readily commutes with $ \intaut (W^{j'}) \tilde{\phi}(F^{a'j'},1) $ (recall that $ W $ has coefficients in $ \lbrace 0,1,u\rbrace $).

\item We check that $ f_{C_{a'}}(F^{b'j})f_{C_b}(\ev (\zeta^{aj'}T))f_{C_{a'}}(F^{-b'j}) = f_{C_b}(F^{b'j} \circ \ev (\zeta^{aj'}T)\circ F^{-b'j}) $. We have
\begin{align*}
&\tilde{\phi}(F^{j(ci+b')},1)\circ \intaut (Y^{j'})\tilde{\phi}(\ev (\zeta^{aj'}T) ,(\frac{F^i(y)}{y})^{j'})\circ \tilde{\phi}(F^{-j(ci+b')},1)\\
&=  \intaut (F^{j(ci+b')}(Y^{j'}))\tilde{\phi}(F^{j(ci+b')}\circ \ev (\zeta^{aj'}T)\circ F^{-j(ci+b')},F^{j(ci+b')}((\frac{F^i(y)}{y})^{j'}))
\end{align*}

Noting that $ F^{j(ci+b')}\circ \ev (\zeta^{aj'}T)\circ F^{-j(ci+b')} = \ev (F^{j(ci+b')}(\zeta^{aj'})T) $, the desired equality follows from the fact that the Frobenius automorphism on $ \mathbf{F}_{p^{id}} $ is just elevating to the power $ p $.

\item One readily check that $ f_{C_{a'}}(F^{b'j})f_{C_a}(\ev (\zeta^{bj'}T))f_{C_{a'}}(F^{-b'j}) = f_{C_a}(F^{b'j} \circ \ev (\zeta^{bj'}T)\circ F^{-b'j}) $.

\item We have $ f_{C_{a'}}(F^{b'j})f_{J(K)}(\alpha)f_{C_{a'}}(F^{-b'j}) = f_{J(K)}(F^{b'j} \circ \alpha\circ F^{-b'j}) $. Indeed, $ F^{j(ci+b')}(x_{\alpha}) $  belongs to $ 1+T\mathbf{F}_{p^i}[\![T]\!] $, and  $ F^{j(ci+b')}(x_{\alpha})^{b'd} \! =\!  F^{j(ci+b')}(\frac{\alpha (T^r)}{T^r}) = \dfrac{F^{j(ci+b')}\alpha F^{-j(ci+b')} (T^r)}{T^r} $ because $ F $ acts trivially on $ T $. Hence
\begin{align*}
&\tilde{\phi}(F^{j(ci+b')},1)\intaut (X_{\alpha})\tilde{\phi}(\alpha ,x_{\alpha}^{b'})\tilde{\phi}(F^{-j(ci+b')},1) \\
&=\intaut (F^{j(ci+b')}(X_{\alpha})) \tilde{\phi}(F^{j(ci+b')}\alpha F^{-j(ci+b')} ,F^{j(ci+b')}(x_{\alpha}^{b'}))\\
&= \intaut (X_{F^{j(ci+b')}\alpha F^{-j(ci+b')}})\tilde{\phi}(F^{j(ci+b')}\alpha F^{-j(ci+b')} ,x^{b'}_{F^{j(ci+b')}\alpha F^{-j(ci+b')}})
\end{align*}
as wanted.

\item We check that $ f_{C_{b'}}(F^{a'j})f_{C_b}(\ev (\zeta^{aj'}T))f_{C_{b'}}(F^{-a'j}) = f_{C_b}(F^{a'j} \circ \ev (\zeta^{aj'}T)\circ F^{-a'j}) $. It is obviously enough to check this when $ j=j'=1 $. Then
\begin{align*}
&\intaut (W)\tilde{\phi}(F^{a'},1)\circ \intaut (Y)\tilde{\phi}(\ev (\zeta^{a}T) ,\frac{F^i(y)}{y})\circ \intaut (W^{-1})\tilde{\phi}(F^{-a'},1) = \\
&\intaut (WF^{a'}(Y.\tilde{\phi}(\ev (\zeta^{a}T) ,\frac{F^i(y)}{y})(W^{-1})))\tilde{\phi}(F^{a'}\circ \ev (\zeta^{a}T)\circ F^{-a'},F^{a'}(\frac{F^i(y)}{y}))
\end{align*}

Again, since $ F^{a'}\circ \ev (\zeta^{a}T)\circ F^{-a'} = \ev (F^{a'}(\zeta^{a})T) $, it suffices to show that the matrix appearing in the argument of $ \intaut () $ is equal to $ F^{a'}(Y) $. Since $ W $ is made up of $ bb'\times bb' $ block matrices, whilst $ Y $ is made up of $ b\times b $ block matrices, we can work with one block at a time. Otherwise stated, we may assume that $ bb' = n $. Now $ \tilde{\phi}(\ev (\zeta^{a}T) ,\frac{F^i(y)}{y})(W^{-1}) $ is the matrix
\begin{equation*}
\begin{pmatrix}
0_b & \Id_b & 0_b  & 0_b  & \dots   & 0_b \\
0_b & 0_b  & \Id_b & 0_b &\dots  & 0_b \\
0_b & 0_b &  0_b &\Id_b&\dots  & 0_b \\
\vdots & \vdots &\vdots & \vdots & \ddots  &\vdots \\
0_b  & 0_b  & 0_b  & 0_b  &\dots &\Id_b \\
(u\frac{F^i(y)}{y})^{-1}.\Id_b & 0_b &  0_b & 0_b &\dots & 0_b
\end{pmatrix}
\end{equation*} so that the matrix appearing in the argument of $ \intaut () $ is the $ n\times n $ block diagonal matrix $ \diag (uF^{a'}(g(u\frac{F^i(y)}{y})^{-1}),F^{a'}(g),\dots ,F^{a'}(g)) $. But $ uF^{a'}(g(u\frac{F^i(y)}{y})^{-1}) = F^{a'}(ug\frac{y}{F^i(y)}u^{-1}) $. Recalling the embedding $ \varphi \colon \mathbf{F}_{p^{idb}}\to M_b(\mathbf{F}_{p^{id}}) $ of Proposition~\ref{Prop:constructing fCk(d)}, we have
\begin{align*}
ug\frac{y}{F^i(y)}u^{-1} &= u\varphi(y^{-1}) \varphi (y) F^i(\varphi (y)^{-1})u^{-1}\\
&= F^{-i}(F^i(g)) = g
\end{align*}
We conclude that the argument of $ \intaut () $ is indeed $ \diag (F^{a'}(g),F^{a'}(g),\dots ,F^{a'}(g)) = F^{a'}(Y) $, as wanted.

\item Let us check that $ f_{C_{b'}}(F^{a'j})f_{C_a}(\ev (\zeta^{bj'}T))f_{C_{b'}}(F^{-a'j}) = f_{C_a}(F^{a'j} \circ \ev (\zeta^{bj'}T)\circ F^{-a'j}) $. It is obviously enough to check this when $ j=j'=1 $. Then
\begin{align*}
&\intaut (W)\tilde{\phi}(F^{a'},1)\intaut (Z)\circ \tilde{\phi}(\ev (\zeta^{b}T) ,z^{b'})\intaut (W^{-1})\tilde{\phi}(F^{-a'},1) = \\
&\intaut (WF^{a'}(Z.\tilde{\phi}(\ev (\zeta^{b}T) ,z^{b'})(W^{-1})))\tilde{\phi}(F^{a'}\circ \ev (\zeta^{b}T)\circ F^{-a'},F^{a'}(z^{b'}))
\end{align*}

Again, since $ F^{a'}\circ \ev (\zeta^{b}T)\circ F^{-a'} = \ev (F^{a'}(\zeta^{b})T) $, it suffices to show that the matrix appearing in the argument of $ \intaut () $ is equal to $ F^{a'}(Z) $. As in the previous item, we can assume that $ bb'=n $, and doing the same kind of computation as in the previous item, we find that the matrix in the argument of $ \intaut () $ is $ F^{a'}(\diag (uz^{-1}u^{-1}\Id_b,\Id_b,z\Id_b,\dots ,z^{b'-2}\Id_b)) $. Since $ z\in \mathbf{F}_{p^{id}} $, $ uz^{-1}u^{-1} = z^{-1} $, and we conclude that multiplication by the scalar matrix $ z\Id_n $ (which belongs to the center of $ \GL_n(D) $ because $ z\in \mathbf{F}_{p^{i}} $), the argument appearing in $ \intaut () $ is equal to $ F^{a'}(Z) $, as wanted.

\item Checking that $ f_{C_{b'}}(F^{a'j})f_{J(K)}(\alpha)f_{C_{b'}}(F^{-a'j}) = f_{J(K)}(F^{a'j} \circ \alpha\circ F^{-a'j}) $ is a similar computation than in the previous item.
\end{enumerate}
We conclude that $ f $ is indeed a homomorphism. The fact that $ f $ is a splitting of the short exact sequence in the statement of the proposition follows from the fact that the restriction of $ f $ to each component is locally a section of $ \Aut (G\to \Spec K)\to \Aut_G(K) $.
\end{proof}

The first step in the proof of Theorem~\ref{Thm:splitting for SLn(D) in char. p} is to translate the existence of Galois subfields of some degree into some divisibility relations between $ p,d,i $ and $ p^i-1 $. We now prove the ad hoc proposition. We warn the reader that the notations of Definition~\ref{Def:notation for splitting explicitly} (and of the subsequent propositions) are not in use any more.
\begin{proposition}\label{Prop:existence of Galois subfield of some degree}
Let $ K = \mathbf{F}_{p^i}(\!(T)\!) $, let $ q $ be a prime number and let $ a\in \N $. There exists a subfield $ K' $ such that $ K/K' $ is finite Galois and $ q^a $ divides $[K:K'] $ if and only if $ q=p $ or $ q^a $ divides $ i(p^i-1) $.
\end{proposition}
\begin{proof}
First assume that such a $ K' $ exists. Since $ K/K' $ is Galois and $ q^a$ divides $[K:K'] $, there exists $ \tilde{K} $ such that $ K/\tilde{K} $ is Galois and $ [K:\tilde{K}] = q^a $. Up to replacing $ K' $ by $ \tilde{K} $, we can thus assume that $ [K:K'] = q^a $. Let also $ K'_{ur} $ be the maximal unramified extension of $ K' $ inside $ K $.

Note that $ K' $ and $ K'_{ur} $ are local fields, so that in particular $ K'\cong \mathbf{F}_{p^j}(\!(T)\!) $ and $ K'_{ur}\cong \mathbf{F}_{p^i}(\!(T)\!) $. Since $ [K'_{ur}:K'] $ divides $ q^a $, there exists $ a_1 $ such that $ q^{a_1} = \frac{i}{j} $. Letting $ a_2 = a-a_1 $, we have that $ K/K'_{ur} $ is a totally ramified extension of degree $ q^{a_2} $.

If $ p=q $, the proposition is proved, hence there just remains to investigate the case $ p\neq q $.  In this case, $ K $ is a tamely totally ramified extension of $ K'_{ur}$. Thus, $ K $ is isomorphic to $ K'_{ur}[X]/(X^{q^{a_2}}-\pi) $ for some uniformiser $ \pi\in \mathbf{F}_{p^i}(\!(T)\!) $. But $ K $ is a Galois extension, and hence this implies that $ \mathbf{F}_{p^i}(\!(T)\!) $ has a primitive $ q^{a_2} $-th root of unity, so that $ q^{a_2} $ divides $ p^i-1 $, as wanted.

To prove the converse, we use a classical fact from local class field theory: there exists an extension $ K_{\pi} $ of $ K $ which is Galois and totally ramified, and such that $ \Gal (K_{\pi}/K)$ is isomorphic to the group of invertible elements $ \mathbf{F}_{p^i}[\![T]\!]^{\times} $ of $ \mathbf{F}_{p^i}[\![T]\!] $ (see for example \cite{Iwa86}*{Section~5.3}). Note that the degree of $ \mathbf{F}_{p^i}^{\times}+T^{a+1}\mathbf{F}_{p^i}[\![T]\!] $ in $ \mathbf{F}_{p^i}[\![T]\!]^{\times} $ is equal to $ p^a $. Let $ L_1 $ be the Galois extension of $ K $ corresponding to $ \mathbf{F}_{p^i}^{\times}+T^{a+1}\mathbf{F}_{p^i}[\![T]\!] $. Let also $ L_2 $ be the splitting field of $ X^{p^i-1}-T $ over $ \mathbf{F}_{p}(\!(T)\!) $. For $ j=1$ or $ 2 $, $ L_j $ is totally ramified of finite degree over $ K $, so that there exists an isomorphism $ \phi_j \colon K\to L_j $. Hence $ K_1 = \phi_1^{-1}(K) $ (respectively $ K_2 = \phi_2^{-1}(\mathbf{F}_{p}(\!(T)\!)) $) is such that $ K/K_1 $ (respectively $ K/K_2 $) is Galois, and $ [K:K_1] = p^a $ (respectively $ [K:K_2] = i(p^i-1) $), which concludes the proof.
\end{proof}

\appendix
\section{Base change of the algebraic group \texorpdfstring{$\textbf{SL}_n(D)$}{SLn(D)}}
We begin by recalling some classical facts about finite dimensional central simple algebras over local fields.
\begin{theorem}\label{Thm:classificatio nof CSA over local fields}
Let $ K $ be a local field. Every central simple algebra over $ K $ is isomorphic to an algebra of the form $ A(d,r) $ as in Definition~\ref{Def:the CSA A_(d,r)}. Furthermore, the map $ \inv \colon Br(K)\to \mathbf{Q}/\mathbf{Z}\colon [A(d,r)]\mapsto [\frac{r}{d}] $ is an isomorphism of groups.
\end{theorem}
\begin{proof}
See for example \cite{Mor97}*{Theorem~8} for the first assertion, while the second is precisely the content of \cite{Pie82}*{Chapter~17, \S 10, Theorem}.%Remark: Morandi only says any CSA is isomorphic to a cyclic $(E/K,\sigma ,a)$. But I proved that if a and a' have the same valuation, then $(E/K,\sigma ,a)$ is iso. to $(E/K,\sigma ,a')$ as an algebra, so that's good enough.
\end{proof}

\begin{corollary}\label{Cor:CSA over local fields and Wedderburn}
Let $ K $ be a local field and let $ d,r \in \mathbf{N} $ with $ d\geq 1 $. Let $ a = \gcd (d,r) $. Then $ A(d,r) $ is a division algebra if and only if $ a=1 $, and $ A(d,r)\cong M_a(A(\frac{d}{a},\frac{r}{a})) $.
\end{corollary}
\begin{proof}
The central simple algebra $ A(d,r) $ is a division algebra if and only if all central simple algebras over $ K $ in the same Brauer class have a higher degree. In view of Theorem~\ref{Thm:classificatio nof CSA over local fields}, it readily implies that $ A(d,r) $ is a division algebra if and only if $ a=1 $. Furthermore, by Wedderburn's theorem, $ A(d,r) $ is isomorphic to $ M_n(D) $ for some division algebra $ D $ and some $ 1\leq n \in \mathbf{N} $, and by definition of the Brauer group, $ [D] = [A(d,r)] $. Hence, using the first part of the Theorem, $ D\cong A(\frac{d}{a},\frac{r}{a}) $. Now, comparing degrees readily imply that $ n=a $, and the result is proved.
\end{proof}

We now study the base change of the algebraic group $ SL_n(A) $.
\begin{lemma}\label{Lem:base change of SL_1}
Let $ A $ be a central simple algebra over a field $ k $, and let $ \textbf{SL}_1(A) $ be the corresponding algebraic $ k $-group (see Definition~\ref{Def:algebraic SL_n(D)}). For $ k' $ a field extension of $ k $, $ \textbf{SL}_1(A)_{k'} = \textbf{SL}_1(A\otimes_k k') $.
\end{lemma}
\begin{proof}
Let $ \overline{k'} $ be an algebraic closure of $ k' $. Since $ \overline{k'} $ splits $ A $, the reduced norm is the map $ f\colon A\to A\otimes_k \overline{k'}\cong M_n(\overline{k'})\xrightarrow[]{\det} \overline{k'} $. Let $ \varphi $ denotes the isomorphism $ A\otimes_k \overline{k'}\cong M_n(\overline{k'}) $. If we take a $ k $-basis of $ A $ to get coordinates on $ A\otimes_k \overline{k'} $, the map $ \det \circ \varphi $ is actually a polynomial map on $ A\otimes_k \overline{k'} $ with coefficients in $ k $, by \cite{Bourb73}*{Chapitre~VIII, \S 12, Proposition~11}. %In fact, they prove that f is a k-polynomial. But basically the same proof shows that $ \det \circ \varphi $ is a k-polynomial.
Hence, $ f_{\overline{k'}} = \det \circ \varphi $. This implies that $ f_{k'}\colon A\otimes_k k' \to k' $ is just the composition $ A\otimes_k k' \to A\otimes_k \overline{k'}\cong M_n(\overline{k'})\xrightarrow[]{\det} \overline{k'} $, i.e.\ $ f_{k'} $ is the reduced norm map of the algebra $ A\otimes_k k' $, as wanted.
\end{proof}

Before giving the formula for the base change of $ \SL_n(A) $, we recall the effect of extending scalars for central simple algebras over local fields.
\begin{lemma}\label{Lem:Base change of CSA}
Let $ K $ be a local field and let $ A(d,r) $ be the central simple algebra over $ K $ defined in Definition~\ref{Def:the CSA A_(d,r)}. Let $ L $ be a finite extension of $ K $. Then $ A(d,r)\otimes_K L\cong A(d,r[L:K]) $.
\end{lemma}
\begin{proof}
By Wedderburn's theorem, a central simple algebra over a field is uniquely determined by its degree and its Brauer class. By \cite{Pie82}*{Chapter~17, Section~17.10, Proposition}, we have $ \inv([A(d,r)\otimes_K L]) = [L:K].\inv([A(d,r)] $. Hence $ A(d,r[L:K]) $ and $ A(d,r)\otimes_K L $ are in the same Brauer class. Since they have the same degree as well, this concludes the proof.
\end{proof}

\begin{proposition}\label{Prop:base change of SL_n(A)}
Let $ A(d',r') $ be a division algebra over a local field $ K' $ as in Definition~\ref{Def:the CSA A_(d,r)}. Let $ K/K' $ be a finite field extension and let $ a = \gcd (d',[K,K']) $. Then the base change of $ \textbf{SL}_{n'}(A(d',r'))$ to $ K $ is isomorphic to $ \textbf{SL}_{an'}(A(\frac{d'}{a},\frac{[K:K']}{a}r')) $.
\end{proposition}
\begin{proof}
The base change of $ \textbf{SL}_{n'}(A(d',r')) = \textbf{SL}_{1}(M_{n'}(A(d',r'))) $ to $ K $ is isomorphic to the algebraic group $ \textbf{SL}_{1}(M_{n'}(A(d',r'))\otimes_{K'} K) \cong \textbf{SL}_{n'}(A(d',r')\otimes_{K'} K)  $ by Lemma~\ref{Lem:base change of SL_1}. But by Corollary~\ref{Cor:CSA over local fields and Wedderburn} and Lemma~\ref{Lem:Base change of CSA}, $ A(d',r')\otimes_{K'} K\cong M_a(A(\frac{d'}{a},\frac{[K:K']}{a}r')) $. To conclude, note that for any central simple algebra $ A $, $ \textbf{SL}_{n'}(M_a(A))\cong \textbf{SL}_{an'}(A) $.
\end{proof}

%**************************************Bibliography***************************************

\clearpage

%\addcontentsline{toc}{chapter}{Bibliography}
%\begin{thebibliography}{999}
\bibliographystyle{alpha}
\bibliography{biblioup}
%\addcontentsline{toc}{chapter}{Bibliography}

%\end{thebibliography}

\end{document}